\documentclass[12pt]{amsart}
\usepackage{amsmath}
\usepackage{amssymb}
\usepackage{amsfonts}
\usepackage{mathrsfs}
\usepackage{latexsym}
\usepackage{mathtools}
\usepackage{amscd}
\usepackage[mathscr]{euscript}
\usepackage{enumitem}
\usepackage{xy} \xyoption{all}
\usepackage{hyperref}

\numberwithin{equation}{section}

\renewcommand{\uparrow}{}

\newcommand{\projection}{\mathcal{P}}
\newcommand{\semigroup}{\mathscr{S}}
\newcommand{\category}[1][(P,Q,\psi)]{\mathcal{C}_{#1}}
\newcommand{\N}{{\mathbb{N}}}
\newcommand{\No}{\mathbb{N}_0}
\newcommand{\Z}{{\mathbb{Z}}}

\newcommand{\inv}{^{-1}}
\newcommand{\uloopr}[1]{\ar@'{@+{[0,0]+(-4,5)}@+{[0,0]+(0,10)}@+{[0,0] +(4,5)}}^{#1}}
\newcommand{\uloopd}[1]{\ar@'{@+{[0,0]+(5,4)}@+{[0,0]+(10,0)}@+{[0,0]+ (5,-4)}}^{#1}}
\newcommand{\dloopr}[1]{\ar@'{@+{[0,0]+(-4,-5)}@+{[0,0]+(0,-10)}@+{[0, 0]+(4,-5)}}_{#1}}
\newcommand{\dloopd}[1]{\ar@'{@+{[0,0]+(-5,4)}@+{[0,0]+(-10,0)}@+{[0,0 ]+(-5,-4)}}_{#1}}

\newcommand{\luloop}[1]{\ar@'{@+{[0,0]+(-8,2)}@+{[0,0]+(-10,10)}@+{[0, 0]+(2,2)}}^{#1}}

\DeclareMathOperator{\spa}{span}

\DeclareMathOperator{\aut}{Aut}
\DeclareMathOperator{\id}{Id}

\newtheorem{lem}{Lemma}[section]
\newtheorem{corol}[lem]{Corollary}
\newtheorem{theor}[lem]{Theorem}
\newtheorem{prop}[lem]{Proposition}
\theoremstyle{definition}
\newtheorem{defi}[lem]{Definition}

\newtheorem{exem}[lem]{Example}
\newtheorem{exems}[lem]{Examples}
\newtheorem{nota}[lem]{Notation}
\newtheorem{rema}[lem]{Remark}

\begin{document}
\title{Algebraic Cuntz-Pimsner rings}
\author{Toke Meier Carlsen}
\address{Department of Mathematical Sciences\\
NTNU\\
NO-7491 Trondheim\\
Norway } \email{Toke.Meier.Carlsen@math.ntnu.no}
\author{Eduard Ortega}
\address{Department of Mathematics and Computer Science University of
  Southern Denmark\\
  Campusvej 55\\
  DK-5230 Odense M\\
  Denmark}
\email{ortega@imada.sdu.dk}
\thanks{The first named author was partly supported by The Danish
  Natural Science Research Council and the Research Council of Norway, and the second named author was
  partially supported by MEC-DGESIC (Spain) through Project
  MTM2008-06201-C02-01/MTM, and by the Comissionat per Universitats i
  Recerca de la Generalitat de Catalunya.}
 \keywords{}
\date{\today}

\begin{abstract}
From a system consisting of a ring $R$, a pair of
$R$-bimodules $Q$ and $P$ and an $R$-bimodule homomorphism
$\psi:P\otimes Q\longrightarrow R$, we construct a $\Z$-graded ring
$\mathcal{T}_{(P,Q,\psi)}$ called the Toeplitz ring and (for certain
systems) a $\Z$-graded quotient $\mathcal{O}_{(P,Q,\psi)}$ of
$\mathcal{T}_{(P,Q,\psi)}$ called the Cuntz-Pimsner ring.
These rings are the algebraic analogs of the Toeplitz $C^*$-algebra and
the Cuntz-Pimsner $C^*$-algebra associated to a $C^*$-correspondence
(also called a Hilbert bimodule).

This new construction generalizes for example the
algebraic crossed product by a single automorphism, fractional skew monoid ring by a single corner automorphism and Leavitt path
algebras.
We also describe the structure of the graded ideals of
our graded rings in terms of pairs of ideals of the coefficient ring and show that our Cuntz-Pimsner rings satisfies the \emph{Graded Uniqueness Theorem}.
\end{abstract}

\maketitle

\tableofcontents

\section*{Introduction}

In \cite{PI} Pimsner introduced a way to construct a $C^*$-algebra
$\mathcal{O}_X$ from a $C^*$-correspondence $X$ over a
$C^*$-algebra $A$. These so-called Cuntz-Pimsner algebras have
been found to be a class of $C^*$-algebras that is extraordinarily rich
and with numerous examples included in the literature: crossed
products by automorphisms, Cuntz algebras, Cuntz-Krieger algebras, $C^*$-algebras
associated to graphs without sinks and Exel-Laca algebras.
Later on Katsura \cite{KS1} improved the construction of Pimsner
in the case that the left action on the correspondence is not
injective, this for example allows us to include the class of
$C^*$-algebras associated with any graph into the Cuntz-Pimsner
algebras class.
Consequently the study of the Cuntz-Pimsner algebras has received
a lot of attention in recent years, and because information about
$\mathcal{O}_X$ is densely codified in $X$ and $A$,
determining how to extract it has been the focus of considerable interest.

It has recently been discovered that many of $C^*$-algebras which can be constructed as Cuntz-Pimsner algebras have algebraic analogs.
For example the crossed product of a ring by an automorphism is the obvious analog of the crossed product of a $C^*$-algebra of an automorphism. In \cite{AGGP} Ara, Gonz\'alez-Barroso, Goodearl and Pardo inspired by a construction in $C^*$-algebra constructed \emph{fractional skew monoid rings} from actions of monoid on rings by endomorphisms.
In \cite{LE} Leavitt described a class of $F$-algebras $L(m,n)$ (where $F$ is an arbitrary algebra) which are
universal with respect to an isomorphism property between finite
rank modules, i.e. $L(m,n)^n\cong L(m,n)^m$. Later Cuntz \cite{CU} (independently)
constructed and investigated the $C^*$-algebra $\mathcal{O}_n$,
called the Cuntz algebras. When $F$ is the complex numbers then
$\mathcal{O}_n$ can be viewed as a completion, in an appropriate
norm, of $L(1,n)$. Soon after the appearance of \cite{CU}, Cuntz
and Krieger \cite{CK} described the significantly more general
notion of the $C^*$-algebra of a (finite) matrix $A$, denoted
$\mathcal{O}_A$. In \cite{KPRR} Cuntz-Krieger algebras were generalized to $C^*$-algebras of locally finite directed graphs, and this construction has later been generalized several time and now apply to arbitrary directed graphs. Inspired by the fractional skew monoid rings and by the graph $C^*$-algebras, Abrams and Aranda Pino \cite{AA1} constructed the Leavitt path algebra of a row-finite directed graph.
This construction was later generalized to apply to arbitrary directed graphs. The Leavitt path algebras provide a generalization of Leavitt algebras of type $(1, n)$
just in the same way as graph $C^*$-algebras $C^*(E)$ provide a
generalization of Cuntz algebras, and they have recently attracted a great deal of interest (see for example \cite{AA2,AALP,AMP,TF}).

It would be interesting and useful to put these rings and algebras in a larger
category of rings whose properties can be studied and analyzed
from more simple objects, just as it has been done in the $C^*$-algebraic setting with Cuntz-Pimsner algebras. This is the purpose of this paper.

From a ring $R$
and a triple $(P,Q,\psi)$, called an $R$-system, consisting of
two $R$-bimodules $P$ and $Q$ and a
$R$-bimodule homomorphism $\psi:P\otimes Q\longrightarrow R$ we construct
a universal $\Z$-graded ring $\mathcal{T}_{(P,Q,\psi)}$, called the \textit{Toeplitz
ring} associated with $(P,Q,\psi)$, which contains copies of $R$, $P$ and $Q$
and which implements the $R$-bimodule structure of $P$ and $Q$ and the
$R$-bimodule homomorphism $\psi$.
We then, for $R$-systems satisfying a certain condition which we call \textbf{(FS)},
carefully study quotients of $\mathcal{T}_{(P,Q,\psi)}$ which preserve the
$\Z$-grading of $\mathcal{T}_{(P,Q,\psi)}$. We show that under a mild assumption
about the $R$-system $(P,Q,\psi)$, there exists a smallest quotient of $\mathcal{T}_{(P,Q,\psi)}$ which preserve the
$\Z$-grading of $\mathcal{T}_{(P,Q,\psi)}$ and which leaves the embedded copy of $R$
intact. We define the \textit{Cuntz-Pimsner ring}
$\mathcal{O}_{(P,Q,\psi)}$ of $(P,Q,\psi)$ to be this quotient.

We show that the construction of Cuntz-Pimsner rings is a
generalization of, for example, \emph{the crossed product of a ring by an automorphism},
\emph{the Leavitt path algebra of a directed graphs} and of \emph{the fractional skew monoid ring of a corner isomorphism}.
We also generalizes the \emph{Graded Uniqueness Theorem}
known from Leavitt path algebras to our class
of Cuntz-Pimsner rings, and describe the structure of the graded ideals of
$\mathcal{T}_{(P,Q,\psi)}$ (and thus of $\mathcal{O}_{(P,Q,\psi)}$, if it exists),
in terms of pairs of ideals of $R$.

We believe that our construction is interesting both from the
point of view of algebra and from the point of view of operator
algebra. Our construction unifies many interesting classes of
rings, and we believe it will provide us with the right frame for
studying properties, such as the ideal structure, the $K$-theory,
purely infiniteness, and the real and stable rank of these rings.
It is also worth mentioning that the construction of Cuntz-Pimsner
algebras have been generalized in several ways in $C^*$-algebra
(see for example \cite{fowler}, \cite{hirsh} and \cite{SY}), and
there is no reason to believe that the same cannot be done in the
algebraic setting. We also expect that other examples of
classes of $C^*$-algebras which can be obtained through the
Cuntz-Pimsner construction, such as $C^*$-algebras associated with
subshifts, can be adapted to the algebraic setting through our
construction. So this paper is hopefully only the first step on
the way of what we hope to be a fruitful adaption of work done in
operator algebra to the algebraic setting.

We also believe that if one is only interested in the
$C^*$-algebraic case, then there is some insight to be gained by
reading this paper. One reason is that $C^*$-algebras have some nice properties
not shared by all rings. For example a $C^*$-algebra is always non-degenerate and
semiprime. This means that things which automatically work in the
$C^*$-algebraic setting do not necessarily work in the algebraic
setting, and we believe that by studying the algebraic case, one
gain some insight into why things work the way they do in the
operator algebraic case.
Here are some of the specific differences between the $C^*$-algebraic case and the purely algebraic case:
\begin{enumerate}
\item In the algebraic case we are not just working with a single bimodule equipped with a inner product, but with more general systems consisting of two $R$-bimodules $Q$ and $P$ connected by a bimodule homomorphism $\psi:P\otimes Q\longrightarrow R$.
\item If we are working with a right degenerate ring, then the \emph{Fock space representation} does not have the universal property the \emph{Toeplitz representation} should have. We therefore have to construct the Toeplitz representation in a different way.
\item Unlike in the $C^*$-algebraic case, we do not in the algebraic case automatically have that
every representation will induce a representation of the finite rank
operators (which correspond to the compact operators) of the $R$-system in question. We therefore have to introduce a condition on the $R$-systems we are working with which
insures that very representation will induce a representation of the
finite rank operators. We do that by introducing the condition we call
\textbf{(FS)}. This is probably not the optimal condition, but it is quite
natural and satisfied by all the interesting examples we consider in this paper.
\item Unlike the Toeplitz and Cuntz-Pimsner $C^*$-algebras, the algebraic Toeplitz and Cuntz-Pimsner rings do not in general carry a \emph{gauge action}. Instead, we have to work with $\Z$-gradings.
\item In the algebraic case, it is not always the case that a representation with all the
properties the Cuntz-Pimsner representation should have, exists (that
it always exists in the $C^*$-algebraic case is because every $C^*$-algebra
is semiprime). We think this is an interesting fact on its own, but it
means that we in general have to work with relative Cuntz-Pimsner rings instead of Cuntz-Pimsner rings. 
\end{enumerate}

Another reason why we believe that our construction is interesting from the point of view of operator algebra is that since we  do not
have to worry about any norms or topology, our arguments become more
tangible than in the $C^*$-algebraic setting. This allows us for example
to put everything into a frame of category theory, something we think
makes this whole construction more transparent. We believe that something
similar can, and ought to, be done in the $C^*$-algebraic setting.

\subsection*{The contents of the paper}

The contents of this paper can be summarized as follows:

In Section $1$ we give some basic definitions and introduce $R$-systems
$(P,Q,\psi)$ (Definition \ref{def:rsystem}). We define the category
$\mathcal{C}_{(P,Q,\psi)}$ of surjective covariant representations of an $R$-system $(P,Q,\psi)$ (Definition \ref{def:category})
and we prove that this category has an initial object which we call
the \textit{Toeplitz  representation} (Theorem \ref{theor:toeplitz}).
We then introduce some
essential examples of this construction, namely $R$-systems
associated with ring automorphisms (Example \ref{examples:item:1})
and with oriented graphs (Example \ref{examples:graph-toeplitz}), and
we study their Toeplitz representations.

Section $2$ defines the ring of adjointable homomorphisms $\mathcal{L}_P(Q)$
(Definition \ref{def:adjoint})
as well as its ideal of the finite rank adjointable homomorphisms
$\mathcal{F}_P(Q)$ (Definition \ref{def:finiterank}
and gives us the \emph{Fock space representation}
(Proposition \ref{toeplitz_rep})
which we later show is isomorphic to the Toeplitz representation under
certain conditions (Proposition \ref{prop:fock}).

In Section $3$ we show that the \emph{Toeplitz ring} $\mathcal{T}_{(P,Q,\psi)}$,
on which the Toeplitz representation of an $R$-system $(P,Q,\psi)$ lives,
comes with a $\Z$-grading (Proposition \ref{prop:Z-grading}).
We then go on to study graded and injective representations of $(P,Q,\psi)$;
that is representations which are compatible with the $\Z$-grading of
$\mathcal{T}_{(P,Q,\psi)}$ (Definition \ref{defi:graded rep})
and for which the representation of $R$ is injective (Definition \ref{def:representations}).
To do this we need that representations of $(P,Q,\psi)$ induces representations of
$\mathcal{F}_P(Q)$. In contract to the $C^*$-algebraic case where a representation of a
Hilbert bimodule always induces a representation of the compact operators of the
bimodule, a representation of $(P,Q,\psi)$ does not automatically induces a representation
of $\mathcal{F}_P(Q)$. We introduce a condition called \textbf{(FS)} on $(P,Q,\psi)$
(Definition \ref{defi:FS}) which guarantees that every representation of $(P,Q,\psi)$ induces
a representation of $\mathcal{F}_P(Q)$ (Proposition \ref{lemma_2}).
Under this condition we define the relative Cuntz-Pimsner
ring $\mathcal{O}_{(P,Q,\psi)}(J) $ of an $R$-system $(P,Q,\psi)$
with respect to an ideal $J$ as a certain quotient of the Toeplitz
ring $\mathcal{T}_{(P,Q,\psi)}$ (Definition \ref{def:relativeCP}), and we show that the
representations of $(P,Q,\psi)$ corresponding to these relative Cuntz-Pimsner rings, up to
isomorphism, include all graded and injective representations of
$(P,Q,\psi)$ (Remark \ref{remark:classi}).

In Section 4 we use the classification of graded and injective representations obtained in
Section 3 to first show that under certain conditions the Fock representation of an
$R$-system is isomorphic to the Toeplitz representation (Proposition \ref{prop:fock}),
and we then show that a relative Cuntz-Pimsner ring $\mathcal{O}_{(P,Q,\psi)}(J)$ satisfies
the \textit{Graded Uniqueness Theorem} (Definition \ref{def:gun})
if and only if the ideal $J$ is maximal among the ideals of $R$ for which the corresponding
relative Cuntz-Pimsner representation is injective (Theorem \ref{uniqueness}).
We also show by example that there can be more than one such maximal ideal
(Example \ref{example_two_maximal}). This is in
contrast to the $C^*$-algebraic case where there always exists a unique such maximal
ideal.

If such a unique maximal ideal exists, then we define the \emph{Cuntz-Pimsner}
representation of the $R$-system in question to be the relative Cuntz-Pimsner
representation corresponding to this maximal ideal (Definition \ref{def:CK}).
We do this in Section 5 where we also give conditions under which such a unique maximal
ideal exists (Lemma \ref{lemma:perp} and \ref{lemma:semiprime})
and show that several interesting examples satisfy these conditions
(Example \ref{examples_cuntz:cross}, \ref{exam:endo}, \ref{examples_cuntz:skew}
and \ref{examples_cuntz:graph_cuntz}).
We then show that the \emph{Cuntz-Pimsner ring}, the ring on which the
Cuntz-Pimsner representation lives, automatically satisfies the Graded Uniqueness
Theorem (Corollary \ref{uniqueness_semiprime})
and use this to show that we can construct the Leavitt path algebras
(Example \ref{examples_cuntz:graph_cuntz})),
the crossed product of a ring $R$ by an automorphism
(Example \ref{examples_cuntz:cross}) and the fractional skew monoid ring of a corner
isomorphism (Example \ref{examples_cuntz:skew}) as Cuntz-Pimsner rings.

In Section $6$ we generalize the
\emph{Algebraic Gauge-Invariant Uniqueness Theorem}
of \cite{AALP} to our Cuntz-Pimsner rings (Corollary \ref{GIUTFCP}),
and thereby to all Leavitt Path algebras (Corollary \ref{cor:GIUTLA}).

Finally in Section $7$ we extend the classification of graded and injective representations
obtained in Section 3 to graded representations which are not necessarily injective
(Remark \ref{remark:classification-noninjective}) and use this classification to give a
complete description of the graded ideals of relative Cuntz-Pimsner rings (and thereby of
Toeplitz rings, and of Cuntz-Pimsner rings) in terms of certain pairs of ideals of $R$
(Theorem \ref{ideals_cuntz} and Corollary \ref{ideals_toeplitz} and \ref{corol:ideal}).

\section{The Toeplitz ring}

First we establish the basic definitions for our setting.
Throughout the paper we set $\N_0=\N\cup\{0\}$.

A  ring $R$ is said to be \textit{right (left) non-degenerate} if
$rR=0$ ($Rr=0$) implies $r=0$. A ring $R$ is said to be
\textit{non-degenerate} if it is both right and left non-degenerate.
A non-degenerate has \textit{local units} if for every finite set
$\{r_1,\ldots,r_n\}\subseteq R$ there exists an idempotent $e\in R$
such that $r_i\in eRe$ for every $i= 1,\ldots,n$.

Let $R$ be a ring.
Given two $R$-bimodules $P$ and $Q$ we will by $P\otimes Q$ denote
the $R$-balanced tensor product.

\subsection{$R$-systems, covariant representations and the Toeplitz representation}
\label{sec:r-systems-covariant}

\begin{defi} \label{def:rsystem}
  Let $R$ be a ring.
  An \emph{$R$-system} is a triple $(P,Q,\psi)$ where $P$ and $Q$
  are $R$-bimodules, and $\psi$ is a $R$-bimodule homomorphism from
  $P\otimes Q$ to $R$.
\end{defi}

\begin{defi}[{Cf. \cite[Definition 2.11]{MS}}] \label{def:representations}
  Let $R$ be a ring and $(P,Q,\psi)$ an $R$-system.
  We say that a quadruple $(S,T,\sigma,B)$ is a \emph{covariant representation}
  of $(P,Q,\psi)$ on  $B$ if
  \begin{enumerate}
  \item $B$ is a ring,
  \item $S:P\longrightarrow B$ and $T:Q\longrightarrow B$ are linear
    maps,
  \item $\sigma:R \longrightarrow B$ is a ring
    homomorphism,
  \item $S(pr)=S(p)\sigma(r)$, $S(rp)=\sigma(r)S(p)$, $T(qr)=T(q)\sigma(r)$ and
    $T(rq)=\sigma(r)T(q)$ for every $r\in R$, $p\in P$ and
    $q\in Q$,
  \item $\sigma(\psi(p\otimes q))=S(p)T(q)$ for every $p\in P$ and $q\in
    Q$.
  \end{enumerate}
  We denote by $\mathcal{R}\langle S,T,\sigma\rangle$ the subring of
  $B$ generated by $\sigma(R)\cup T(Q) \cup S(P)$.
  If $\mathcal{R}\langle S,T,\sigma\rangle=B$, then we say that the
  covariant representation $(S,T,\sigma,B)$ is \emph{surjective}, and
  if the ring homomorphism $\sigma$ is injective, then we say that the
  covariant representation $(S,T,\sigma,B)$ is \emph{injective}.
\end{defi}

\begin{exems}
\text{}

\begin{enumerate}
\item Let $R$ be any ring and let $P=Q=R$ be the regular
$R$-bimodules. Defining  $\psi:P\otimes Q\longrightarrow R$ by
$\psi(p\otimes q)=pq$. We then have that $(P,Q,\psi)$ is an $R$-system. We
can define a covariant representation $(S,T,\sigma,R[t,t^{-1}])$,
where $R[t,t^{-1}]$ is the Laurent polynomial ring with coefficients
in $R$, by letting $T(q)=qt$, $S(p)=pt^{-1}$ and $\sigma(r)=r$ for every
$p\in P$, $q\in Q$ and $r\in R$. It is easy to check that
$(S,T,\sigma,R[t,t^{-1}])$ is indeed a covariant representation of
$(P,Q,\psi)$. Observe that this representation is injective and surjective.

\item Let $P=Q$ be the $\mathbb{R}$-module $\mathbb{R}$.
Define $\psi:\mathbb{R}\otimes \mathbb{R}\longrightarrow
\mathbb{R}$ by $\psi(p\otimes q)=-pq$. We then have that $(P,Q,\psi)$ is
an $\mathbb{R}$-system. We can then define a covariant
representation $(S,T,\sigma,\mathbb{C})$ by letting $T(q)=q\textbf{i}$,
$S(p)=p\textbf{i}$ and $\sigma(r)=r$ for every $p\in P$, $q\in Q$
and $r\in R$. This representation is injective and surjective.

\item Let $P=Q$ be the $\mathbb{Z}$-module $\mathbb{Z}$.
Then if given any $a\in \mathbb{Z}$ we define
$\psi_a:\mathbb{Z}\otimes \mathbb{Z}\longrightarrow \mathbb{Z}$ by
$\psi(p\otimes q)=apq$, then  $(P,Q,\psi_a)$ is a $\mathbb{Z}$-system.
We can then define a covariant representation
$(S,T,\sigma,\mathbb{C})$  by letting $T(q)=q\sqrt{a}$, $S(p)=p\sqrt{a}$ and
$\sigma(r)=r$ for every $p\in P$, $q\in Q$ and $r\in \mathbb{Z}$. Notice that
the representation $(S,T,\sigma,\mathbb{C})$ is injective but not surjective.

\item Let $V$ be a $K$-vector space and let $\mathcal{Q}(-,-):V\times
V\longrightarrow K$ be a non-degenerate quadratic form. Then $V$ is a
$K$-module, and if we let $P=Q=V$ and define
$\psi_\mathcal{Q}:V\otimes V\longrightarrow K$ by
$\psi_\mathcal{Q}(p\otimes q)=\mathcal{Q}(p,q)$, then
$(P,Q,\psi_\mathcal{Q})$ is a $K$-system. Recall that the Clifford
algebra $\mathcal{C}l(V,\mathcal{Q})$ is the universal unital
$K$-algebra generated by $V$ and with the relation
$v^2=\mathcal{Q}(v,v)\textbf{1}$ for every $v\in V$. Therefore we
can define a covariant representation
$(S,T,\sigma,\mathcal{C}l(V,\mathcal{Q}))$ of
$(P,Q,\psi_\mathcal{Q})$ by letting $T(v)=v$, $S(v)=v$ and
$\sigma(k)=k\textbf{1}$. This representation is surjective.

\end{enumerate}
\end{exems}

\begin{defi} \label{def:category}
  Let $R$ be a ring and $(P,Q,\psi)$ an $R$-system. We denote by
  $\category$ the category whose objects are surjective covariant
  representations $(S,T,\sigma,B)$ of $(P,Q,\psi)$,
  and where the class of morphisms between two
  representations $(S_1,T_1,\sigma_1,B_1)$ and
  $(S_2,T_2,\sigma_2,B_2)$ is the class of ring homomorphisms
  $\phi:B_1\longrightarrow B_2$ such that $\phi\circ T_1=T_2$,
  $\phi\circ S_1=S_2$ and $\phi\circ\sigma_1=\sigma_2$.
\end{defi}

The main purpose of this paper is, for a given $R$-system
$(P,Q,\psi)$, to study the category $\category$. First we will show
that $\category$ has an initial object, but we begin with some more
definitions and an easy lemma.

Given an $R$-system $(P,Q,\psi)$ we define recursively the
$R$-bimodules $P^{\otimes n}$ and $Q^{\otimes n}$ by letting
$P^1=P$ and $Q^1=Q$, and letting
$P^{\otimes n}=P^{\otimes n-1}\otimes P$ and $Q^{\otimes n}=
Q^{\otimes n-1}\otimes Q$ for $n>1$. We also let
$P^{\otimes 0}=Q^{\otimes 0}=R$.
We then define $\psi_0:P^0\otimes Q^0\longrightarrow R$ by
\begin{equation*}
  r_1\otimes r_2 \longmapsto r_1r_2
\end{equation*}
for $r_1,r_2\in R$, and we let $\psi_1=\psi$ and define recursively
$\psi_n: P^{\otimes n}\otimes Q^{\otimes n} \longrightarrow R$ for $n>1$
by
\begin{equation*}
  (p_1\otimes p_2)\otimes (q_1\otimes q_2) \longmapsto
  \psi\bigl(p_1\cdot\psi_{n-1}(p_2\otimes q_1)\otimes q_2\bigr)
\end{equation*}
for $p_2\in P^{\otimes n-1}$, $p_1\in P$, $q_1\in Q^{\otimes
n-1}$ and $q_2\in Q$.

\begin{lem} \label{lemma:tensor}
Let $R$ be ring and $(P,Q,\psi)$ an $R$-system, and let $(S,T,\sigma,B)$
be a covariant representation of $(P,Q,\psi)$. For each $n\in\N$ there
exist linear maps $T^n:Q^{\otimes n}\longrightarrow B$ and
$S^n:P^{\otimes n}\longrightarrow B$ such that $T^n(q_1\otimes
q_2\otimes\dots\otimes q_n)=T(q_1)T(q_2)\dots T(q_n)$ and
$S^n(p_1\otimes p_2\otimes\dots\otimes p_n)=S(p_1)S(p_2)\dots
S(p_n)$.
\end{lem}

\begin{proof}
Easily follows from the universal property of tensor products.
\end{proof}

Gradings by the following semigroup will play an important role in
this paper.

\begin{defi} \label{defi:semigroup}
  We define $\semigroup$ to be the semigroup $\No^2$ with
  multiplication defined by
  \begin{equation*}
    (m,n)(k,l)=
    \begin{cases}
      (m,n-k+l)&\text{if }n\ge k,\\
      (m+k-n,l)&\text{if }k\ge n.
    \end{cases}
  \end{equation*}
\end{defi}

We are now ready to show that the category $\category$ has an initial object.

\begin{theor}[{Cf. \cite{PI}}] \label{theor:toeplitz}
  Let $R$ be a ring and $(P,Q,\psi)$ an $R$-system. Then there exists
  an injective and surjective covariant representation
  $(\iota_P,\iota_Q,\iota_R,\mathcal{T}_{(P,Q,\psi)})$ with the
  following property:
  \begin{enumerate}[label=\textbf{(TP)}]
  \item If $(S,T,\sigma,B)$ is a covariant representation of $(P,Q,\psi)$, then
    there exists a unique ring homomorphism
    $\eta_{(S,T,\sigma,B)}:\mathcal{T}_{(P,Q,\psi)}\longrightarrow B$ such that
    $\eta_{(S,T,\sigma,B)}\circ\iota_R=\sigma$,
    $\eta_{(S,T,\sigma,B)}\circ\iota_Q=T$ and
    $\eta_{(S,T,\sigma,B)}\circ\iota_P=S$. \label{item:6}
  \end{enumerate}
  Moreover,
  $(\iota_P,\iota_Q,\iota_R,\mathcal{T}_{(P,Q,\psi)})$ is the, up to
  isomorphism in $\category$, unique surjective covariant
  representation of $(P,Q,\psi)$ which possesses
  the property \ref{item:6}; in fact, if $(S,T,\sigma,B)$ is a
  surjective covariant representation of $(P,Q,\psi)$ and
  $\phi:B\longrightarrow\mathcal{T}_{(P,Q,\psi)}$ is a ring homomorphism such that
  $\phi\circ\sigma=\iota_R$, $\phi\circ S=\iota_P$ and $\phi\circ
  T=\iota_Q$, then $\phi$ is an isomorphism.

  If we for $m,n\in\N$ let $\mathcal{T}_{(m,n)}:=
  \spa\{\iota_Q^m(q)\iota_P^n(p)\mid q\in Q^{\otimes m},\ p\in P^{\otimes n}\}$,
  and we for $k\in\N$ let $\mathcal{T}_{(k,0)}:=\iota_Q^k(Q^{\otimes
    k})$ and $\mathcal{T}_{(0,k)}:=\iota_P^k(P^{\otimes k})$, and we
  let $\mathcal{T}_{(0,0)}:=\iota_R(R)$, then
  $\oplus_{(m,n)\in\semigroup}\mathcal{T}_{(m,n)}$ is a
  $\semigroup$-grading of $\mathcal{T}_{(P,Q,\psi)}$.
  The grading
  $\oplus_{(m,n)\in\semigroup}\mathcal{T}_{(m,n)}$ is the only
  $\semigroup$-grading $\oplus_{(m,n)\in\semigroup}\mathcal{Y}_{(m,n)}$ of
  $\mathcal{T}_{(P,Q,\psi)}$ such that $\iota_R(R)\subseteq
  \mathcal{Y}_{(0,0)}$,  $\iota_Q(Q)\subseteq \mathcal{Y}_{(1,0)}$,
  and  $\iota_P(P)\subseteq \mathcal{Y}_{(0,1)}$.

  We call $(\iota_P,\iota_Q,\iota_R,\mathcal{T}_{(P,Q,\psi)})$
  \emph{the Toeplitz representation of $(P,Q,\psi)$}, and
  $\mathcal{T}_{(P,Q,\psi)}$ for \emph{the Toeplitz ring of $(P,Q,\psi)$}.
\end{theor}

\begin{proof}
  For $(m,n)\in\N^2$ let $\mathcal{T}_{(m,n)}$ be the free abelian group
  generated by elements $\{[q,p]\mid q\in Q^{\otimes m},\ p\in
  P^{\otimes n}\}$ satisfying the relations
  \begin{itemize}
  \item $[q,p_1]+[q,p_2]=[q,p_1+p_2]$ for $q\in Q^{\otimes m}$ and
    $p_1,p_2\in P^{\otimes n}$,
  \item $[q_1,p]+[q_2,p]=[q_1+q_2,p]$ for $q_1,q_2\in Q^{\otimes m}$ and
    $p\in P^{\otimes n}$,
  \item $[qr,p]=[q,rp]$ for $r\in R$, $q\in Q^{\otimes m}$ and
    $p\in P^{\otimes n}$.
  \end{itemize}
  For $k\in\N$ let $\mathcal{T}_{(k,0)}$ be the abelian group
  $\{[q]\mid q\in Q^{\otimes k}\}$ with addition defined by
  $[q_1]+[q_2]=[q_1+q_2]$ for $q_1,q_2\in Q^{\otimes k}$ (so
  $\mathcal{T}_{(k,0)}$ is just a copy of the abelian group $Q^{\otimes k}$),
  and let
  $\mathcal{T}_{(0,k)}$ be the abelian group $\{[p]\mid p\in
  P^{\otimes k}\}$ with addition defined by $[p_1]+[p_2]=[p_1+p_2]$
  for $p_1,p_2\in P^{\otimes k}$ (so
  $\mathcal{T}_{(0,k)}$ is just a copy of the abelian group $P^{\otimes
    k}$). Finally, let $\mathcal{T}_{(0,0)}$ be the abelian group
  $\{[r]\mid t\in R\}$ with addition defined by $[r_1]+[r_2]=[r_1+r_2]$
  for $r_1,r_2\in R$ (so
  $\mathcal{T}_{(0,0)}$ is just a copy of the abelian group $R$). We let
  $\mathcal{T}_{(P,Q,\psi)}:=\oplus_{(m,n)\in\semigroup}\mathcal{T}_{(m,n)}$.
  It is not difficult (but a bit tedious) to show that there exists a
  unique multiplication on $\mathcal{T}_{(P,Q,\psi)}$ satisfying
  \begin{itemize}
  \item $[r_1][r_2]=[r_1r_2]$ for $r_1,r_2\in R$,
  \item $[r][q]=[rq]$ for $r\in R$, $q\in Q^{\otimes k}$, $k\in\N$,
  \item $[q][r]=[qr]$ for $q\in Q^{\otimes k}$, $r\in R$, $k\in\N$,
  \item $[r][p]=[rp]$ for $r\in R$, $p\in P^{\otimes k}$, $k\in\N$,
  \item $[p][r]=[pr]$ for $p\in P^{\otimes k}$, $r\in R$, $k\in\N$,
  \item $[q][p]=[q,p]$ for $q\in Q^{\otimes k}$, $p\in P^{\otimes l}$, $k,l\in\N$,
  \item $[p][q]=[\psi_k(p\otimes q)]$ for $q\in Q^{\otimes k}$,
    $p\in P^{\otimes k}$, $k\in\N$,
  \item $[p][q_1\otimes q_2]=[\psi_k(p\otimes q_1)q_2]$ for $p\in
    P^{\otimes k}$, $q_1\in Q^{\otimes k}$, $q_2\in Q^{\otimes l}$, $k,l\in\N$,
  \item $[p_1\otimes p_2][q]=[p_1,\psi_l(p_2\otimes q)]$ for $p_1\in
    P^{\otimes k}$, $p_2\in Q^{\otimes l}$, $q\in Q^{\otimes l}$, $k,l\in\N$,
  \item $[r][q,p]=[rq,p]$ for $r\in R$, $q\in Q^{\otimes k}$, $p\in
    P^{\otimes l}$, $k,l\in\N$,
  \item $[q,p][r]=[q,pr]$ for $q\in Q^{\otimes k}$, $p\in
    P^{\otimes l}$, $r\in R$, $k,l\in\N$,
  \item $[q_1][q_2,p]=[q_1\otimes q_2,p]$ for $q_1\in Q^{\otimes k}$,
    $q_2\in Q^{\otimes l}$, $p\in P^{\otimes m}$, $k,l,m\in\N$,
  \item $[q_1,p][q_2]=[q_1\psi_l(p\otimes q_2)]$ for $q_1\in
    Q^{\otimes k}$, $p\in P^{\otimes l}$, $q_2\in Q^{\otimes l}$, $k,l\in\N$,
  \item $[q_1,p_1\otimes p_2][q_2]=[q_1,p_1\psi_m(p_2\otimes q_2)]$
    for $q_1\in Q^{\otimes k}$, $p_1\in P^{\otimes l}$, $p_2\in
    P^{\otimes m}$, $q_2\in Q^{\otimes m}$, $k,l,m\in\N$,
  \item $[q_1,p][q_2\otimes q_3]=[q_1\psi_l(p\otimes q_2)\otimes q_3]$
    for $q_1\in Q^{\otimes k}$, $p\in P^{\otimes l}$, $q_2\in
    Q^{\otimes l}$, $q_3\in Q^{\otimes m}$, $k,l,m\in\N$,
  \item $[p_1][q,p_2]=[\psi_k(p_1\otimes q)p_2]$ for $p_1\in
    P^{\otimes k}$, $q\in Q^{\otimes k}$, $p_2\in P^{\otimes l}$,
    $k,l\in\N$,
  \item $[p_1\otimes p_2][q,p_3]=[p_1\psi_l(p_2\otimes q)\otimes p_3]$
    for $p_1\in P^{\otimes k}$, $p_2\in P^{\otimes l}$, $q\in
    Q^{\otimes l}$, $p_3\in P^{\otimes m}$, $k,l,m\in\N$,
  \item $[p_1][q_1\otimes q_2,p_2]=[\psi_k(p_1\otimes q_1)q_2,p_2]$
    for $p_1\in P^{\otimes k}$, $q_1\in Q^{\otimes k}$, $q_2\in
    Q^{\otimes l}$, $p_2\in P^{\otimes m}$, $k,l,m\in\N$,
  \item $[q,p_1][p_2]=[q,p_1\otimes p_2]$ for $q\in Q^{\otimes k}$,
    $p_1\in P^{\otimes l}$, $p_2\in P^{\otimes m}$, $k,l,m\in\N$,
  \item $[q_1,p_1][q_2,p_2]=[q_1\psi_l(p_1\otimes q_2),p_2]$
    for $q_1\in Q^{\otimes k}$, $p_1\in P^{\otimes l}$, $q_2\in Q^{\otimes
      l}$, $p_2\in P^{\otimes m}$, $k,l,m\in\N$,
  \item $[q_1,p_1][q_2\otimes q_3,p_2]=[q_1\psi_l(p_1\otimes
    q_2)\otimes q_3,p_2]$ for $q_1\in Q^{\otimes k}$, $p_1\in
    P^{\otimes l}$, $q_2\in Q^{\otimes l}$, $q_3\in Q^{\otimes m}$,
    $p_2\in P^{\otimes n}$, $k,l,m,n\in\N$,
  \item $[q_1,p_1\otimes p_2][q_2,p_3]=[q_1,p_1\psi_m(p_2\otimes q_2)\otimes p_3]$
    for $q_1\in Q^{\otimes k}$, $p_1\in P^{\otimes l}$, $p_2\in P^{\otimes
      m}$, $q_2\in Q^{\otimes m}$, $p_3\in P^{\otimes n}$,
    $k,l,m,n\in\N$.
  \end{itemize}
  With this $\mathcal{T}_{(P,Q,\psi)}$ becomes a ring.

  Let $\iota_R:R\longrightarrow\mathcal{T}_{(P,Q,\psi)}$ be the map $r\longmapsto [r]$,
  $\iota_Q:Q\longrightarrow\mathcal{T}_{(P,Q,\psi)}$ the map $q\longmapsto [q]$, and
  $\iota_P:P\longrightarrow\mathcal{T}_{(P,Q,\psi)}$ the map $p\longmapsto [p]$. Then
  $(\iota_P,\iota_Q,\iota_R,\mathcal{T}_{(P,Q,\psi)})$ is an injective
  and surjective
  covariant representation of $(P,Q,\psi)$.

  Let $(S,T,\sigma,B)$ be a covariant representation of $(P,Q,\psi)$. Since
  $\mathcal{T}_{(P,Q,\psi)}$ is generated by
  $\iota_R(R)\cup\iota_Q(Q)\cup\iota_P(P)$, there can at most be one
  ring homomorphism $\eta_{(S,T,\sigma,B)}:\mathcal{T}_{(P,Q,\psi)}\longrightarrow
  B$ such that
  $\eta_{(S,T,\sigma,B)}\circ\iota_R=\sigma$,
  $\eta_{(S,T,\sigma,B)}\circ\iota_Q=T$ and
  $\eta_{(S,T,\sigma,B)}\circ\iota_P=S$. For
  $(m,n)\in\N^2$ the set $\spa\{T^m(q)S^n(p)\mid q\in Q^{\otimes m},
  p\in P^{\otimes n}\}$ is a subgroup of $B$ in which the relations
  \begin{itemize}
  \item $T^m(q)S^n(p_1)+T^m(q)S^n(p_2)=T^m(q)S^n(p_1+p_2)$ for $q\in
    Q^{\otimes m}$ and $p_1,p_2\in P^{\otimes n}$,
  \item $T^m(q_1)S^n(p)+T^m(q_2)S^n(p)=T^m(q_1+q_2)S^n(p)$ for
    $q_1,q_2\in Q^{\otimes m}$ and $p\in P^{\otimes n}$,
  \item $T^m(qr)S^n(p)=T^m(q)S^n(rp)$ for $r\in R$, $q\in Q^{\otimes m}$ and
    $p\in P^{\otimes n}$,
  \end{itemize}
  are satisfied,
  so there exists a group homomorphism $\eta_{(m,n)}$ from
  $\mathcal{T}_{(m,n)}$ to $B$ such that
  $\eta_{(m,n)}([q,p])=T^m(q)S^n(p)$ for $q\in Q^{\otimes m}$ and
  $p\in P^{\otimes n}$. For $k\in\N$ let $\eta_{(k,0)}$ denote the map
  $T^k$, and let $\eta_{(0,k)}$ denote the map $S^k$. Finally, let
  $\eta_{(0,0)}$ denote the map $\sigma$. Then there exists a linear map
  $\eta_{(S,T,\sigma,B)}:\mathcal{T}_{(P,Q\psi)}\longrightarrow B$ such that for each
  $(m,n)\in\semigroup$ the
  restriction of $\eta_{(S,T,\sigma,B)}$ to $\mathcal{T}_{(m,n)}$ is equal
  to $\eta_{(m,n)}$. It is not difficult to check
  that $\eta_{(S,T,\sigma,B)}$ is multiplicative, and thus a ring
  homomorphism. It is
  clear that $\eta_{(S,T,\sigma,B)}\circ\iota_R=\sigma$,
  $\eta_{(S,T,\sigma,B)}\circ\iota_Q=T$ and
  $\eta_{(S,T,\sigma,B)}\circ\iota_P=S$.
  Thus the representation
  $(\iota_P,\iota_Q,\iota_R,\mathcal{T}_{(P,Q,\psi)})$ possesses
  property \ref{item:6}.

  If $(S,T,\sigma,B)$ is a
  surjective covariant representation of $(P,Q,\psi)$ and
  $\phi:B\longrightarrow\mathcal{T}_{(P,Q,\psi)}$ is a ring homomorphism such that
  $\phi\circ\sigma=\iota_R$, $\phi\circ S=\iota_P$ and $\phi\circ
  T=\iota_Q$, then
  $\eta_{(S,T,\sigma,B)}\circ\phi(\sigma(r))=\sigma(r)$ for all $r\in
  R$, $\eta_{(S,T,\sigma,B)}\circ\phi(S(p))=S(p)$ for all $p\in P$ and
  $\eta_{(S,T,\sigma,B)}\circ\phi(T(q))=T(q)$ for all $q\in Q$.
  Since $B$ is generated by $\sigma(R)\cup S(P)\cup T(Q)$, it follows
  that $\eta_{(S,T,\sigma,B)}\circ\phi$ is equal to the identity map
  of $B$. One can in a similar way show that
  $\phi\circ\eta_{(S,T,\sigma,B)}$ is equal to the identity map of
  $\mathcal{T}_{(P,Q,\psi)}$. Thus $\phi$ and $\eta_{(S,T,\sigma,B)}$
  are each other inverse,
  and $\phi$ is an isomorphism.

  It is clear that $\mathcal{T}_{(m,n)}=
  \spa\{\iota_Q^m(q)\iota_P^n(p)\mid q\in Q^{\otimes m},\ p\in
  P^{\otimes n}\}$ for $m,n\in\N$, that
  $\mathcal{T}_{(k,0)}=\iota_Q^k(Q^{\otimes
    k})$ and $\mathcal{T}_{(0,k)}=\iota_P^k(P^{\otimes k})$ for $k\in\N$, that
  $\mathcal{T}_{(0,0)}=\iota_R(R)$, and that
  $\oplus_{(m,n)\in\semigroup}\mathcal{T}_{(m,n)}$ is a
  $\semigroup$-grading of $\mathcal{T}_{(P,Q,\psi)}$.

  If $\oplus_{(m,n)\in\semigroup}\mathcal{Y}_{(m,n)}$ is another
  $\semigroup$-grading of $\mathcal{T}_{(P,Q,\psi)}$ such that
  $\iota_R(R)\subseteq \mathcal{Y}_{(0,0)}$, $\iota_Q(Q)\subseteq
  \mathcal{Y}_{(1,0)}$, and $\iota_P(P)\subseteq\mathcal{Y}_{(0,1)}$,
  then it follows that
  $\mathcal{T}_{(m,n)}\subseteq\mathcal{Y}_{(m,n)}$ for each
  $(m,n)\in\semigroup$, and thus that
  $\mathcal{T}_{(m,n)}=\mathcal{Y}_{(m,n)}$ for each $(m,n)\in\semigroup$.
\end{proof}

\begin{rema}
  It follows from Theorem \ref{theor:toeplitz} that the Toeplitz
  representation is an initial object of $\category$. It also follows
  that there is a bijective correspondence between covariant representations of
  an $R$-system $(P,Q,\psi)$ and ring homomorphisms defined on
  $\mathcal{T}_{(P,Q,\psi)}$.
\end{rema}

\subsection{Examples}
\label{sec:examples}

We end this section by looking at some examples. We will return to
these examples later in the paper.

\begin{exem}\label{examples:item:1}
Let $R$ be a ring and let $\varphi\in \aut(R)$ be a
ring automorphism. Let $P=:R_\varphi$ be the $R$-bimodule with the
right action defined by $p\cdot r=p\varphi(r)$ and the left action
defined by $r\cdot p=rp$ for $p\in P$ and $r\in R$. Likewise, let
$Q:=R_{\varphi\inv}$ be the $R$-bimodule with the right action
defined by $q\cdot r=q\varphi\inv(r)$ and the left action defined by
$r\cdot q=rq$ for $q\in Q$ and $r\in R$. Thus we can define the
following bimodule homomorphism:
\begin{eqnarray*}
  \psi:P\otimes_R Q & \longrightarrow R \\
p\otimes q & \longmapsto p\varphi(q).
\end{eqnarray*}

Notice that we for every $n\in\N$ have that $P^{\otimes n}$ is
isomorphic to $R_{\varphi^n}$ and that
$Q^{\otimes n}$ is isomorphic to $R_{\varphi^{-n}}$. We will in the
following for every $n\in\No$ identify $P^{\otimes n}$ and $Q^{\otimes
  n}$ with $R$. We then have that $p_1\otimes
p_2=p_1\varphi^{n_1}(p_2)$ for $p_1\in P^{\otimes n_1}$ and $p_2\in
P^{\otimes n_2}$, and that $q_1\otimes q_2=q_1\varphi^{-n_1}(q_2)$ for
$q_1\in Q^{\otimes n_1}$ and $q_2\in Q^{\otimes n_2}$.

Let $(S,T,\sigma,B)$ be a covariant representation of $(P,Q,\psi)$.
For $r\in R$ and $n\in\N$ let $[r,n]:=S^n(r)$, $[r,-n]:=T^n(r)$ and
$[r,0]:=\sigma(r)=T^0(r)=S^0(r)$. For $r_1,r_2\in R$ and
$n_1,n_2\in\No$ choose $u_1,u_2\in R$ such that $ur_1=r_1$ and
$r_2u_2=r_2$. Then we have
\begin{gather*}
\begin{split}
[r_1,n_1][r_2,n_2]&=S^{n_1}(r_1)S^{n_2}(r_2)=
S^{n_1+n_2}(r_1\otimes r_2)\\
&=S^{n_1+n_2}\bigl(r_1\varphi^{n_1}(r_2)\bigr)=[r_1\varphi^{n_1}(r_2),n_1+n_2],
\end{split}\\
\begin{split}
[r_1,-n_1][r_2,-n_2]&=T^{n_1}(r_1)T^{n_2}(r_2)=
T^{n_1+n_2}(r_1\otimes r_2)\\
&=T^{n_1+n_2}\bigl(r_1\varphi^{-n_1}(r_2)\bigr)
=[r_1\varphi^{-n_1}(r_2),-n_1-n_2],
\end{split}\\
\begin{split}
[r_1,n_1][r_2,-n_1]&=S^{n_1}(r_1)T^{n_1}(r_2)
=\sigma\bigl(\psi_{n_1}(r_1\otimes r_2)\bigr)\\
&=\sigma\bigl(r_1\varphi^{n_1}(r_2)\bigr)
=[r_1\varphi^{n_1}(r_2),0],
\end{split}\\
\begin{split}
[r_1,n_1+n_2][r_2,-n_2]&=[u_1r_1,n_1+n_2][r_2,-n_2]\\
&=[u_1,n_1][\varphi^{-n_1}(r_1),n_2][r_2,-n_2]\\
&=[u_1,n_1][\varphi^{-n_1}(r_1)\varphi^{n_2}(r_2),0]\\
&=[u_1r_1\varphi^{n_1+n_2}(r_2),n_1]=[r_1\varphi^{n_1+n_2}(r_2),n_1]
\end{split}
\end{gather*}
and
\begin{equation*}
\begin{split}
[r_1,n_1][r_2,-n_1-n_2]&=[r_1,n_1][r_2,-n_1][\varphi^{n_1}(u_2),-n_2]\\
&=[r_1\varphi^{n_1}(r_2),0][\varphi^{n_1}(u_2),-n_2]\\
&=[r_1\varphi^{n_1}(r_2),-n_2]
\end{split}
\end{equation*}
Thus $[r_1,k_1][r_2,k_2]=[r_1\varphi^{k_1}(r_2),k_1+k_2]$ for
$r_1,r_2\in R$ and $k_1,k_2\in\Z$ if $k_1$ and $k_2$ both are
non-positive, or both are non-negative, or if $k_1$ is
non-negative and $k_2$ is non-positive. We also have that
$[r_1,k]+[r_2,k]=[r_1+r_2,k]$ for $r_1,r_2\in R$ and $k\in\Z$.

If on the other hand we have a ring $B$ which contains a set of
elements $\{[r,k]: r\in R,\ k\in\Z\}$ satisfying
$[r_1,k]+[r_2,k]=[r_1+r_2,k]$ and
$[r_1,k_1][r_2,k_2]=[r_1\varphi^{k_1}(r_2),k_1+k_2]$ if $k_1$ and
$k_2$ both are non-positive, or both are non-negative, or if $k_1$
is non-negative and $k_2$ is non-positive, and we define
$\sigma:R\longrightarrow B$ by $\sigma(r)=[r,0]$,
$S:P\longrightarrow B$ by $S(p)=[p,1]$, and $T:Q\longrightarrow B$
by $T(q)=[q,-1]$, then $(S,T,\sigma,B)$ is a covariant
representation of $(P,Q,\psi)$.

Thus $\mathcal{T}_{(P,Q,\psi)}$ is the universal ring generated by
elements  $\{[r,k]: r\in R,\ k\in\Z\}$ satisfying
$[r_1,k]+[r_2,k]=[r_1+r_2,k]$ and
$[r_1,k_1][r_2,k_2]=[r_1\varphi^{k_1}(r_2),k_1+k_2]$ if $k_1$ and
$k_2$ both are non-positive, or both are non-negative, or if $k_1$
is non-negative and $k_2$ is non-positive. We will in Example
\ref{examples_cuntz:cross} see that if $R$ has local units then a
certain quotient of $\mathcal{T}_{(P,Q,\psi)}$ (\emph{the
  Cuntz-Pimsner ring of $(P,Q,\psi)$}) is isomorphic to the crossed product
$R\times_\varphi\Z$.
\end{exem}

\begin{exem}\label{examples:graph-toeplitz}
  Let $E=(E^0,E^1)$ be an oriented graph and let $F$ be a commutative unital
  ring. We define the ring $R:=\oplus_{v\in E^0} F_v$ where every
  $F_v$ is a copy of $F$, and we denote for
  each $v\in E^0$ by $\textbf{1}_v$ the unit of $F_v$. Observe that $R$ is
  non-degenerate with local units. We also define $Q:=\oplus_{e\in
    E^1} F_e$ and $P:=\oplus_{e\in E^1} F_{\overline{e}}$ where every
  $F_e$ and $F_{\overline{e}}$ are copies of $F$ with units
  $\textbf{1}_{e}$ and $\textbf{1}_{\overline{e}}$ respectively,
  with the following $R$-bimodule operations:
\begin{align*}
  \left(\sum_{e\in E^1} \lambda_e\textbf{1}_{e}\right)\cdot \left(\sum_{v\in E^0}
  s_v\textbf{1}_{v}\right) &= \sum_{e\in E^1} \left(\sum_{r(e)=v}\lambda_e
  s_v\right)\textbf{1}_{e} \,,\\
  \left(\sum_{v\in E^0} s_v\textbf{1}_{v}\right)\cdot\left(\sum_{e\in E^1}
  \lambda_e\textbf{1}_{e}\right) &= \sum_{e\in E^1} \left(\sum_{s(e)=v}s_v\lambda_e\right)\textbf{1}_{e}\,,\\
  \left(\sum_{e\in E^1} \lambda_e\textbf{1}_{\overline{e}}\right)\cdot \left(\sum_{v\in E^0}
  s_v\textbf{1}_{v}\right) &= \sum_{e\in E^1} \left(\sum_{s(e)=v}\lambda_e
  s_v\right)\textbf{1}_{\overline{e}} \,,\\
  \left(\sum_{v\in E^0} s_v\textbf{1}_{v}\right)\cdot\left(\sum_{e\in E^1}
  \lambda_e\textbf{1}_{\overline{e}}\right) &= \sum_{e\in E^1} \left(\sum_{r(e)=v}
  s_v\lambda_e\right)\textbf{1}_{\overline{e}}\,,
\end{align*}
for every $\{s_v\}_{v\in E^0}\subseteq F$ and $\{\lambda_e\}_{e\in
E^1}\subseteq F$.

Now if we define the following $R$-bimodule homomorphism
$$\begin{array}{rl}
\psi:P\otimes_R Q & \longrightarrow R \\
(\sum_{e\in E^1} p_e\textbf{1}_{\overline{e}})\otimes (\sum_{e\in
E^1} q_e\textbf{1}_{e}) & \longmapsto \sum_{v\in E^0} (\sum_{r(e)=v}
p_{e} q_e )\textbf{1}_{v},
\end{array}$$
then $(P,Q,\psi)$ is an $R$-system.

Let $(S,T,\sigma,B)$ be a covariant representation of $(P,Q,\psi)$
and let $p_v:=\sigma(\textbf{1}_v)$ for $v\in E^0$, and let
$x_e=T(\textbf{1}_e)$ and $y_e=S(\textbf{1}_{\overline{e}})$ for
$e\in E^1$. It is easy to check that $\{p_v\}_{v\in E^0}$ is a
family of pairwise orthogonal idempotents, and that for all $e,f\in
E^1$ we have that $p_{s(e)}x_{e}=x_{e}=x_{e}p_{r(e)}$,
$p_{r(e)}y_{e}=y_{e}=y_{e}p_{s(e)}$, and
$y_{e}x_{f}=\delta_{e,f}p_{r(e)}$. Since $R$ is an $F$-algebra, and
$P$ and $Q$ are $F$-modules, the ring $\mathcal{R}\langle
S,T,\sigma\rangle$ becomes an $F$-algebra when we equip it with an
$F$-multiplication of $F$ defined by
$\lambda\sigma(r)=\sigma(\lambda r)$, $\lambda S(p)=S(\lambda p)$
and $\lambda T(q)=T(\lambda q)$ for $\lambda\in F$, $r\in R$, $p\in
P$ and $q\in Q$.

If on the other hand $B$ is an $F$-algebra which contains a family
$\{p_{v}\}_{v\in E^0}$ of pairwise orthogonal idempotents and
families $\{x_{e}\}_{e\in E^1}$ and $\{y_{e}\}_{e\in E^1}$
satisfying for all $e,f\in E^1$ that
$p_{s(e)}x_{e}=x_{e}=x_{e}p_{r(e)}$,
$p_{r(e)}y_{e}=y_{e}=y_{e}p_{s(e)}$, and
$y_{e}x_{f}=\delta_{e,f}p_{r(e)}$, and we for $r=\sum_{v\in
E^0}s_v\textbf{1}_v\in R$ let $\sigma(r):=\sum_{v\in E^0}s_vp_{v}$,
for $p=\sum_{e\in E^1} \lambda_e\textbf{1}_{\overline{e}}\in P$ let
$S(p):=\sum_{e\in E^1}\lambda_ex_{e}$, and for $q=\sum_{e\in
E^1}\lambda_e\textbf{1}_{e}\in Q$ let $T(q):=\sum_{e\in
E^1}\lambda_ex_{e}$, then $(S,T,\sigma,B)$ is a covariant
representation of $(P,Q,\psi)$.

Thus $\mathcal{T}_E:=\mathcal{T}_{(P,Q,\psi)}$ is the universal
$F$-algebra generated by a set $\{p_{v}: v\in E^0\}$ of pairwise
orthogonal idempotents, together with a set $\{x_{e},y_{e}: e\in
E^1\}$ of elements satisfying for $e,f\in E^1$
\begin{enumerate}
\item $p_{s(e)}x_{e}=x_{e}=x_{e}p_{r(e)}$,
\item $p_{r(e)}y_{e}=y_{e}=y_{e}p_{s(e)}$,
\item $y_{e}x_{f}=\delta_{e,f}p_{r(e)}$.
\end{enumerate}
We will in Example \ref{examples_cuntz:graph_cuntz}
see that a certain quotient of $\mathcal{T}_{E}$ is
isomorphic to the Leavitt path algebra $L_F(E)$ associated with the
graph $E$, cf. \cite{AA1},\cite{AA2},\cite{AALP},\cite{AMP}\&\cite{TF}.
\end{exem}

\section{The Fock space representation}

We will in this section for an arbitrary ring $R$ and an arbitrary
$R$-system $(P,Q,\psi)$
construct a representation which we call \emph{the Fock space
  representation}. This construction is inspired by a similar
construction in the $C^*$-algebra setting, cf. \cite{PI} and \cite{KS1}. We will
later show (see Proposition \ref{prop:fock}) that the Fock space
representation under certain conditions is
isomorphic to the Toeplitz representation.

We begin by establishing some notation which will be used in the rest
of the paper.

\begin{defi} \label{def:adjoint}
Let $R$ be a ring and $(P,Q,\psi)$ an $R$-system. Then a right
$R$-module homomorphism $T:Q_R\longrightarrow Q_R$ is called
\textit{adjointable with respect to $\psi$} if there exists a left
$R$-module homomorphism
$S:{}_RP\longrightarrow {}_RP$ such that
$$\psi\bigl(p\otimes T(q)\bigr)
=\psi\bigl(S(p)\otimes q\bigr)\qquad \forall p\in P\qquad \forall q\in Q\,.$$

We call $S$ an \textit{adjoint} of $T$ with respect to $\psi$. We
write $\mathcal{L}_P(Q)$ for the set of all the adjointable
homomorphisms (with respect to $\psi$). Notice that without further
conditions the adjoint can be non-unique. We denote by
$\mathcal{L}_Q(P)$ the set of all the adjoints.
\end{defi}

Observe that $\mathcal{L}_P(Q)$ and $\mathcal{L}_Q(P)$ are subrings
of $\text{End}(Q_R)$ and $\text{End}({}_RP)$ respectively.

\begin{defi} \label{def:finiterank}
Let $R$ be a ring and $(P,Q,\psi)$ an $R$-system.
For every $p\in P$ and $q\in Q$ we
define the following homomorphisms
$$\begin{array}{rl}
\theta_{q,p}:Q_R & \longrightarrow Q_R \\
x & \longmapsto q\psi(p\otimes x)
\end{array}
\qquad
\begin{array}{rl}
\theta_{p,q}:{}_RP & \longrightarrow {}_RP \\
y & \longmapsto \psi(y\otimes q)p
\end{array}
\,.$$ Then $\theta_{q,p}\in \mathcal{L}_P(Q)$ and has $\theta_{p,q}$
as an adjoint.

We call these homomorphisms \textit{rank 1 adjointable}
homomorphisms, and we denote by $\mathcal{F}_P(Q)$ the linear span
of all the rank 1 adjointable homomorphisms. Similarly, we denote by
$\mathcal{F}_Q(P)$ the set of all rank 1 adjoints.
\end{defi}

\begin{lem}
Let $R$ be a ring and $(P,Q,\psi)$ an $R$-system. If $T\in
\mathcal{L}_P(Q)$ (with an adjoint $S$), $p\in P$ and $q\in Q$, then
we have that
$$T\Theta_{q,p}= \Theta_{T(q),p}\qquad \text{and}\qquad
\Theta_{q,p}T = \Theta_{q,S(p)}\,.$$ Thus $\mathcal{F}_P(Q)$ is a
two-sided ideal of $\mathcal{L}_P(Q)$.
\end{lem}
\begin{proof}
Is easy to check using the definitions.
\end{proof}
Notice that the above result does
not depend on the choice of the adjoint.
Notice also that by a dual argument we have that $\mathcal{F}_Q(P)$ is a
two-sided ideal of $\mathcal{L}_Q(P)$.

\begin{defi}[{Cf. \cite[Section 2.2]{MS} and \cite{PI}}]
Given a ring $R$ and an $R$-bimodule $Q$ we define the tensor ring or
\textit{Fock ring} $F(Q)$ by $$F(Q)=\bigoplus^{\infty}_{n=0} Q^{\otimes n}\,.$$
\end{defi}

Despite the inherited ring structure of $F(Q)$ (see \cite{GREEN} for
more information about tensor rings) we are only
interested in the $R$-bimodule structure of $F(Q)$. If $(P,Q,\psi)$ is
an $R$-system, then we can
define an $R$-balanced $R$-bilinear form
$$\begin{array}{rl}
\langle\cdot,\cdot \rangle: F(P)\times F(Q) & \longrightarrow R \\
(\{p_n\},\{q_n\}) & \longmapsto  \sum_{n\in\mathbb{N}_0}
\psi_n(p_n\otimes q_n)
\end{array}
$$
that one can extend to a $R$-bimodule homomorphism
$\psi:F(P)\otimes  F(Q)\longrightarrow R$ by the universal
property of the tensor product.

Define the ring homomorphism $\phi_\infty:R \longrightarrow
\mathcal{L}_{F(P)}(F(Q))$  assigning to $r\in R$ the adjointable
homomorphism $\phi_\infty(r)$ of $F(Q)$ defined by
$\phi_\infty(r)(\{q_n\})=\{rq_n\}$. Notice that $\varphi_\infty(r)$
defined as $\varphi_\infty(r)(\{p_n\})=\{p_nr\}$ is an adjoint of
$\phi_\infty(r)$.

If for every $n\in \No$ we define $\phi^n_\infty:R
\longrightarrow \mathcal{L}_{P^{\otimes n}}(Q^{\otimes n})$ as
$\phi^n_\infty(r)(q_n)=rq_n$, then we can write $\phi_\infty(r)$ in the
following matrix form
$$\phi_\infty(r)(\{q_n\})=\left( \begin{array}{cccc}
\phi^0_\infty(r) &  & 0 &  \\
 & \phi^1_\infty(r) &  & \\
0 &  & \phi^2_\infty(r) &  \\
 & & & \ddots
\end{array}\right) \left(  \begin{array}{c} q_0 \\ q_1 \\ q_2 \\ \vdots
\end{array}\right).$$

Given an $R$-system $(P,Q,\psi)$, for every $n,m\in\No$ with
$n\leq m$ and $q\in Q^{\otimes m-n}$, we define the following right
$R$-module homomorphism
$$\begin{array}{rl}
T^{(n,m)}_q: Q^{\otimes n} & \longrightarrow Q^{\otimes m} \\
q_n & \longmapsto  q\otimes q_n
\end{array}$$
and the left $R$-module homomorphism
$$\begin{array}{rl}
U^{(m,n)}_q: P^{\otimes m} & \longrightarrow P^{\otimes n} \\
p_1\otimes p_2 & \longmapsto  p_1\psi_{m-n}(p_2\otimes q)
\end{array},$$
where $p_1\in P^{\otimes n}$ and $p_2\in P^{\otimes m-n}$.

For $q\in Q$ let $T^{(n)}_q:=T^{(n,n+1)}_q$ and
$U^{(n)}_q:=U^{(n+1,n)}_q$. We define the \textit{creator
homomorphism} $T_q:F(Q)\longrightarrow F(Q)$ by
$$T_q(\{q_n\}):=\{0,T^{(0)}_q(q_0),T^{(1)}_q(q_1),\ldots\}=\{0,qq_0,q\otimes
q_1,\ldots\}\,.$$ Observe that we can write $T_q$ in the following
matrix form

$$T_q(\{q_n\})=\left( \begin{array}{ccccc}
0 &   &  &  & \\
 T^{(0)}_q& 0 &  &  &\\
   & T^{(1)}_q & 0 &  & \\
 &  & T^{(2)}_q& 0  & \\
& & & \ddots  & \ddots
\end{array}\right) \left(  \begin{array}{c} q_0 \\ q_1 \\ q_2 \\ q_3 \\ \vdots
\end{array}\right).$$
One gets that $T_q\in \mathcal{L}_{F(P)}(F(Q))$ with an adjoint
homomorphism $U_q:F(P)\longrightarrow F(P)$  defined by
$U_q(\{p_n\})=\{U^{(0)}_q(p_1),U^{(1)}_p(p_2),\ldots\}$ and which
can be written in the matrix form

$$U_q(\{p_n\})=\left( \begin{array}{cccccc}
0 &  U^{(0)}_q &   &  & & \\
 & 0 & U^{(1)}_q &   &  & \\
   &   & 0 & U^{(2)}_q &  &  \\
 & & & \ddots & \ddots &
\end{array}\right) \left(  \begin{array}{c} p_0 \\ p_1 \\ p_2 \\ p_3 \\ \vdots
\end{array}\right).$$

Similarly, for every $n,m\in\No$ with $n\leq m$ and  given any $p\in
P^{\otimes m-n}$ we define  the following right $R$-module
homomorphism
$$\begin{array}{rl}
S^{(n,m)}_p: Q^{\otimes m} & \longrightarrow Q^{\otimes n} \\
q_1\otimes q_2 & \longmapsto  \psi_{m-n}(p\otimes q_1)q_2
\end{array}\,,$$
where $q_1\in Q^{\otimes m-n}$ and $q_2\in Q^{\otimes n}$,
and the left $R$-module homomorphism

$$\begin{array}{rl}V^{(n,m)}_p: P^{\otimes n} & \longrightarrow P^{\otimes m} \\
p_n & \longmapsto  p_n\otimes p.
\end{array}$$

We denote by $S^{(n)}_p:=S^{(n,n+1)}_p$ and
$V^{(n)}_p:=V^{(n+1,n)}_p$ where $p\in P$, and we then define the
right $R$-module homomorphism $S_p:F(Q)\longrightarrow F(Q)$ by
$S_p(\{q_n\}):=\{S^{(0)}_p(q_1),S^{(1)}_p(q_2),\ldots\}$ which can
be written in the following matrix form

$$S_p(\{q_n\})=\left( \begin{array}{cccccc}
0 &  S^{(0)}_p & &  & & \\
 & 0 & S^{(1)}_p & &  & \\
   &   & 0 & S^{(2)}_p &  &  \\
 & & & \ddots & \ddots &
\end{array}\right) \left(  \begin{array}{c} q_0 \\ q_1 \\ q_2 \\ q_3 \\ \vdots
\end{array}\right).$$

One gets that $S_p\in \mathcal{L}_{F(P)}(F(Q))$ with an adjoint
homomorphism $V_p:F(P)\longrightarrow F(P)$ given by
$V_p(\{p_n\}):=\{0,V^{(0)}_p(p_0),V^{(1)}_p(p_1),\ldots\}$ and with
matrix form
$$V_p(\{p_n\})=\left( \begin{array}{ccccc}
0 &   &  &  & \\
 V^{(0)}_p& 0 &  &  &\\
  & V^{(1)}_p & 0 &  & \\
 & & V^{(2)}_p& 0  & \\
& & & \ddots  & \ddots
\end{array}\right) \left(  \begin{array}{c} p_0 \\ p_1 \\ p_2 \\ p_3 \\ \vdots
\end{array}\right).$$

\begin{prop}\label{toeplitz_rep}
  Let $R$ be a ring and $(P,Q,\psi)$ an $R$-system. Denote by
  $T_{\mathcal{F}}$ the map from $Q$ to
  $\mathcal{L}_{F(P)}(F(Q))$ given by $q\longmapsto T_q$, by $S_{\mathcal{F}}$
  the map from $P$ to $\mathcal{L}_{F(P)}(F(Q))$
  given by $p\longmapsto S_p$, and by $\sigma_{\mathcal{F}}$
  the map from $R$ to $\mathcal{L}_{F(P)}(F(Q))$ given by $r\longmapsto
  \phi_\infty(r)$, and let $\mathcal{F}_{(P,Q,\psi)}$ be the subring
  of $\mathcal{L}_{F(P)}(F(Q))$ generated by
  $T_{\mathcal{F}}(Q)\cup S_{\mathcal{F}}(P)\cup
  \sigma_{\mathcal{F}}(R)$. Then
  $(S_{\mathcal{F}},T_{\mathcal{F}},\sigma_{\mathcal{F}},\mathcal{F}_{(P,Q,\psi)}))$
  is a surjective covariant
  representation of $(P,Q,\psi)$. This representation is injective if
  and only if $R$ is right non-degenerate.

  We call
  $(S_{\mathcal{F}},T_{\mathcal{F}},\sigma_{\mathcal{F}},\mathcal{F}_{(P,Q,\psi)}))$
  for the Fock space representation of $(P,Q,\psi)$.
\end{prop}

\begin{proof}
  It is clear that the maps $T_{\mathcal{F}}$,
  $S_{\mathcal{F}}$ and
  $\sigma_{\mathcal{F}}$
  are linear, and that for every $r\in R$, $p\in P$ and $q\in Q$
  we have that
  \begin{equation*}
    \phi_{\infty}(r)T_q=T_{rq}{,}\qquad T_q \phi_{\infty}(r)=T_{qr}{,}
    \qquad \phi_{\infty}(r)S_p=S_{rp}{,}\qquad S_p\phi_{\infty}(r)=S_{pr}{,}
  \end{equation*}
  from which it follows that $\sigma_{\mathcal{F}}$ is a
  ring homomorphism
  and that
  $$S_{\mathcal{F}}(pr)=S_{\mathcal{F}}(p)\sigma_{\mathcal{F}}(r)\,,\qquad S_{\mathcal{F}}(rp)=\sigma_{\mathcal{F}}(r)S_{\mathcal{F}}(p)\,,$$
  $$T_{\mathcal{F}}(rq)=\sigma_{\mathcal{F}}(r)T_{\mathcal{F}}(q)\,,\qquad T_{\mathcal{F}}(qr)=T_{\mathcal{F}}(q)\sigma_{\mathcal{F}}(r)$$
  for every $r\in R$, $p\in P$ and $q\in Q$.

  Given any $p\in P$ and $q\in Q$ we have for every $n\in
  \No$ that $S^{(n)}_pT^{(n)}_q (q_n)=\psi(p\otimes
  q)q_n$ for $q_n\in Q^{\otimes n}$, and hence the composition
  homomorphism $S_pT_q$ gives
  \begin{equation*}
    \begin{split}
      S_pT_q(\{q_n\})&=\left( \begin{array}{cccccc}
          0 &  S^{(0)}_p & 0 &  & & \\
          & 0 & S^{(1)}_p & 0 &  & \\
          &   & 0 & S^{(2)}_p & 0 &  \\
          & & & \ddots & \ddots & \ddots
        \end{array}\right)\left( \begin{array}{ccccc}
          0 &   &  &  & \\
          T^{(0)}_q& 0 &  &  &\\
          0& T^{(1)}_q & 0 &  & \\
          & 0& T^{(2)}_q& 0  & \\
          & & & \ddots  & \ddots
        \end{array}\right) \left(  \begin{array}{c} q_0 \\ q_1 \\ q_2 \\ q_3 \\ \vdots
        \end{array}\right)\\
      &=\left( \begin{array}{cccccc}
          0 &  S^{(0)}_p & 0 &  & & \\
          & 0 & S^{(1)}_p & 0 &  & \\
          &   & 0 & S^{(2)}_p & 0 &  \\
          & & & \ddots & \ddots & \ddots
        \end{array}\right)\left(  \begin{array}{c} 0 \\ qq_0 \\ q\otimes q_1 \\ q\otimes q_2 \\ \vdots
        \end{array}\right)\\
      &=\left(  \begin{array}{c} \psi(p\otimes q)q_0 \\ \psi(p\otimes q) q_1 \\ \psi(p\otimes q) q_2 \\ \vdots
        \end{array}\right)
      =
      \phi_\infty\bigl(\psi(p\otimes q)\bigr)(\{q_n\}){,}
    \end{split}
  \end{equation*}
  from which it follows that $\sigma_{\mathcal{F}}(\psi(p\otimes
  q))=S_{\mathcal{F}}(p)T_{\mathcal{F}}(q)$
  for every $p\in P$ and $q\in Q$.
  Thus
  $(S_{\mathcal{F}},T_{\mathcal{F}},\sigma_{\mathcal{F}},\mathcal{F}_{(P,Q,,\psi)})$
  is a surjective covariant representation of $(P,Q,\psi)$.

  Finally it is clear that
  $(S_{\mathcal{F}},T_{\mathcal{F}},\sigma_{\mathcal{F}},\mathcal{F}_{(P,Q,\psi)})$ is injective if
  and only if $R$ is right non-degenerate.
\end{proof}

\begin{nota}\label{opposite} 
  Let us denote by $B^{op}$ the opposite
  ring of $B$. Given $a,b\in B^{op}$ we write $a\cdot b$ for the
  product of $a$ and $b$ in $B^{op}$. Thus $a\cdot b=ba$.
\end{nota}

\begin{rema}
Let $R$ be a ring and let $(P,Q,\psi)$ be an $R$-system. We could
define an \emph{anti-representation} of $(P,Q,\psi)$ to be a
quadruple $(V,U,\eta,B^{op})$ where $B$ is a ring,
$\eta:R\longrightarrow B^{op}$ is an ring homomorphism,
$U:Q\longrightarrow B^{op}$ and $V:P\longrightarrow B^{op}$ are
linear maps, and $U(qr)=U(q)\cdot \eta(r)$, $U(rq)=\eta(r)\cdot
U(q)$, $V(rp)=\eta(r)\cdot V(p)$, $V(pr)=V(p)\cdot\eta(r)$ and
$V(p)\cdot U(q)=\eta(\psi(p\otimes q))$ for every $r\in R$, $q\in Q$
and $p\in P$. If we then denoted by $U_{\mathcal{F}^{\sharp}}$ the
map from $Q$ to $\mathcal{L}_{F(Q)}(F(P))$ given by $q\longmapsto
U_q$, by $V_{\mathcal{F}^{\sharp}}$ the map from $P$ to
$\mathcal{L}_{F(Q)}(F(P))$ given by $p\longmapsto V_p$, and by
$\eta_{\mathcal{F}^{\sharp}}$ the map from $R$ to
$\mathcal{L}_{F(Q)}(F(P))$ given by $r\longmapsto
\varphi_\infty(r)$, then
$(V_{\mathcal{F}^{\sharp}},U_{\mathcal{F}^{\sharp}},
\eta_{\mathcal{F}^{\sharp}},(\mathcal{F}_{(P,Q,\psi)}^{\sharp})^{op})$
would be an anti-representation of $(P,Q,\psi)$, where
$\mathcal{F}_{(P,Q,\psi)}^{\sharp}$ is the subring of
$\mathcal{L}_{F(Q)}(F(P))$ generated by
$U_{\mathcal{F}^{\sharp}}(Q)\cup
V_{\mathcal{F}^{\sharp}}(P)\cup\eta_{\mathcal{F}^{\sharp}}(R)$

Notice that in general the rings $\mathcal{F}_{(P,Q,\psi)}$ and
$\mathcal{F}_{(P,Q,\psi)}^{\sharp}$ are not isomorphic. For example
if $R$ is a right non-degenerate ring, but not a left non-degenerate
ring, then if we consider the $R$-system $(P,Q,\psi)$ where $P=Q=0$
and $\psi$ is the zero homomorphism, we have that
$\mathcal{F}_{(P,Q,\psi)}\cong R$ and
$\mathcal{F}_{(P,Q,\psi)}^{\sharp}\cong R/I$ where $I=\{r\in R:
Rr=0\}$.
\end{rema}

\section{Relative Cuntz-Pimsner rings} \label{sec:relat-cuntz-pimsn}

The Toeplitz representation of an $R$-system $(P,Q,\psi)$ is in general too
big to be an attractive representation of $(P,Q,\psi)$. We will in
this section study a certain subclass of covariant representations of
$(P,Q,\psi)$ and, for $R$-systems satisfying the condition
\textbf{(FS)} defined below, completely classify these representations up to
isomorphism in $\category$. We begin by describing this class of
representations.

Remember (cf. Theorem \ref{theor:toeplitz}) that
$\mathcal{T}_{(P,Q,\psi)}$ comes with a $\semigroup$-grading
$\oplus_{(m,n)}\mathcal{T}_{(m,n)}$ where $\semigroup$ is the
semigroup defined in Definition \ref{defi:semigroup}.
It will often be more convenient to work with a $\Z$-grading instead
of this $\semigroup$-grading.

\begin{prop} \label{prop:Z-grading}
  Let $R$ be a ring and let $(P,Q,\psi)$ be an $R$-system.
  If we for $k\in\Z$ let
  $\mathcal{T}_{(P,Q,\psi)}^{(k)}=\oplus_{(m,n)\in\semigroup,\
    m-n=k}\mathcal{T}_{(m,n)}$, then
  $\oplus_{n\in\Z}\mathcal{T}_{(P,Q,\psi)}^{(n)}$ is a $\Z$-grading of
  $\mathcal{T}_{(P,Q,\psi)}$. The grading
  $\oplus_{n\in\Z}\mathcal{T}_{(P,Q,\psi)}^{(n)}$ is the only
  $\Z$-grading
  $\oplus_{n\in\Z}\mathcal{Y}^{(n)}$ of
  $\mathcal{T}_{(P,Q,\psi)}$ for which $\iota_R(R)\subseteq\mathcal{Y}^{(0)}$,
  $\iota_Q(Q)\subseteq\mathcal{Y}^{(1)}$, and
  $\iota_P(P)\subseteq\mathcal{Y}^{(-1)}$.
\end{prop}

\begin{proof}
  It easily follows from Theorem \ref{theor:toeplitz} that
  $\oplus_{n\in\Z}\mathcal{T}_{(P,Q,\psi)}^{(n)}$ is a $\Z$-grading of
  $\mathcal{T}_{(P,Q,\psi)}$ and that
  $\iota_R(R)\subseteq\mathcal{T}_{(P,Q,\psi)}^{(0)}$,
  $\iota_Q(Q)\subseteq\mathcal{T}_{(P,Q,\psi)}^{(1)}$, and
  $\iota_P(P)\subseteq\mathcal{T}_{(P,Q,\psi)}^{(-1)}$.

  Suppose $\oplus_{n\in\Z}\mathcal{Y}^{(n)}$ is another $\Z$-grading of
  $\mathcal{T}_{(P,Q,\psi)}$ and that
  $\iota_R(R)\subseteq\mathcal{Y}^{(0)}$,
  $\iota_Q(Q)\subseteq\mathcal{Y}^{(1)}$, and
  $\iota_P(P)\subseteq\mathcal{Y}^{(-1)}$. Then
  $\mathcal{T}_{(P,Q,\psi)}^{(n)}\subseteq \mathcal{Y}^{(n)}$ for each
  $n\in\Z$ from which it follows that $\mathcal{T}_{(P,Q,\psi)}^{(n)}=
  \mathcal{Y}^{(n)}$ for each $n\in\Z$.
\end{proof}

\begin{prop} \label{prop:graded}
  Let $R$ be a ring, $(P,Q,\psi)$ an $R$-system, $(S,T,\sigma,B)$ a surjective
  covariant representation of $(P,Q,\psi)$, and let
  $\eta_{(S,T,\sigma,B)}:\mathcal{T}_{(P,Q,\psi)}\longrightarrow B$ be the ring
  homomorphism from Theorem \ref{theor:toeplitz}. If
  $\oplus_{n\in\Z}B^{(n)}$ is a $\Z$-grading of
  $B$ such that $\sigma(R)\subseteq
  B^{(0)}$, $T(Q)\subseteq B^{(1)}$ and $S(P)\subseteq B^{(-1)}$,
  then $\eta_{(S,T,\sigma,B)}(\mathcal{T}_{(P,Q,\psi)}^{(n)})=B^{(n)}$
  for every $n\in\Z$.
\end{prop}

\begin{proof}
  If $\oplus_{n\in\Z}B^{(n)}$ is a $\Z$-grading of
  $B$ such that $\sigma(R)\subseteq
  B^{(0)}$, $T(Q)\subseteq B^{(1)}$ and $S(P)\subseteq B^{(-1)}$, then
  $\eta_{(S,T,\sigma,B)}(\mathcal{T}_{(P,Q,\psi)}^{(n)})\subseteq B^{(n)}$
  for every $n\in\Z$. It follows that
  $\oplus_{n\in\Z}\eta_{(S,T,\sigma,B)}(\mathcal{T}_{(P,Q,\psi)}^{(n)})$
  is a $\Z$-grading of $B$, and
  thus that $\eta_{(S,T,\sigma,B)}(\mathcal{T}_{(P,Q,\psi)}^{(n)})=B^{(n)}$
  for every $n\in\Z$.
\end{proof}

\begin{defi} \label{defi:graded rep}
Let $R$ be a ring and $(P,Q,\psi)$ an $R$-system. A surjective
covariant representation $(S,T,\sigma,B)$ of $(P,Q,\psi)$ is
\emph{graded} if there exists a $\Z$-grading
$\oplus_{n\in\Z}B^{(n)}$ of $B$ such that $\sigma(R)\subseteq
B^{(0)}$, $T(Q)\subseteq B^{(1)}$, and $S(P)\subseteq B^{(-1)}$.
\end{defi}

The aim of this section is to classify all surjective, injective and
graded representations
of an $R$-system. Unfortunately, we do not know how to do that for
general $R$-systems, but only for $R$-systems satisfying a condition
we have chosen to call \textbf{(FS)} and which is defined below.
This condition is probably not the
optimal one, but many interesting examples do satisfy this condition.

\subsection{Condition \textbf{(FS)}}
\label{sec:condition-FS}

We will now introduce the condition \textbf{(FS)} and show some
fundamental results for $R$-systems satisfying this condition.

\begin{defi} \label{defi:FS}
Let $R$ be a ring. An $R$-system $(P,Q,\psi)$ is said to satisfy
condition \textbf{(FS)} if for every finite set $\{q_1,\ldots,
q_n\}\subseteq Q$ and $\{p_1,\ldots,p_m\}\subseteq P$  there exist
$\Theta\in \mathcal{F}_P(Q)$ and $\Delta\in \mathcal{F}_Q(P)$ such
that $\Theta(q_i)=q_i$ and $\Delta(p_j)=p_j$ for every $i=1,\ldots,
n$ and $j=1,\ldots, m$ respectively.
\end{defi}

\begin{exem}
Observe that condition \textbf{(FS)} appears in a natural
context. Let $Q$ be an $R$-bimodule such that $Q_R$ is a finitely
generated projective right $R$-module. Then define
$P:=Q^*=\text{Hom}_R(Q_R,R)$. We then have that $P$ is an $R$-bimodule
such that ${}_RP$ is a finitely generated projective left $R$-module
with $P^*=Q^{**}=Q$. Therefore we can define
\begin{eqnarray*}
  \psi:P\otimes_R Q & \longrightarrow R \\
f\otimes q & \longmapsto f(q).
\end{eqnarray*}

Observe that by the Dual Basis Lemma there exist $q_1,\ldots,q_n\in
Q$ and $f_1,\ldots,f_n\in P$ such that $\sum_{i=1}^nq_if_i(q)=q$ for
every $q\in Q$. Dually and since $P^*=Q$, there exist
$p_1,\ldots,p_m\in Q$ and $g_1,\ldots,g_m\in P^*=Q$ such that
$\sum_{j=1}^m g_j(p)p_j=p$ for every $p\in P$, from where condition
\textbf{(FS)} follows.
\end{exem}

\begin{defi}
Let $R$ be a ring.
An $R$-system $(P,Q,\psi)$ is \textit{non-degenerate} if whenever
$\psi(p\otimes q)=0$ for every $p\in P$ then $q=0$, and whenever
$\psi(p\otimes q)=0$ for every $q\in Q$ then $p=0$.
\end{defi}
\noindent Notice that if $(P,Q,\psi)$ is non-degenerate then every $T\in
\mathcal{L}_P(Q)$ has a unique adjoint.

\begin{lem}\label{rema_1}
Let $R$ be a ring and $(P,Q,\psi)$ an $R$-system satisfying condition
\textbf{(FS)}. Then $(P,Q,\psi)$ is non-degenerate.
\end{lem}

\begin{proof}
Let  $\psi(p\otimes q)=0$ for every $p\in P$. Then by condition
\textbf{(FS)} there exists
$\Theta=\sum^n_{i=1}\theta_{q_i,p_i}\in \mathcal{F}_P(Q)$ such
that
$q=\Theta(q)=\sum^n_{i=1}\theta_{q_i,p_i}(q)=\sum^n_{i=1}q_i\psi(p_i\otimes
q)=0$. Thus $(P,Q,\psi)$ is non-degenerate.
\end{proof}

Observe that if $R$ is right non-degenerate then $\psi_0:P^{\otimes
0}\otimes Q^{\otimes 0}\longrightarrow R$ is non-degenerate. For
general $n\in \mathbb{N}$ we need the condition $(\textbf{FS})$.

\begin{lem}\label{lemma_1}
Let $R$ be a ring and $(P,Q,\psi)$ an $R$-system satisfying condition
\textbf{(FS)}.
For every $n\in\mathbb{N}$ we have that the $R$-system
$(P^{\otimes n},Q^{\otimes n},\psi_n)$ satisfies condition \textbf{(FS)}.
\end{lem}

\begin{proof}
We will prove by induction that $\psi_n:P^{\otimes n}\otimes
Q^{\otimes n}\longrightarrow R$ satisfies condition \textbf{(FS)}
for every $n\in \mathbb{N}$. By hypothesis $(P,Q,\psi)$ satisfies
\textbf{(FS)}. Now suppose that $(P^{\otimes n-1},Q^{\otimes
n-1},\psi_{n-1})$ satisfies condition \textbf{(FS)}. Let
$q^1_1\otimes q^2_1,\ldots,q^1_m\otimes q^2_m\in Q^{\otimes n}$
where $q^1_1,\dots,q^1_m\in Q$ and $q^2_1,\dots,q^2_m\in Q^{\otimes
n-1}$. Since $(P,Q,\psi)$ satisfies condition \textbf{(FS)} there
exists $\Theta_1=\sum^l_{j=1}\theta_{a_j,b_j}\in \mathcal{F}_P(Q)$
with $a_j\in Q$ and $b_j\in P$ for every $j=1,\ldots,l$ such that
$\Theta_1(q^1_i)=q^1_i$ for every $i=1,\ldots, m$. Now since
$(P^{\otimes n-1},Q^{n-1},\psi_{n-1})$ satisfies condition
\textbf{(FS)}, by induction hypothesis, there exists
$\Theta_2=\sum^t_{k=1}\theta_{c_k,d_k}\in \mathcal{F}_{P^{\otimes
n-1}}(Q^{\otimes n-1})$ with $c_k\in Q^{\otimes n-1}$ and $d_k\in
P^{\otimes n-1}$ for every $k=1,\ldots,t$ such that
$\Theta_2(\psi(b_j\otimes q^1_i)q^2_i)=\psi(b_j\otimes q^1_i)q^2_i$
for every $i=1,\ldots, m$ and $j=1,\ldots, l$. Then define
$$\Theta=\sum^l_{j=1}\sum^t_{k=1}\theta_{a_j\otimes c_k, d_k\otimes
  b_j}\in\mathcal{F}_{P^{\otimes n}}(Q^{\otimes n})\,. $$
It is then straightforward to check that $\Theta(q^1_i\otimes
q^2_i)=q^1_i\otimes q^2_i$ for every $i=1,\ldots, m=$. Therefore
$(P^{\otimes n},Q^{\otimes n},\psi_{n})$ satisfies condition
\textbf{(FS)}.
\end{proof}

\begin{lem} \label{lemma:inj}
Let $R$ be a ring and let $(S,T,\sigma,B)$ be a covariant
representation of a $R$-system $(P,Q,\psi)$ satisfying condition
\textbf{(FS)}. If $\sigma$ is injective, then so are $T^n$ and $S^n$
for every $n\in \mathbb{N}$.
\end{lem}

\begin{proof}
Let $q\in Q^{\otimes n}$ such that $T^n(q)=0$. Then for every $p\in P^{\otimes n}$
we have that $0=S^n(p)T^n(q)=\sigma(\psi_n(p\otimes q))$,
and since $\sigma$ is injective, it follows that $\psi_n(p\otimes q)=0$ for
every $p\in P^{\otimes n}$, and it then follows from the
non-degeneracy of $\psi_n$ (cf. Lemma \ref{rema_1} and \ref{lemma_1})
that $q=0$.
Similarly one can check that $S^n$ is injective.
\end{proof}

\begin{defi}
  Let $R$ be a ring and $(P,Q,\psi)$ an $R$-system. We define the ring
  homomorphism
  $\Delta: R \longrightarrow \textrm{End}_R(Q_R)$ and the ring
  homomorphism
  $\Gamma: R \longrightarrow \textrm{End}_R({}_RP)^{op}$ by
  \begin{equation*}
    \Delta(r)(q)=rq\,,\qquad\Gamma(r)(p)=pr
  \end{equation*}
  for $r\in R$, $p\in P$ and $q\in Q$.
\end{defi}
Notice that for every $r\in R$ we have that $\Gamma(r)$ is the
adjoint of $\Delta(r)$, and thus that $\Delta(r)\in\mathcal{L}_P(Q)$
and $\Gamma(r)\in\mathcal{L}_Q(P)$.

\begin{prop}[{Cf. \cite[Lemma 2.2]{KPW} and \cite{PI}}] \label{lemma_2}
Let $R$ be a ring and $(P,Q,\psi)$ an $R$-system satisfying
condition \textbf{(FS)} and let $(S,T,\sigma,B)$ be a covariant
representation of $(P,Q,\psi)$. Then there exist a unique ring
homomorphism $\pi_{T,S}:\mathcal{F}_P(Q)\longrightarrow B$ such that
$\pi_{T,S}(\theta_{q,p})=T(q)S(p)$ for $p\in P$ and $q\in Q$, and a
unique ring homomorphism $\chi_{S,T}:\mathcal{F}_Q(P)\longrightarrow
B^{op}$ such that $\chi_{S,T}(\theta_{p,q})=S(p)\cdot T(q)$ for
$p\in P$ and $q\in Q$. These maps satisfy
\begin{align*}
  \pi_{T,S}(\Delta(r)\Theta)&=\sigma(r)\pi_{T,S}(\Theta)&
  \pi_{T,S}(\Theta\Delta(r))&=\pi_{T,S}(\Theta)\sigma(r)\\
  \chi_{S,T}(\Gamma(r)\Omega)&=\sigma(r)\cdot \chi_{S,T}(\Omega)&
  \chi_{S,T}(\Omega\Gamma(r))&=\chi_{S,T}(\Omega)\cdot \sigma(r)\\
  \pi_{T,S}(\Theta)T(q)&=T(\Theta(q))&S(p)\cdot\chi_{S,T}(\Omega)&=S(\Omega(p))
\end{align*}
for $r\in R$, $p\in P$, $q\in Q$, $\Omega\in \mathcal{F}_Q(P)$ and
$\Theta\in \mathcal{F}_P(Q)$. If $\Omega\in\mathcal{F}_Q(P)$ is the
adjoint of $\Theta\in\mathcal{F}_P(Q)$, then
$\pi_{T,S}(\Theta)=\chi_{S,T}(\Omega)$. Moreover
$\pi_{T,S}(\mathcal{F}_P(Q))=\chi_{S,T}(\mathcal{F}_Q(P))=\spa\{T(q)S(p):
q\in Q,\ p\in P\}$, and if $\sigma$ is injective, then $\pi_{T,S}$
and $\chi_{S,T}$ are injective too.
\end{prop}

\begin{proof}
Since $\mathcal{F}_P(Q)=\text{span}\{\theta_{q,p}:p\in P,\ q\in Q\}$,
there can at most be one ring homomorphism from $\mathcal{F}_P(Q)$ to
$B$ which for all $p\in P$ and $q\in Q$ sends
$\theta_{q,p}$ to $T(q)S(p)$.

Assume $p_1,p_2,\dots,p_n\in P$, $q_1,q_2,\dots,q_n\in Q$ and
$\sum_{i=1}^n \theta_{q_i,p_i}=0$. Then $\sum_{i=1}^n
q_i\psi(p_i\otimes z)=0$ for every $z\in Q$. By
condition \textbf{(FS)} there exists $\Theta=\sum_{j=1}^k
\theta_{e_j,f_j}\in \mathcal{F}_Q(P)$ such that
$$\Theta(p_i)=\sum_{j=1}^k \theta_{e_j,f_j}(p_i)=\sum_{j=1}^k
\psi(p_i\otimes f_j)e_j=p_i$$ for every $i=1,\ldots,n$. We then have
that
\begin{align*}
\sum_{i=1}^n T(q_i)S(p_i)
&=\sum_{i=1}^n T(q_i)S\bigl(\Theta(p_i)\bigr)
=\sum_{i=1}^n T(q_i)S\left(\sum_{j=1}^k \psi(p_i\otimes f_j)e_j\right)\\
&=\sum_{i=1}^n\sum_{j=1}^k T\bigl(q_i\psi(p_i\otimes f_j)\bigr)S(e_j)
=\sum_{j=1}^k T\left(\sum_{i=1}^n q_i\psi(p_i\otimes f_j)\right)S(e_j)=0,
\end{align*}
since $\sum_{i=1}^n q_i\psi(p_i\otimes f_j)=0$ for every
$j=1,\ldots, k$. Thus there exists a linear map
$\pi_{T,S}:\mathcal{F}_P(Q)\longrightarrow B$ which for $p\in P$ and
$q\in Q$ sends $\theta_{q,p}$ to $T(q)S(p)$.

Let $r\in R$, $p\in P$ and $q\in Q$. Then we have
\begin{equation*}
  \pi_{T,S}(\Delta(r)\theta_{q,p})=\pi_{T,S}(\theta_{rq,p})=T(rq)S(p)
  =\sigma(r)T(q)S(p)=\sigma(r)\pi_{T,S}(\theta_{q,p}),
\end{equation*}
from which it follows that
$\pi_{T,S}(\Delta(r)\Theta)=\sigma(r)\pi_{T,S}(\Theta)$ for every
$\Theta\in\mathcal{F}_P(Q)$. One can in a similar way show that
$\pi_{T,S}(\Theta\Delta(r))=\pi_{T,S}(\Theta)\sigma(r)$ for every
$\Theta\in\mathcal{F}_P(Q)$.

Let $p\in P$ and $q,q'\in Q$. Then we have
$$\pi_{T,S}(\theta_{q,p})T(q')=T(q)S(p)T(q')
=T(q)\sigma\bigl(\psi(p\otimes q')\bigr) =T\bigl(q\psi(p\otimes
q')\bigr)=T\bigl(\theta_{q,p}(q')\bigr)$$ from which it follows that
$\pi_{T,S}(\Theta)T(q')=T(\Theta(q'))$ for all
$\Theta\in\mathcal{F}_P(Q)$.

If $p\in P$, $q\in Q$ and $\Theta\in\mathcal{F}_P(Q)$, then we have
$$\pi_{T,S}(\Theta)\pi_{T,S}(\theta_{q,p})=\pi_{T,S}(\Theta)T(q)S(p)
=T\bigl(\Theta(q)\bigr)S(p)=\pi_{T,S}(\theta_{\Theta(q),p})
=\pi_{T,S}(\Theta\theta_{q,p})$$ from which it follows that
$\pi_{T,S}(\Theta)\pi_{T,S}(\Theta')=\pi_{T,S}(\Theta\Theta')$ for
all $\Theta'\in\mathcal{F}_P(Q)$. Thus $\pi_{T,S}$ is a ring
homomorphism.

Now suppose that $\sigma:R\longrightarrow B$ is injective and let
$\sum^{n}_{i=1}\theta_{q_i,p_i}\in \mathcal{F}_P(Q)$ with
$\pi_{T,S}(\sum^{n}_{i=1}\theta_{q_i,p_i})=\sum^{n}_{i=1}T(q_i)S(p_i)=0$.
Then for every $p\in P$ and $q\in Q$ we have that
$$0=S(p)\left(\sum^{n}_{i=1}T(q_i)S(p_i)\right)T(q)
=\sigma\left(\sum^{n}_{i=1}\psi(p\otimes q_i)\psi(p_i\otimes q)\right).$$
Since $\sigma$ is injective it follows that
$\sum^{n}_{i=1}\psi(p\otimes q_i)\psi(p_i\otimes q)=\psi(p\otimes
\sum^{n}_{i=1}q_i\psi(p_i\otimes q))=0$ for every $p\in P$ and
$q\in Q$. By Lemma \ref{rema_1} $\psi$ is non-degenerate, so it follows that
$\sum^{n}_{i=1}q_i\psi(p_i\otimes q)=0$ for every $q\in Q$. Thus
$\sum^{n}_{i=1}\theta_{q_i,p_i}=0$ which proves that $\pi_{T,S}$ is
injective.

The existence and uniqueness of $\chi_{S,T}$ and that $\chi_{S,T}$
is a ring homomorphism and has the properties
$\chi_{S,T}(\Gamma(r)\Omega)=\sigma(r)\cdot \chi_{S,T}(\Omega)$,
$\chi_{S,T}(\Omega\Gamma(r))=\chi_{S,T}(\Omega)\cdot\sigma(r)$ and
$S(p)\cdot\chi_{S,T}(\Omega)=S(\Omega(p))$ for $r\in R$, $p\in P$
and $\Omega\in\mathcal{F}_Q(P)$, and that $\chi_{S,T}$ is injective
if $\sigma$ is injective, can be proved in a similar way.

If $p\in P$ and $q\in Q$, then $\theta_{p,q}$ is the adjoint of
$\theta_{q,p}$ and $\pi_{T,S}(\theta_{q,p})=T(q)S(p)=S(p)\cdot
T(q)=\chi_{S,T}(\theta_{p,q})$. It follows that if
$\Omega\in\mathcal{F}_Q(P)$ is the adjoint of
$\Theta\in\mathcal{F}_P(Q)$, then
$\pi_{T,S}(\Theta)=\chi_{S,T}(\Omega)$.

Finally we see that
$\pi_{T,S}(\mathcal{F}_P(Q))=\text{span}\{T(q)S(p):p\in P,\ q\in Q\}
=\chi_{S,T}(\mathcal{F}_Q(P))$.
\end{proof}

\begin{nota}
To avoid too heavy notation, we will often when working with a given
$R$-system $(P,Q,\psi)$ satisfying condition \textbf{(FS)} let $\pi$
denote $\pi_{\iota_Q^n,\iota_P^n}$ and let $\chi$ denote
$\chi_{\iota_P^n,\iota_Q^n}$ for any $n\in\N$. We will then view $\pi$
as a map from 
$\bigcup_{n\in\N}\mathcal{F}_{P^{\otimes n}}(Q^{\otimes n})$ to
$\mathcal{T}_{(P,Q,\psi)}$ 
and $\chi$ as a map from
$\bigcup_{n\in\N}\mathcal{F}_{Q^{\otimes n}}(P^{\otimes n})$ to
$\mathcal{T}_{(P,Q,\psi)}^{op}$. 
\end{nota}

\begin{rema} \label{remark:comp}
  Let $R$ be a ring and $(P,Q,\psi)$ an $R$-system satisfying
  condition \textbf{(FS)}. If $(S_1,T_1,\sigma_1,B_1)$ and
  $(S_2,T_2,\sigma_2,B_2)$ are two covariant representations of
  $(P,Q,\psi)$ and $\phi:B_1\longrightarrow B_2$ is a ring homomorphism such that
  $\phi\circ T_1=T_2$, $\phi\circ S_1=S_2$ and
  $\phi\circ\sigma_1=\sigma_2$, then $\phi\circ\pi_{T_1,S_1}=\pi_{T_2,S_2}$
  and $\phi\circ\chi_{S_1,T_1}=\chi_{S_2,T_2}$.
\end{rema}

\subsection{Cuntz-Pimsner invariant representations}
\label{sec:cuntz-pimsn-invar}

As already mentioned, the aim of this section is to classify all
injective and graded representation of an $R$-system satisfying condition
\textbf{(FS)}. We will now for a given $R$-system $(P,Q,\psi)$
satisfying condition \textbf{(FS)} construct a family of surjective,
injective and
graded representation of $(P,Q,\psi)$. We will later show that up to
isomorphism this family of surjective, injective and
graded representation of $(P,Q,\psi)$ contains all surjective,
injective and
graded representation of $(P,Q,\psi)$.

\begin{defi}
  Let $R$ be a ring and let $(P,Q,\psi)$ be an $R$-system satisfying
  condition \textbf{(FS)}. We say that a two-sided ideal $J$ of $R$ is 
  \emph{$\psi$-compatible} if $J\subseteq\Delta^{-1}(\mathcal{F}_P(Q))$, 
  and we say that a $\psi$-compatible two-sided ideal $J$ of $R$ is \emph{faithful}
  if $J\cap\ker\Delta=\{0\}$.
\end{defi}

\begin{defi}
  Let $R$ be a ring and let $(P,Q,\psi)$ be an $R$-system satisfying
  condition \textbf{(FS)}.
  For a $\psi$-compatible two-sided ideal
  $J$ of $R$, we define
  $\mathcal{T}(J)$ to be the minimal two-sided ideal of
  $\mathcal{T}_{(P,Q,\psi)}$ that contains
  $\{\iota_R(x)-\pi(\Delta(x))\mid
  x\in J\}$.
\end{defi}

\begin{defi}[{Cf. \cite[Proposition 1.3]{FMR} and \cite[Proposition 2.18]{MS}}] \label{def:relativeCP}
Let $R$ be a ring, let $(P,Q,\psi)$ be an $R$-system
satisfying condition \textbf{(FS)}
and let $J$ be a $\psi$-compatible two-sided ideal of $R$.
We define the
\emph{Cuntz-Pimsner ring relative to the ideal $J$} to be the
quotient ring
$\mathcal{O}_{(P,Q,\psi)}(J):=\mathcal{T}_{(P,Q,\psi)}/\mathcal{T}(J)$.
We denote by $\rho_J$ the quotient map $\rho_J: \mathcal{T}_{(P,Q,\psi)}
\longrightarrow \mathcal{O}_{(P,Q,\psi)}(J)$.
\end{defi}

\begin{defi}[{Cf. \cite[Definition 1.1]{FMR}}]
Let $R$ be a ring, let $(P,Q,\psi)$ be an $R$-system satisfying
condition \textbf{(FS)} and let $J$ be a $\psi$-compatible two-sided ideal of $R$. 
A covariant representation $(S,T,\sigma,B)$ of $(P,Q,\psi)$ is said to be
\textit{Cuntz-Pimsner invariant representation relative to $J$} if
$\pi_{T,S}(\Delta(x))=\sigma(x)$ for every $x\in J$.
\end{defi}

The following theorem gives a complete characterization of
$\mathcal{O}_{(P,Q,\psi)}(J)$.

\begin{theor}[{Cf. \cite[Proposition 1.3]{FMR}}]\label{univ_cuntz}
Let $R$ be a ring, let $(P,Q,\psi)$ be an $R$-system satisfying
condition \textbf{(FS)} and let $J$ be a $\psi$-compatible two-sided ideal of $R$.
Let
$\iota_R^J:=\rho_J\circ\iota_R$, $\iota_Q^J:=\rho_J\circ\iota_Q$ and
$\iota_P^J:=\rho_J\circ\iota_P$. Then
$(\iota_P^J,\iota_Q^J,\iota_R^J,\mathcal{O}_{(P,Q,\psi)}(J))$ is a
surjective covariant representation of $(P,Q,\psi)$ which is
Cuntz-Pimsner invariant representation relative to $J$ with the
following property:
\begin{enumerate}[label=\textbf{(CP)}]
\item If $(S,T,\sigma,B)$ is a covariant representation of
  $(P,Q,\psi)$ which is Cuntz-Pimsner invariant
  relative to $J$, then there exists a unique
  ring homomorphism $$\eta_{(S,T,\sigma,B)}^J:\mathcal{O}_{(P,Q,\psi)}(J) \longrightarrow
  B$$ such that $\eta_{(S,T,\sigma,B)}^J\circ \iota^J_R=\sigma$,
  $\eta_{(S,T,\sigma,B)}^J\circ \iota^J_Q=T$ and
  $\eta_{(S,T,\sigma,B)}^J\circ \iota^J_P=S$. \label{item:10}
\end{enumerate}

The representation
$(\iota_P^J,\iota_Q^J,\iota_R^J,\mathcal{O}_{(P,Q,\psi)}(J))$ is
the, up to isomorphism in $\category$, unique surjective covariant
representation of $(P,Q,\psi)$ which is Cuntz-Pimsner invariant
representation relative to $J$ and which possesses the property
\ref{item:10}; in fact if $(S,T,\sigma,B)$ is a surjective covariant
representation of $(P,Q,\psi)$ which is Cuntz-Pimsner invariant
representation relative to $J$ and
$\phi:B\longrightarrow\mathcal{O}_{(P,Q,\psi)}(J)$ is a ring homomorphism such
that $\phi\circ\sigma=\iota_R^J$, $\phi\circ S=\iota_P^J$ and
$\phi\circ T=\iota_Q^J$, then $\phi$ is an isomorphism.

  We have moreover that the ring
  homomorphism $\iota_R^J$ is injective if and only if
  $J$ is faithful, and that the representation
  $(\iota_P^J,\iota_Q^J,\iota_R^J,\mathcal{O}_{(P,Q,\psi)}(J))$ is graded.

  We call $(\iota_P^J,\iota_Q^J,\iota_R^J,\mathcal{O}_{(P,Q,\psi)}(J))$
  \emph{the Cuntz-Pimsner representation of $(P,Q,\psi)$ relative to $J$}.
\end{theor}

\begin{rema}
  If we for a ring $R$, an $R$-system $(P,Q,\psi)$ satisfying
  condition \textbf{(FS)}, and a $\psi$-compatible two-sided ideal $J$ of $R$, let
  $\category^J$ be the subcategory of $\category$ consisting of all
  surjective covariant representation of $(P,Q,\psi)$ which are Cuntz-Pimsner
  invariant representation relative to $J$,
  then it follows from Theorem \ref{univ_cuntz} that
  $(\iota_P^J,\iota_Q^J,\iota_R^J,\mathcal{O}_{(P,Q,\psi)}^J)$ is an
  initial object in $\category^J$.
\end{rema}

To prove Theorem \ref{univ_cuntz} we need a definition, a lemma and a
proposition:

\begin{defi}
  Let $R$ be a ring, let $(P,Q,\psi)$ be an $R$-system and let
  $(S,T,\sigma,B)$ be a surjective and graded covariant representation of
  $(P,Q,\psi)$. It follows from Proposition \ref{prop:graded} and
  Definition \ref{defi:graded rep} that there is a unique $\Z$-grading
  $\oplus_{n\in\Z}B^{(n)}$ of
  $B$ such that $\sigma(R)\subseteq
  B^{(0)}$, $T(Q)\subseteq B^{(1)}$ and $S(P)\subseteq B^{(-1)}$.

  A two-sided ideal $I$ of
  $B$ is said to be \emph{graded} if
  $\oplus_{n\in \mathbb{Z}} I^{(n)}$ is a $\Z$-grading of $I$ where $I^{(n)}:=I\cap B^{(n)}$
  for each $n\in\Z$.
  It is not difficult to show that in this case
  $\oplus_{n\in\Z}\wp_I(B^{(n)})$ is a
  $\Z$-grading of the quotient ring $B/I$ where
  $\wp_I$ denotes the quotient map from $B$ to
  $B/I$ and that the covariant representation $(S_I,T_I,\sigma_I,B/I)$
  where $T_I:=\wp_I\circ T$, $S_I=\wp_I\circ S$ and $\sigma_I=\wp_I\circ\sigma$,
  is graded.
\end{defi}

For $(m,n)\in\semigroup$ let $\projection_{(m,n)}$ denote the projection
of $\mathcal{T}_{(P,Q,\psi)}$ onto $\mathcal{T}_{(m,n)}$ given by the
$\semigroup$-grading
$\oplus_{(k,l)\in\semigroup}\mathcal{T}_{(k,l)}$ (cf. Theorem
\ref{theor:toeplitz}).

\begin{lem}[{Cf. \cite[Lemma 2.20]{MS}}]\label{lemma_5}
Let $R$ be a ring, let be $(P,Q,\psi)$ an $R$-system satisfying
condition \textbf{(FS)} and let
$J$ be a $\psi$-compatible two-sided ideal of
$R$. For $n\in\N$ let
\begin{align*}
  &\begin{multlined}[c][.97\columnwidth]
    \mathcal{T}^{(n)}(J)=\spa\Bigl(\{\iota_Q^k(q)
    \bigl(\iota_R(x)-\pi(\Delta(x))\bigr)
    \iota_P^l(p): x\in J,\ q\in
    Q^{\otimes k},\ p\in P^{\otimes l},\\
    k,l\in\N\text{ with }k-l=n\}
    \cup \{\iota_Q^k(q)
    \bigl(\iota_R(x)-\pi(\Delta(x))\bigr) :x\in J,\
    q\in Q^{\otimes n}\}\Bigr)
  \end{multlined}
  \shortintertext{and}
  &\begin{multlined}[c][.97\columnwidth]
    \mathcal{T}^{(-n)}(J)=\spa\Bigl(\{\iota_Q^k(q)
    \bigl(\iota_R(x)-\pi(\Delta(x))\bigr)
    \iota_P^l(p): x\in J,\ q\in
    Q^{\otimes k},\ p\in P^{\otimes l},\\
    k,l\in\N\text{ with }l-k=n\}
    \cup \{ \bigl(\iota_R(x)-\pi(\Delta(x))\bigr)
    \iota_P(p): x\in J,\ p\in P^{\otimes n}\}\Bigr),
  \end{multlined}
  \shortintertext{and let}
  &\begin{multlined}[c][.97\columnwidth]
    \mathcal{T}^{(0)}(J)=\spa\Bigl(\{\iota_Q^k(q)
    \bigl(\iota_R(x)-\pi(\Delta(x))\bigr)
    \iota_P^k(p): x\in J,\ q\in
    Q^{\otimes k},\ p\in P^{\otimes k},\ k\in\N\}\\
    \cup\{\iota_R(x)-\pi(\Delta(x)):x\in J\}\bigr).
  \end{multlined}
\end{align*}
Then we have that
$\mathcal{T}^{(m)}(J)=\mathcal{T}_{(P,Q,\psi)}^{(m)}\cap
\mathcal{T}(J)$ for each $m\in\Z$, and that
$\oplus_{m\in\mathbb{Z}}\mathcal{T}^{(m)}(J)$ is a $\Z$-grading of
$\mathcal{T}(J)$. Thus $\mathcal{T}(J)$ is a
graded two-sided ideal of $\mathcal{T}_{(P,Q,\psi)}$.

We furthermore have that the following holds for every
$x\in\mathcal{T}(J)$:
\begin{enumerate}
\item $\mathcal{P}_{(0,0)}(x)\in \iota_R(J)$, \label{item:7}
\item there exists an $n\in\N$ such that $x\iota_Q^m(q)=0$ for every
  $m\ge n$ and every $q\in Q^{\otimes m}$. \label{item:8}
\end{enumerate}
\end{lem}

\begin{proof}
  It is clear that
  $\mathcal{T}^{(m)}(J)\subseteq\mathcal{T}_{(P,Q,\psi)}^{(m)}\cap
  \mathcal{T}(J)$ for each $m\in\Z$. It is also clear that
  $\oplus_{m\in\mathbb{Z}}(\mathcal{T}_{(P,Q,\psi)}^{(m)}\cap
  \mathcal{T}(J))\subseteq \mathcal{T}(J)$.

  If $x\in J$, $q\in Q$ and $p\in P$, then we have that
  \begin{equation}
    \bigl(\iota_R(x)-\pi(\Delta(x))\bigr)\iota_Q(q)=
    \iota_Q(xq)-\iota_Q(\Delta(x)q)=\iota_Q(xq)-\iota_Q(xq)=0, \label{eq:3}
  \end{equation}
  and that
  \begin{equation*}
    \begin{split}
      \iota_P(p)\bigl(\iota_R(x)-\pi(\Delta(x))\bigr)&=
      \iota_P(p)\bigl(\iota_R(x)-\chi(\Gamma(x))\bigr)\\
      &=\iota_P(px)-\iota_P(\Gamma(x)p)=\iota_P(px)-\iota_P(px)=0,
    \end{split}
  \end{equation*}
  from which it follows that
  $\oplus_{m\in\mathbb{Z}}\mathcal{T}^{(m)}(J)$ is a two-sided
  ideal of $\mathcal{T}_{(P,Q,\psi)}$. Since
  $\{\iota_R(x)-\pi(\Delta(x)): x\in J\}\subseteq
  \mathcal{T}^{(0)}(J)$, it follows that $\mathcal{T}(J)\subseteq
  \oplus_{m\in\mathbb{Z}}\mathcal{T}^{(m)}(J)$. Thus we have that
  \begin{equation} \label{eq:4}
    \oplus_{m\in\mathbb{Z}}\mathcal{T}^{(m)}(J)=\mathcal{T}(J)
  \end{equation}
  and  that $\mathcal{T}^{(m)}(J)=\mathcal{T}_{(P,Q,\psi)}^{(m)}\cap
  \mathcal{T}(J)$ for each $m\in\Z$.

  Let $x\in\mathcal{T}(J)$. That \eqref{item:7} holds directly follows
  from \eqref{eq:4}, and that \eqref{item:8}
  holds directly follows from \eqref{eq:3} and \eqref{eq:4}.
\end{proof}

\begin{prop}[{Cf. \cite[Proposition 2.21]{MS}}]\label{prop_1}
Let $R$ be a ring, let $(P,Q,\psi)$ be an
$R$-system satisfying condition \textbf{(FS)}
and let $J$ be a faithful $\psi$-compatible two-sided ideal of $R$. Then
the ring homomorphism $\rho:R\longrightarrow
\mathcal{T}_{(P,Q,\psi)}/\mathcal{T}(J)$ given by
$\rho(r)=\iota_R(r)+\mathcal{T}(J)$ is injective.
\end{prop}

\begin{proof}
  Assume that $r\in R$ and that $\iota_R(r)\in\mathcal{T}(J)$. It
  follows from Lemma \ref{lemma_5} that there exists an $n\in\N$ such
  that $\iota_R(r)\iota_Q^m(q)=0$ for every
  $m\ge n$ and every $q\in Q^{\otimes m}$. We will show that we can
  choose $n$ to be equal to $1$. We will do that by showing that if
  $n>1$ and $\iota_R(r)\iota_Q^n(q)=0$ for every $q\in Q^{\otimes n}$,
  then $\iota_R(r)\iota_Q^{n-1}(q)=0$ for every $q\in Q^{\otimes
    n-1}$.
  So assume that $n>1$ and $\iota_R(r)\iota_Q^n(q)=0$ for
  every $q\in Q^{\otimes n}$. Let  $q\in Q^{\otimes n-1}$, then we have that for every
  $q'\in Q$
  \begin{equation*}
    \iota_Q^n(rq\otimes q')=
    \iota_R(r)\iota_Q^n(q\otimes q')=0.
  \end{equation*}
  Since $\iota_Q^n$ is injective (cf. Lemma \ref{lemma:inj}), it
  follows that $rq\otimes q'=0$.
  Hence for every $p\in P^{\otimes n-1}$ and every $p'\in P$
  we have that
  \begin{equation*}
    \psi\bigl(p'\otimes \psi_{n-1}(p\otimes rq)q'\bigr)
    =\psi_n\bigl((p'\otimes p)\otimes(rq\otimes q')\bigr)=0.
  \end{equation*}
  The above holds for every $p'\in P$, so by Lemma
  \ref{rema_1} we have that
  \begin{equation*}
    \psi_{n-1}(p\otimes rq)q'=0.
  \end{equation*}
  Since the last
 equation holds for every $q'\in Q$, it follows that
  $\psi_{n-1}(p\otimes rq)\in\ker\Delta$ for every $p\in P^{\otimes
    n-1}$. We have that $\iota_P^{n-1}(p)\iota_R(r)\iota_Q^{n-1}(q)\in
  \mathcal{T}(J)$, so it follows from Lemma \ref{lemma_5} that
  \begin{equation*}
    \iota_R(\psi_{n-1}(p\otimes rq))
    =\mathcal{P}_0(\iota_P^{n-1}(p)\iota_R(r)\iota_Q^{n-1}(q))\in \iota_R(J).
  \end{equation*}
  Thus $\psi_{n-1}(p\otimes rq)\in J\cap\ker\Delta=\{0\}$ for all $p\in P^{n-1}$, so by Lemma \ref{rema_1} and
  \ref{lemma_1} we have that $rq=0$. Hence $\iota_R(r)\iota_Q^{n-1}(q)=0$.

  Thus $\iota_Q(\Delta(r)q)=\iota_R(r)\iota_Q(q)=0$ for
  every $q\in Q$. From the injectivity of $\iota_Q$ (cf. Lemma
  \ref{lemma:inj}) it follows that $r\in\ker\Delta$. Then by Lemma
  \ref{lemma_5} we have that $\iota_R(r)=\projection_{(0,0)}(\iota_R(r))\in
  \iota_R(J)$. Therefore $r\in J\cap\ker\Delta=\{0\}$, which shows that
  $r=0$ as desired.
\end{proof}

It follows from Lemma \ref{lemma_5} and Proposition \ref{prop_1} that
if $R$ is a ring, $(P,Q,\psi)$ is an
$R$-system satisfying condition \textbf{(FS)}
and $J$ is a faithful $\psi$-compatible two-sided ideal of $R$, then
$\mathcal{T}(J)$ is a graded two-sided ideal of
$\mathcal{T}_{(P,Q,\psi)}$ which satisfies that
$\iota_R(R)\cap\mathcal{T}(J)=\{0\}$. We will show (see Remark
\ref{rema:correspondence}) that every graded two-sided ideal $I$
of $\mathcal{T}_{(P,Q,\psi)}$ such that $\iota_R(R)\cap I=\{0\}$ is of this form.

\begin{proof}[Proof of Theorem \ref{univ_cuntz}]
  It is clear that
  $(\iota_P^J,\iota_Q^J,\iota_R^J,\mathcal{O}_{(P,Q,\psi)}^J)$ is a
  covariant representation of $(P,Q,\psi)$ which is Cuntz-Pimsner
  invariant representation relative to $J$, and that it possesses
  property \ref{item:10} follows from
  Theorem \ref{theor:toeplitz} and the definition of $\mathcal{T}(J)$
  and $(\iota_P^J,\iota_Q^J,\iota_R^J,\mathcal{O}_{(P,Q,\psi)}^J)$.

  If $(S,T,\sigma,B)$ is a
  surjective covariant representation of $(P,Q,\psi)$ which is
  Cuntz-Pimsner invariant representation relative to $J$ and
  $\phi:B\longrightarrow\mathcal{O}_{(P,Q,\psi)}(J)$ is a ring homomorphism such that
  $\phi\circ\sigma=\iota_R^J$, $\phi\circ S=\iota_P^J$ and $\phi\circ
  T=\iota_Q^J$, then
  $\eta_{(S,T,\sigma,B)}^J\circ\phi(\sigma(r))=\sigma(r)$ for all $r\in
  R$, $\eta_{(S,T,\sigma,B)}^J\circ\phi(S(p))=S(p)$ for all $p\in P$ and
  $\eta_{(S,T,\sigma,B)}^J\circ\phi(T(q))=T(q)$ for all $q\in Q$, and
  since $B$ is generated by $\sigma(R)\cup S(P)\cup T(Q)$, it follows
  that $\eta_{(S,T,\sigma,B)}^J\circ\phi$ is equal to the identity map
  of $B$. One can in a similar way show that
  $\phi\circ\eta_{(S,T,\sigma,B)}^J$ is equal to the identity map of
  $\mathcal{O}_{(P,Q,\psi)}(J)$. Thus $\phi$ and $\eta_{(S,T,\sigma,B)}^J$
  are each other inverse,
  and $\phi$ is an isomorphism.

  If $J$ is faithful, then it follows from Proposition
  \ref{prop_1} that $\iota_R^J$ is injective. If $x\in
  J\cap\ker\Delta$, then $\iota_R^J(x)=0$; so if
  $J$ is not faithful, then $\iota_R^J$ is not injective.

  It follows directly from Lemma \ref{lemma_5} that
  $(\iota_P^J,\iota_Q^J,\iota_R^J,\mathcal{O}_{(P,Q,\psi)}(J))$ is graded.
\end{proof}

\subsection{Injective and  graded covariant representations}
\label{sec:inject-grad-covar}

Let $R$ be a ring and $(P,Q,\psi)$ an $R$-system which satisfies
condition \textbf{(FS)}. We will, as mentioned at previously, show
that every surjective, injective and graded covariant representations of
$(P,Q,\psi)$ is isomorphic to
$(\iota_P^J,\iota_Q^J,\iota_R^J,\mathcal{O}_{(P,Q,\psi)}(J))$ for some
faithful $\psi$-compatible two-sided ideal $J$ of $R$.

\begin{defi} \label{def:J}
Let $R$ be a ring, $(P,Q,\psi)$ an $R$-system satisfying condition
\textbf{(FS)} and let $(S,T,\sigma,B)$ be a covariant
representation of $(P,Q,\psi)$. We define
$$J_{(S,T,\sigma,B)}:=\{r\in R:\sigma(r)\in\pi_{T,S}(\mathcal{F}_P(Q))\}.$$
\end{defi}

\begin{lem}[{Cf. \cite[Proposition 3.3]{KS1}}] \label{lemma:compact}
Let $R$ be a ring and let $(S,T,\sigma,B)$ be an injective covariant
representation of an $R$-system $(P,Q,\psi)$ that satisfies
condition \textbf{(FS)}. Then $r\in R$ is in $J_{(S,T,\sigma,B)}$
if and only if $r\in\Delta^{-1}(\mathcal{F}_P(Q))$ and
$\sigma(r)=\pi_{T,S}(\Delta(r))$.
\end{lem}

\begin{proof}
It is obvious that if $r\in\Delta^{-1}(\mathcal{F}_P(Q))$ and
$\sigma(r)=\pi_{T,S}(\Delta(r))$, then $r\in J_{(S,T,\sigma,B)}$.

If $\Theta\in\mathcal{F}_P(Q)$ and $\sigma(r)=\pi_{T,S}(\Theta)$, then
we have for every $q\in Q$ that
$$T(rq)=\sigma(r)T(q)=\pi_{T,S}(\Theta)T(q)=T(\Theta(q)),$$
and since $T$ is injective (cf. Lemma \ref{lemma:inj} and \ref{rema_1}),
it follows that $rq=\Theta(q)$. Hence $\Delta(r)=\Theta$.
\end{proof}

\begin{rema}
  Let $R$ be a ring, $(P,Q,\psi)$ an $R$-system
  satisfying condition \textbf{(FS)}, let $J$ be a $\psi$-compatible 
  two-sided ideal of $R$, and let
  $(S,T,\sigma,B)$ be an injective covariant representation of
  $(P,Q,\psi)$. Then it follows from Lemma \ref{lemma:compact} that
  $(S,T,\sigma,B)$ is Cuntz-Pimsner invariant with respect to $J$ if
  and only if $J\subseteq J_{(S,T,\sigma,B)}$.
\end{rema}

\begin{lem}\label{lemma_6}
Let $R$ be a ring, $(P,Q,\psi)$ an $R$-system satisfying condition
\textbf{(FS)} and let $(S,T,\sigma,B)$ be a covariant
representation of $(P,Q,\psi)$. Then $J_{(S,T,\sigma,B)}$ is a
$\psi$-compatible two-sided ideal of $R$. If $(S,T,\sigma,B)$ is injective, then
$J_{(S,T,\sigma,B)}$ is faithful.
\end{lem}

\begin{proof}
It easily follows from Proposition \ref{lemma_2} that
$J_{(S,T,\sigma,B)}$ is a two-sided ideal of $R$ and it is $\psi$-compatible by construction.
If $x\in J_{(S,T,\sigma,B)}\cap \ker\Delta$ and $(S,T,\sigma,B)$ is
injective, then it follows from Lemma \ref{lemma:compact} that
$\sigma(x)=\pi_{T,S}(\Delta(x))=0$, and since $\sigma$ is injective,
it follows that $x=0$. Thus $J_{(S,T,\sigma,B)}$ is faithful if $(S,T,\sigma,B)$ is injective.
\end{proof}

\begin{nota}
To avoid too heavy notation, we will often when working with a given
$R$-system $(P,Q,\psi)$ satisfying condition \textbf{(FS)} and a
faithful $\psi$-compatible
two-sided ideal $J$ of $R$,
let $\pi^J$ denote $\pi_{(\iota^J_Q)^n,(\iota^J_P)^n}$ for any
$n\in\N$. We will then view $\pi$ as a map from
$\bigcup_{n\in\N}\mathcal{F}_{P^{\otimes n}}(Q^{\otimes n})$ to $\mathcal{O}_{(P,Q,\psi)}(J)$.
\end{nota}

\begin{prop}\label{prop_3}
Let $R$ be a ring, let $(P,Q,\psi)$ be an $R$-system satisfying
condition \textbf{(FS)} and let $J$ be a faithful $\psi$-compatible 
two-sided ideal of $R$. Then $J=
J_{(\iota_P^J,\iota_Q^J,\iota_R^J,\mathcal{O}_{(P,Q,\psi)}(J))}$.
\end{prop}

\begin{proof}
If $x\in J$, then
$\iota_R^J(x)=\pi^J(\Delta(x))\in\pi^J(\mathcal{F}_P(Q))$, and so
$x\in
J_{(\iota_P^J,\iota_Q^J,\iota_R^J,\mathcal{O}_{(P,Q,\psi)}(J))}$.

If $x\in
J_{(\iota_P^J,\iota_Q^J,\iota_R^J,\mathcal{O}_{(P,Q,\psi)}(J))}$,
then it follows from Lemma \ref{lemma:compact} that
$x\in\Delta^{-1}(\mathcal{F}_P(Q))$ and
$\iota_R^J(x)=\pi^J(\Delta(x))$. So
$\iota_R(x)-\pi(\Delta(x))\in\mathcal{T}(J)$, and we then get from
Lemma \ref{lemma_5} that
$\iota_R(x)=\projection_{(0,0)}(\iota_R(x)-\pi(\Delta(x)))\in\iota_R(J)$,
and thus that $x\in J$.
\end{proof}

We are now ready to show that every surjective, injective and graded
covariant representation of an $R$-system $(P,Q,\psi)$ satisfying
condition \textbf{(FS)} is isomorphic to
$(\iota_P^J,\iota_Q^J,\iota_R^J,\mathcal{O}_{(P,Q,\psi)}(J))$ for
some faithful $\psi$-compatible two-sided ideal $J$ of $R$.

\begin{theor} \label{prop:psi}
  Let $R$ be a ring, let $(P,Q,\psi)$ be an $R$-system satisfying
  condition \textbf{(FS)}, let $J$ be a two-sided ideal of $R$ such
  that $J\subseteq \Delta^{-1}(\mathcal{F}_P(Q))$ and
  $J\cap\ker\Delta=\{0\}$, and let $(S,T,\sigma,B)$ be a covariant
  representation of $(P,Q,\psi)$. Then we have:
  \begin{enumerate}
  \item \label{item:19}
    If there exists a ring homomorphism
    $\eta:\mathcal{O}_{(P,Q,\psi)}(J)\longrightarrow
    B$ such that $\eta\circ\iota_Q^J=T$,
    $\eta\circ\iota_P^J=S$ and
    $\eta\circ\iota_R^J=\sigma$, then the representation $(S,T,\sigma,B)$
    is Cuntz-Pimsner invariant with respect to $J$.
  \item \label{item:20}
    If the representation $(S,T,\sigma,B)$ is Cuntz-Pimsner invariant with
    respect to $J$, then there exists a unique ring homomorphism
    $\eta_{(S,T,\sigma,B)}^J:\mathcal{O}_{(P,Q,\psi)}(J)\longrightarrow
    B$ such that $\eta_{(S,T,\sigma,B)}^J\circ\iota_Q^J=T$,
    $\eta_{(S,T,\sigma,B)}^J\circ\iota_P^J=S$ and
    $\eta_{(S,T,\sigma,B)}^J\circ\iota_R^J=\sigma$.
  \item \label{item:21}
    If the representation $(S,T,\sigma,B)$ is Cuntz-Pimsner invariant with
    respect to $J$, then there the ring homomorphism
    $\eta_{(S,T,\sigma,B)}$ is an isomorphism if and only if
    $(S,T,\sigma,B)$ is surjective, injective and graded and
    $J=J_{(S,T,\sigma,B)}$.
  \end{enumerate}
\end{theor}

For the proof of Theorem \ref{prop:psi} we need some lemmas, but
before we introduce them, let us notice that the promised classification of all
surjective, injective and graded covariant representations of a given $R$-system
$(P,Q,\psi)$ satisfying condition \textbf{(FS)} follows from Lemma
\ref{lemma_6} and Theorem \ref{prop:psi}.

\begin{rema} \label{remark:classi}
  Let $R$ be a ring and let $(P,Q,\psi)$ be an
  $R$-system satisfying condition \textbf{(FS)}. It follows from
  Lemma \ref{lemma_6} and Theorem
  \ref{prop:psi} that every surjective, injective and graded covariant
  representation
  of $(P,Q,\psi)$ is isomorphic to
  $(\iota_P^J,\iota_Q^J,\iota_R^J,\mathcal{O}_{(P,Q,\psi)}(J))$ for
  some faithful $\psi$-compatible two-sided ideal $J$ of $R$. 
  And it follows from Remark
  \ref{remark:comp} and Proposition
  \ref{prop_3} that if $J_1$ and $J_2$ are two faithful $\psi$-compatible 
  two-sided ideals of $R$, then there exists a
  ring homomorphism $\phi$ from $\mathcal{O}_{(P,Q,\psi)}(J_1)$ to
  $\mathcal{O}_{(P,Q,\psi)}(J_2)$ satisfying
  $\phi\circ\iota_Q^{J_1}=\iota_Q^{J_2}$,
  $\phi\circ\iota_P^{J_1}=\iota_P^{J_2}$ and
  $\phi\circ\iota_R^{J_1}=\iota_R^{J_2}$ if and only if $J_1\subseteq
  J_2$.
\end{rema}

We will now introduce and prove the lemmas which we will use in the
proof of Theorem \ref{prop:psi}.

\begin{lem}
  Let $R$ be a ring and $(P,Q,\psi)$ an $R$-system. Let $n\in\N$ and
  $T\in\mathcal{L}_{P^{\otimes n}}(Q^{\otimes n})$. Then there is a
  unique $T\otimes 1_Q\in\mathcal{L}_{P^{\otimes n+1}}(Q^{\otimes n+1})$
  such that $(T\otimes 1_Q)(q\otimes q')=T(q)\otimes q'$ for $q\in
  Q^{\otimes n}$ and $q'\in Q$.
\end{lem}

\begin{proof}
  It easily follows from the universal property of tensor products
  that their exists a unique map $T\otimes 1_Q:Q^{\otimes
    n+1}\longrightarrow Q^{\otimes n+1}$ which for all $q\in
  Q^{\otimes n}$ and $q'\in Q$ maps $q\otimes q'$ to $T(q)\otimes
  q'$. Likewise, if $S$ denote an adjoint of $T$, then there is a
  unique map $1_P\otimes S:P^{\otimes n+1}\longrightarrow P^{\otimes n+1}$ which
  for all $p\in P^{\otimes n}$ and $p'\in P$ maps $p'\otimes p$ to
  $p'\otimes S(p)$. We have
  \begin{align*}
    \psi_{n+1}\bigl((p'\otimes p)\otimes (T(q)\otimes q')\bigr)
    &=\psi\bigl(p'\psi_n(p\otimes T(q))\otimes q'\bigr)
    =\psi\bigl(p'\psi_n(S(p)\otimes q)\otimes q'\bigr)\\
    &=\psi_{n+1}\bigl((p'\otimes S(p))\otimes (q\otimes q')\bigr)
  \end{align*}
  for $p'\in P$, $p\in P^{\otimes n}$, $q'\in Q$ and $q\in Q^{\otimes
    n}$, from which it follows that $1_P\otimes S$ it an adjoint of
  $T\otimes 1_Q$ and thus that $T\otimes 1_Q\in
  \mathcal{L}_{P^{\otimes n+1}}(Q^{\otimes n+1})$ (and $1_P\otimes
  S\in\mathcal{L}_{Q^{n+1}}(P^{n+1})$).
\end{proof}

\noindent The following abuse of notation will be convenient in the following.

\begin{nota}
   Let $R$ be a ring and $(P,Q,\psi)$ an $R$-system. If $n=0$, then we
   will on occasions let $\mathcal{F}_{P^{\otimes n}}(Q^{\otimes n})$
   denote $R$, and we will for $T\in\mathcal{L}_{P^{\otimes
       n}}(Q^{\otimes n})$ use $T\otimes 1_Q$ to denote $\Delta(T)$.
\end{nota}

\begin{lem} \label{lemma:two-compact}
  Let $R$ be a ring, $(P,Q,\psi)$ an $R$-system satisfying condition
  \textbf{(FS)}, let $(S,T,\sigma,B)$ a
  covariant representation and let $n\in\No$. Then
  \begin{equation*}
    \pi_{T^{n+1},S^{n+1}}\bigl((\Theta_1\otimes 1_Q)\Theta_2\bigr)
    =\pi_{T^n,S^n}(\Theta_1)\pi_{T^{n+1},S^{n+1}}(\Theta_2)
  \end{equation*}
  for $\Theta_1\in\mathcal{F}_{P^{\otimes n}}(Q^{\otimes n})$ and
  $\Theta_2\in\mathcal{F}_{P^{\otimes n+1}}(Q^{\otimes n+1})$.
\end{lem}

\begin{proof}
  If $n=0$, then the result follows directly from Proposition
  \ref{lemma_2}. Assume that $n\in\N$.
  It is enough to prove the lemma in the case where
  $\Theta_2=\theta_{q\otimes q',p}$ and $q\in Q^{\otimes n}$, $q'\in Q$
  and $p\in P^{\otimes n+1}$. In that case $(\Theta_1\otimes
  1_Q)\theta_{q\otimes q',p}=\theta_{\Theta_1(q)\otimes q',p}$, so it
  follows from Proposition \ref{lemma_2} that
  \begin{align*}
    \pi_{T^{n+1},S^{n+1}}\bigl((\Theta_1\otimes 1_Q)\theta_{q\otimes
      q',p}\bigr)
    &=\pi_{T^{n+1},S^{n+1}}(\theta_{\Theta_1(q)\otimes q',p})
    =T(\Theta_1(q)\otimes q')S(p)\\
    &=T\bigl(\Theta_1(q)\bigr)T(q')S(p)
    =\pi_{T^n,S^n}(\Theta_1)T(q)T(q')S(p)\\
    &=\pi_{T^n,S^n}(\Theta_1)T(q\otimes q')S(p)
    =\pi_{T^n,S^n}(\Theta_1)\pi_{T^{n+1},S^{n+1}}(\theta_{\Theta_1(q)\otimes q',p}).
  \end{align*}
\end{proof}

\begin{lem} \label{lemma:graded}
  Let $R$ be a ring, let $(P,Q,\psi)$ be an $R$-system, let
  $(S,T,\sigma,B)$ be a surjective and graded
  covariant representation of $(P,Q,\psi)$ and let $H$ be a two-sided
  ideal of $B$. If $H$ is generated as a two-sided ideal of $B$ by
  $H\cap B^{(0)}$, then $H$ is graded. If $(P,Q,\psi)$ satisfies
  condition \textbf{(FS)} and $H$ is graded, then $H$ is
  generated as a two-sided ideal of $B$ by $H\cap B^{(0)}$.
\end{lem}

\begin{proof}
  For each $n\in\Z\setminus\{0\}$ let
  \begin{multline*}
    H^{(n)}=\spa\biggl(\bigcup_{m\in\Z}\{yxz\mid y\in B^{(m)},\  x\in
    H\cap B^{(0)},\ z\in B^{(n-m)}\}\\
    \cup\{xz\mid x\in
    H\cap B^{(0)},\ z\in B^{(n)}\}\cup\{yx\mid y\in B^{(n)},\  x\in
    H\cap B^{(0)}\}\biggr),
  \end{multline*}
  and let
  \begin{equation*}
    H^{(0)}=H\cap B^{(0)}.
  \end{equation*}
  Then $H^{(n)}\subseteq B^{(n)}$ for all $n\in\Z$, and it is not
  difficult to show that $\oplus_{n\in\Z}H^{(n)}$ is
  a graded two-sided ideal of
  $B$ which contains
  $H\cap B^{(0)}$, and that every two-sided
  of $B$ which contains
  $H\cap B^{(0)}$ also contains
  $\oplus_{n\in\Z}H^{(n)}$. So
  if $H$ is generated by $H\cap B^{(0)}$, then it is equal to
  $\oplus_{n\in\Z}H^{(n)}$ and thus graded.

  For the last assertion assume that $H$ is graded and that
  $(P,Q,\psi)$ satisfies condition \textbf{(FS)}. We will
  show that $H=\oplus_{n\in\Z}H^{(n)}$. Since $H$ is graded it
  is enough to show that if $n\in\Z$ and $x\in
  H\cap B^{(n)}$, then $x\in H^{(n)}$. If
  $n>0$ and $x\in
  H\cap B^{(n)}$, then there exists $q_0,q_1,q_2,\dots,q_k\in
  Q^{\otimes n}$ and $y_1,y_2,\dots,y_k\in B^{(0)}$ such that
  $x=T^n(q_0)+\sum_{i=1}^kT^n(q_i)y_i$. It follows from Lemma \ref{lemma_1}
  that there exist $q_1',q_2',\dots,q_l'\in Q^{\otimes n}$ and
  $p_1',p_2',\dots,p_l'\in P^{\otimes n}$ such that $\sum_{j=1}^lq_j'\psi_n(p_j'\otimes
  q_i)=q_i$ for $i\in\{0,1,2,\dots,k\}$. We then have that
  \begin{align*}
    \sum_{j=1}^lT^n(q_j')S^n(p_j')x
     & =\sum_{j=1}^lT^n(q_j')S^n(p_j')T^n(q_0)
    +\sum_{i=1}^k\sum_{j=1}^lT^n(q_j')S^n(p_j')T^n(q_i)y_i\\
    & =T^n(q_0)+\sum_{i=1}^kT^n(q_i)y_i
    =x\,,
  \end{align*}
  and that $S^n(p_j')x\in H^{(0)}$ for every $j\in\{1,2,\dots,l\}$,
  from which it follows that $x\in H^{(n)}$. One can in a similar way
  show that if $n<0$ and $x\in
  H\cap B^{(n)}$, then $x\in H^{(n)}$. Thus we have for all $n\in\Z$
  that if $x\in
  H\cap B^{(n)}$, then $x\in H^{(n)}$, from which it follows that
  $H=\oplus_{n\in\Z}H^{(n)}$.
\end{proof}

\begin{lem} \label{lemma:invariantideal}
  Let $R$ be a ring, $(P,Q,\psi)$ an $R$-system
  satisfying condition \textbf{(FS)} and let $H$ be a two-sided
  ideal of $\mathcal{T}_{(P,Q,\psi)}$. Then we have that
  \begin{equation*}
    J_{\uparrow H}:=\{r\in\Delta^{-1}(\mathcal{F}_P(Q))\mid
    \iota_R(r)-\pi(\Delta(r))\in H\}
  \end{equation*}
  is a $\psi$-compatible two-sided ideal of $R$ and 
  $\mathcal{T}(J_{\uparrow H})\subseteq H$. If in addition $H$ is graded and 
  $H\cap\iota_R(R)=\{0\}$, then
  $J_{\uparrow H}$ is faithful and $\mathcal{T}(J_{\uparrow H})=H$.
\end{lem}

\begin{proof}
  It directly follows from Proposition \ref{lemma_2} that $J_{\uparrow H}$ is a
  two-sided ideal of $R$, and it is $\psi$-compatible by construction. It follows directly 
  from the definition of $\mathcal{T}(J_{\uparrow H})$ that 
  $\mathcal{T}(J_{\uparrow H})\subseteq H$.

  Assume that $H$ is graded and $H\cap\iota_R(R)=\{0\}$.
  If $x\in J_{\uparrow H}\cap\ker\Delta$, then
  $\iota_R(x)=\iota_R(x)-\pi(\Delta(r))\in \mathcal{T}(J_{\uparrow H})\subseteq
  H$ and so $x=0$ proving that $J_{\uparrow H}\cap\ker\Delta=\{0\}$.

  We will then prove
  that $H\subseteq \mathcal{T}(J_{\uparrow H})$. It follows from Lemma
  \ref{lemma:graded} that it is enough to show that
  $H\cap\mathcal{T}_{(P,Q,\psi)}^{(0)}\subseteq \mathcal{T}(J_{\uparrow H})$.
  It follows from Theorem \ref{theor:toeplitz} and Proposition
  \ref{prop:Z-grading} and \ref{lemma_2} that
  $\mathcal{T}_{(P,Q,\psi)}^{(0)}=\oplus^\infty_{i=0}
  \pi(\mathcal{F}_{P^{\otimes i}}(Q^{\otimes i}))$ (where we let
  $\mathcal{F}_{P^{\otimes 0}}(Q^{\otimes 0})=R$ and
  $\pi:\mathcal{F}_{P^{\otimes 0}}(Q^{\otimes 0})\longrightarrow
  \mathcal{T}_{(P,Q,\psi)}=\iota_R$), so it is enough to prove that the
  following inclusion holds
  \begin{equation} \label{eq:1}
    H\cap\left(\bigoplus^n_{i=0}\pi\bigl(\mathcal{F}_{P^{\otimes
          i}}(Q^{\otimes i})\bigr)\right)\subseteq \mathcal{T}(J_{\uparrow H}),
  \end{equation}
  for every $n\in \mathbb{N}$.
  We will prove that \eqref{eq:1} holds
  by induction over $n$.

  First we notice that
  $H\cap(\pi(\mathcal{F}_{P^{\otimes 0}}(Q^{\otimes 0})))=H\cap
  \iota_R(R)=\{0\}\subseteq \mathcal{T}(J_{\uparrow H})$, proving that
  \eqref{eq:1} holds for $n=0$.

  Assume now that $n\in\No$ and that \eqref{eq:1} holds.
  Let $\Theta_i\in \mathcal{F}_{P^{\otimes i}}(Q^{\otimes i})$ for
  $i\in \{0,1,\ldots,n+1\}$ such that
  $\sum^{n+1}_{i=0}\pi(\Theta_i)\in H$. We want to prove that
  $\sum^{n+1}_{i=0}\pi(\Theta_i)\in\mathcal{T}(J_{\uparrow H})$.
  Let $\rho_H:\mathcal{T}_{(P,Q,\psi)}\longrightarrow \mathcal{T}_{(P,Q,\psi)}/H$
  denote the quotient map, and let $\sigma_H:=\rho_H\circ\iota_R$,
  $T_H:=\rho_H\circ\iota_Q$ and $S_H:=\rho_H\circ\iota_P$. Then
  $(S_H,T_H,\sigma_H,\mathcal{T}_{(P,Q,\psi)}/H)$ is an injective
  covariant representation of $(P,Q,\psi)$ and
  $\rho_H\circ\pi=\pi_{T_H,S_H}$.
  We then have that
  $\sum^{n+1}_{i=0}\pi_{T_H^i,S_H^i}(\Theta_i)
  =\rho_H(\sum^{n+1}_{i=0}\pi(\Theta_i))=0$.
  Choose $q_j\in
  Q^{\otimes n}$, $p_j\in P^{\otimes n}$, $q'_j\in Q$,
  $p'_j\in P$ for $j\in \{1,\ldots,m\}$ such that
  $\Theta_{n+1}=\sum^{m}_{j=1}\theta_{q_j\otimes q'_j, p'_j\otimes
    p_j}$
  and $a_h\in Q^{\otimes n}$, $b_h\in P^{\otimes n}$ for $h\in
  \{1,\ldots,l\}$ such
  that $\sum^l_{h=1}\theta_{a_h,b_h}(q_j)=q_j$ for every $j\in
  \{1,\ldots,m\}$. We then have that
  $\sum^l_{h=1}(\theta_{a_h,b_h}\otimes 1_Q)\Theta_{n+1}=\Theta_{n+1}$.
  Let $\Theta=
  \left(\sum^l_{h=1}\theta_{a_h,b_h}\right)
  \left(\sum^{n}_{i=0}\Theta_i\otimes
    1_{Q^{\otimes n-i}}\right)\in
  \mathcal{F}_{P^{\otimes n}}(Q^{\otimes n})$.
  It follows from Lemma \ref{lemma:two-compact} that we then have
  that
  \begin{align*}
    \pi_{T_H^n,S_H^n}(\Theta)
    &= \pi_{T_H^n,S_H^n}\left(\biggl(\sum^l_{h=1}\theta_{a_h,b_h}\biggr)
      \biggl(\sum^{n}_{i=0}\Theta_i\otimes
      1_{Q^{\otimes n-i}}\biggr)\right)
    = \pi_{T_H^n,S_H^n}\biggl(\sum^l_{h=1}\theta_{a_h,b_h}\biggr)
    \sum^{n}_{i=0}\pi_{T_H^i,S_H^i}(\Theta_i)\\
    &=-\pi_{T_H^n,S_H^n}\biggl(\sum^l_{h=1}\theta_{a_h,b_h}\biggr)
    \pi_{T_H^{n+1},S_H^{n+1}}(\Theta_{n+1})
    =-\pi_{T_H^{n+1},S_H^{n+1}}
    \biggl(\sum^l_{h=1}(\theta_{a_h,b_h}\otimes
    1_Q)\Theta_{n+1}\biggr) \\
    &=-\pi_{T_H^{n+1},S_H^{n+1}}(\Theta_{n+1}){,}
    \end{align*}
    so
    $\sum^{n}_{i=0}\pi_{T_H^i,S_H^i}(\Theta_i)-\pi_{T_H^n,S_H^n}(\Theta)
    =\sum^{n+1}_{i=0}\pi_{T_H^i,S_H^i}(\Theta_i)=0$, and therefore
    $\sum^{n}_{i=0}\pi(\Theta_i)-\pi(\Theta)\in H$. Thus it
    follows from the induction assumption that
    $\sum^{n}_{i=0}\pi(\Theta_i)-\pi(\Theta)\in\mathcal{T}(J_{\uparrow H})$.
    Therefore it is enough to prove that
    $\pi(\Theta)+\pi(\Theta_{n+1})\in\mathcal{T}(J_{\uparrow H})$.

    Choose $q_j\in Q^{\otimes n}$, $p_j\in P^{\otimes n}$ for $j\in
    \{1,\ldots,m\}$ such that $\Theta=\sum^m_{j=1}\theta_{q_j,p_j}$
    and $q'_h\in Q^{\otimes n}$, $p'_h\in P^{\otimes n}$, $q''_h\in
    Q$, $p''_h\in P$ for $h\in \{1,\ldots,l\}$ such that
    $\Theta_{n+1}=\sum^l_{h=1}\theta_{q'_h\otimes q''_h,p''_h\otimes
      p'_h}$. Now since $(P^{\otimes n},Q^{\otimes n},\psi_n)$ satisfies condition \textbf{(FS)} there exist $a_r\in Q^{\otimes n}$ and $b_r\in P^{\otimes
      n}$ for $r\in \{1,\ldots,s\}$ such that
    $\sum^s_{r=1}\theta_{a_r,b_r}(q_j)=q_j$ for all $j\in
    \{1,\ldots,m\}$, and $\sum^s_{r=1}\theta_{a_r,b_r}(q'_h)=q'_h$ for
    all $h\in \{1,\ldots,l\}$. There also exist $c_t\in P^{\otimes n}$ and  $d_t\in
    Q^{\otimes n}$ for $t\in \{1,\ldots,v\}$ such that
    $\sum^v_{t=1}\theta_{c_t,d_t}(p_j)=p_j$ for all $j\in
    \{1,\ldots,m\}$, and $\sum^v_{t=1}\theta_{c_t,d_t}(p'_h)=p'_h$ for
    all $h\in \{1,\ldots,l\}$.

    Then we have
    \begin{equation*}
      \sum^s_{r=1}\iota_Q^n(a_r)\iota_P^n(b_r)
      \bigl(\pi(\Theta)+\pi(\Theta_{n+1})\bigr)
      \sum^v_{t=1}\iota_Q^n(d_t)\iota_P^n(c_t)
      =\pi(\Theta)+\pi(\Theta_{n+1}){,}
    \end{equation*}
    so it is enough to prove that $\iota_P^n(b)
    (\pi(\Theta)+\pi(\Theta_{n+1}))\iota_Q^n(d)\in
    \mathcal{T}(J_{\uparrow H})$ for every $b\in P^{\otimes n}$ and $d\in
    Q^{\otimes n}$.
    Let $r=\psi_n(b\otimes \Theta(d))\in R$.
    We then have that
    $$\sigma_H(r)=S_H(b)\pi_{T_H,S_H}(\Theta)T_H(d)
    =-S_H(b)\pi_{T_H,S_H}(\Theta_{n+1})T_H(d)\in
    \pi_{T_H,S_H}(\mathcal{F}_P(Q))\,,$$ such it follows from Lemma
    \ref{lemma:compact} that $r\in\Delta^{-1}(\mathcal{F}_P(Q))$ and
    $\sigma_H(r)=\pi_{T_H,S_H}(\Delta(r))$. Hence
    $r\in J_{\uparrow H}$.
    Thus
    \begin{equation*}
      \iota_P^n(b)\bigl(\pi(\Theta)+\pi(\Theta_{n+1})\bigr)\iota_Q^n(d)
     =\iota_P^n(b)\pi(\Theta)\iota_Q^n(d)
     +\iota_P^n(b)\pi(\Theta_{n+1})\iota_Q^n(d)
     =\iota_R(r)-\pi(\Delta(r))\in
     \mathcal{T}(J_{\uparrow H}){.}
   \end{equation*}
\end{proof}

\begin{proof}[Proof of Theorem \ref{prop:psi}]
  \eqref{item:19}: If there exists a ring homomorphism
  $\eta:\mathcal{O}_{(P,Q,\psi)}(J)\longrightarrow
  B$ such that $\eta\circ\iota_Q^J=T$,
  $\eta\circ\iota_P^J=S$ and
  $\eta\circ\iota_R^J=\sigma$, and $x\in J$, then $\sigma(x)=\eta(\iota_R^J(x))
  =\eta(\pi^J(\Delta(x)))=\pi_{T,S}(\Delta(x))$, which proves that the representation
  $(S,T,\sigma,B)$ is Cuntz-Pimsner invariant with respect to $J$.

  \eqref{item:20}: If the representation $(S,T,\sigma,B)$ is Cuntz-Pimsner invariant with
    respect to $J$, then the existence and uniqueness of $\eta_{(S,T,\sigma,B)}^J$ follows
    from Proposition \ref{univ_cuntz}.

    \eqref{item:21}: Assume that $\eta_{(S,T,\sigma,B)}^J$ is an isomorphism. Then
    $\sigma=\eta_{(S,T,\sigma,B)}^J\circ\iota_R^J$ is injective, and
    $$\oplus_{n\in\Z}\eta_{(S,T,\sigma,B)}^J(\mathcal{O}_{(P,Q,\psi)}^{(n)}(J))$$
    is a $\Z$-grading of $B$ such that
    $$\sigma(R)\subseteq \eta_{(S,T,\sigma,B)}^J(\mathcal{O}_{(P,Q,\psi)}^{(0)}(J))\,,$$
    $$T(Q)\subseteq \eta_{(S,T,\sigma,B)}^J(\mathcal{O}_{(P,Q,\psi)}^{(1)}(J)) \text{ and}$$
    $$S(P)\subseteq
    \eta_{(S,T,\sigma,B)}^J(\mathcal{O}_{(P,Q,\psi)}^{(-1)}(J))\,.$$
    Hence $(S,T,\sigma,B)$ is injective, surjective and graded.
    If $x\in J$, then we have that
    \begin{equation*}
      \sigma(x)=\eta_{(S,T,\sigma,B)}^J(\iota_R^J(x))
      =\eta_{(S,T,\sigma,B)}^J(\pi^J(\Delta(x)))=\pi_{T,S}(\Delta(x)),
    \end{equation*}
    and thus $x\in J_{(S,T,\sigma,B)}$. If $x\in J_{(S,T,\sigma,B)}$,
    then it follows from Lemma \ref{lemma:compact} that
    $x\in\Delta^{-1}(\mathcal{F}_P(Q))$ and
    \begin{equation*}
      \eta_{(S,T,\sigma,B)}^J(\iota_R^J(x))=\sigma(x)
      =\pi_{T,S}(\Delta(x))=\eta_{(S,T,\sigma,B)}^J(\pi^J(\Delta(x))),
    \end{equation*}
    and since $\eta_{(S,T,\sigma,B)}^J$ is injective, it follows that
    $\iota_R^J(x)=\pi^J(\Delta(x))$. It follows that
    $\iota_R(x)-\pi(\Delta(x))\in\mathcal{T}(J)$, and we then get from
    Lemma \ref{lemma_5} that
    $\iota_R(x)=\projection_{(0,0)}(\iota_R(x)-\pi(\Delta(x)))\in\iota_R(J)$,
    and thus that $x\in J$. Hence $J=J_{(S,T,\sigma,B)}$.

    Assume then that $(S,T,\sigma,B)$ is surjective, injective and
    graded and that $J=J_{(S,T,\sigma,B)}$. Then
    $\eta_{(S,T,\sigma,B)}^J$ is surjective. Let
    $\eta_{(S,T,\sigma,B)}:\mathcal{T}_{(P,Q,\psi)}\longrightarrow B$ be as in
    Theorem \ref{theor:toeplitz}. Then
    $\eta_{(S,T,\sigma,B)}=\eta_{(S,T,\sigma,B)}^J\circ \rho_J$, so
    $\eta_{(S,T,\sigma,B)}^J$ is injective if $\ker
    \eta_{(S,T,\sigma,B)}=\ker\rho_J=\mathcal{T}(J)$. Let
    $H=\ker\eta_{(S,T,\sigma,B)}$. Then $H$ is a graded two-sided ideal
    of $\mathcal{T}_{(P,Q,\psi)}$ and $H\cap\iota_R(R)=\{0\}$, so it
    follows from Lemma \ref{lemma:invariantideal} that
    $\mathcal{T}(J_{\uparrow H})=H$. It easily follows from Lemma
    \ref{lemma:compact} that $J=J_{(S,T,\sigma,B)}=J_{\uparrow H}$, so
    we have that $\ker \eta_{(S,T,\sigma,B)}=H
    =\mathcal{T}(J_{\uparrow H})=\mathcal{T}(J)$ as desired.
\end{proof}

\section{The Graded Uniqueness Theorem}

We will in this section look at some consequences of the
classification of the surjective, injective and graded
representations of an $R$-system $(P,Q,\psi)$ satisfying condition
\textbf{(FS)}. We begin by noticing that we get a description of all
graded two-sided ideal $H$ of $\mathcal{T}_{(P,Q,\psi)}$ satisfying 
$\iota_R(R)\cap H=\{0\}$, and then
that the Fock space representation of $(P,Q,\psi)$ is isomorphic to
the Toeplitz representation if $R$ is right non-degenerate and
$(P,Q,\psi)$ satisfies condition \textbf{(FS)}.
Finally we will characterize the faithful $\psi$-compatible
two-sided ideals $J$ of $R$ 
for which $\mathcal{O}_{(P,Q,\psi)}(J)$ satisfies the
\emph{Graded Uniqueness Theorem}, cf. \cite[Theorem 4.8]{TF}.

\begin{rema} \label{rema:correspondence}
Let $R$ be a right non-degenerate ring and let $(P,Q,\psi)$ be an
$R$-system sa\-tis\-fy\-ing condition \textbf{(FS)}. It easily follows
from Lemma \ref{lemma:compact} that if $H$ is a graded two-sided of
$\mathcal{T}_{(P,Q,\psi)}$ satisfying $\iota_R(R)\cap H=\{0\}$, then
$J_{\uparrow H}=J_{(\iota_P^{J_{\uparrow H}},\iota_Q^{J_{\uparrow
H}},\iota_R^{J_{\uparrow H}},\mathcal{O}_{(P,Q,\psi)}(J_{\uparrow
H}))}$. Thus it follows from Proposition \ref{prop_3} and Lemma
\ref{lemma:invariantideal} that
\begin{equation*}
  H\longmapsto J_{\uparrow H}\qquad J\longmapsto \mathcal{T}(J)
\end{equation*}
is an order preserving bijective correspondence between the set of
graded two-sided ideal $H$ of $\mathcal{T}_{(P,Q,\psi)}$ satisfying $\iota_R(R)\cap
H=\{0\}$, and the set of faithful $\psi$-compatible two-sided ideals $J$ of $R$.

We will later (cf. Corollary \ref {ideals_toeplitz}) classify all graded
two-sided ideals of $\mathcal{T}_{(P,Q,\psi)}$.
\end{rema}

We will now show that the Fock space representation of an $R$-system
$(P,Q,\psi)$ is isomorphic to the Toeplitz representation if $R$ is
right non-degenerate and $(P,Q,\psi)$ satisfies condition
\textbf{(FS)}.

\begin{prop} \label{prop:fock}
  Let $R$ be a right non-degenerate ring and let $(P,Q,\psi)$ be an
  $R$-system satisfying condition \textbf{(FS)}. Then the Fock space
  representation
  $(S_{\mathcal{F}},T_{\mathcal{F}},\sigma_{\mathcal{F}},\mathcal{F}_{(P,Q,\psi)})$
  of $(P,Q,\psi)$ is isomorphic to the Toeplitz representation
  $(\iota_P,\iota_Q,\iota_R,\mathcal{T}_{(P,Q,\psi)})$.
\end{prop}

\begin{proof}
  To ease the notation let $T=T_{\mathcal{F}}$,
  $S=S_{\mathcal{F}}$, $\sigma=\sigma_{\mathcal{F}}$
  and $B=\mathcal{F}_{(P,Q,\psi)}$.
  It follows from Theorem \ref{theor:toeplitz} that there exists a
  unique ring homomorphism
  $\eta_{(S,T,\sigma,B)}: \mathcal{T}_{(P,Q,\psi)}\longrightarrow B$
  such that $\eta_{(S,T,\sigma,B)}\circ\iota_R=\sigma$,
  $\eta_{(S,T,\sigma,B)}\circ\iota_Q=T$ and
  $\eta_{(S,T,\sigma,B)}\circ\iota_P=S$.

  For each $m\in\No$ let $\iota_m$ denote the inclusion of $Q^{\otimes
    m}$ into $F(Q)$. It is easy to check that if
  $x\in\mathcal{T}_{(P,Q,\psi)}^{(n)}$ where $n\ge -m$, then
  $\eta_{(S,T,\sigma,B)}(x)\iota_m(Q^{\otimes
    m})\subseteq Q^{\otimes n+m}$. It follows that
  $(S,T,\sigma,B)$ is graded. It follows from the right non-degeneracy
  of $R$ that the covariant representation $(S,T,\sigma,B)$ is injective.

  Let $q\in Q$ and $p\in P$. Then
  $\pi_{T,S}(\theta_{q,p})=T(q)S(p)$
  acts as the zero map on $\iota_0(R)$. Thus it follows that if
  $\Theta\in\mathcal{F}_P(Q)$, then
  $\pi_{T,S}(\Theta)$
  acts as the zero map on $\iota_0(R)$. If $r\in R$, then it follows
  from the right non-degeneracy
  of $R$ that if $\sigma(r)=\phi_\infty(r)$
  acts as the zero map on $\iota_0(R)$,
  then $r=0$. Thus
  $J_{(S,T,\sigma,0)}=0$, and it follows from Theorem \ref{prop:psi} that
  $\eta_{(S,T,\sigma,B)}$
  is an isomorphism from $\mathcal{T}_{(P,Q,\psi)}$ to $\mathcal{F}_{(P,Q,\psi)}$.
\end{proof}

\begin{rema}
  Let $R$ be a ring and $(P,Q,\psi)$ an $R$-system. It is clear that
  it is a necessary condition for the the Fock space
  representation
  of $(P,Q,\psi)$ to be isomorphic to the Toeplitz representation is that
  $R$ is right non-degenerate. The following example shows that it is
  not in general sufficient. This is in contrast to the
  $C^*$-algebraic case where the Fock representation is always
  isomorphic to the universal Toeplitz representation,
  cf. \cite[Proposition 6.5]{KS1}
\end{rema}

\begin{exem}
  Let $R=Q=P=\Z$, let $R$ act on the left and the right on $Q$ and $P$
  by multiplication, and let $\psi:P\otimes Q\longrightarrow R$ be the zero
  map. Then $R$ is a non-degenerate ring, and $(P,Q,\psi)$ is an
  $R$-system. It is easy to check that $S_{\mathcal{F}}$
  is the zero map.

  Let $B=\oplus_{n\in\Z}\Z$, and for each $n\in\Z$ let $e_n$ be
  the element of $B$ given by $e_n(m)$ is 1 if and only if $n=m$ and
  0 otherwise. We turn $B$ into a ring by using the usual addition
  and defining a multiplication by
  \begin{equation*}
    e_me_n=
    \begin{cases}
      e_{m+n}&\text{if }nm\ge 0,\\
      0&\text{if }nm<0.
    \end{cases}
  \end{equation*}
  We define maps $\sigma:R\longrightarrow B$ by
  $\sigma(r)=re_0$, $S:P\longrightarrow B$ by
  $S(p)=pe_{-1}$ and $T:Q\longrightarrow B$ by $T(q)=qe_1$. It is easy to check
  that $(S,T,\sigma,B)$ is a covariant representation of
  $(P,Q,\psi)$. Since $S\ne 0$, it follows that $\iota_P\ne 0$ (in
  fact, it is not difficult to show that $(S,T,\sigma,B)$ is
  isomorphic to the Toeplitz representation of $(P,Q,\psi)$). Thus,
  the Fock space representation cannot be isomorphic to the Toeplitz
  representation in this example.
\end{exem}

We now define what it means for a relative Cuntz-Pimsner ring of an
$R$-system to satisfy the \emph{Graded Uniqueness Theorem}, and then
characterize when it does that.

\begin{defi}[{cf. \cite[Theorem 4.8]{TF}}] \label{def:gun}
Let $R$ be a ring, $(P,Q,\psi)$ an $R$-system
satisfying condition
\textbf{(FS)} and let $J$ be a faithful $\psi$-compatible two-sided ideal of $R$.
We say that the relative Cuntz-Pimsner ring $\mathcal{O}_{(P,Q,\psi)}(J)$ satisfies the \emph{Graded Uniqueness Theorem} if and only if the following holds:

If $B$ is a $\Z$-graded ring and
$\eta:\mathcal{O}_{(P,Q,\psi)}(J)\longrightarrow B$ is a graded ring homomorphism
such that $\eta\circ\iota_R^J$ is injective, then
$\eta$ is injective.
\end{defi}

\begin{defi}
  Let $R$ be a ring, and let $(P,Q,\psi)$ be an $R$-system satisfying condition
  \textbf{(FS)}. A faithful $\psi$-compatible two-sided ideal $J$ of $R$ is called 
  \emph{maximal} if $J=J'$ for any faithful $\psi$-compatible two-sided ideal $J'$ 
  of $R$ satisfying $J\subseteq J'$.
\end{defi}

\begin{theor}\label{uniqueness}
Let $R$ be a ring and let $(P,Q,\psi)$ be an
$R$-system satisfying condition \textbf{(FS)}. Let
$\category^{inj,grad}$ be the subcategory of $\category$ consisting of
all surjective, injective and graded covariant representation of $(P,Q,\psi)$.
Let $J$ be a two-sided ideal of $R$ such that $J\subseteq
\Delta^{-1}(\mathcal{F}_P(Q))$ and $J\cap \ker\Delta=0$. Then the
following three statements are equivalent:
\begin{enumerate}
\item The Cuntz-Pimsner ring $\mathcal{O}_{(P,Q,\psi)}(J)$ of
  $(P,Q,\psi)$ relative to $J$ satisfies the Graded Uniqueness
  Theorem. \label{item:9}
\item The Cuntz-Pimsner representation
  $(\iota_P^J,\iota_Q^J,\iota_R^J,\mathcal{O}_{(P,Q,\psi)}(J))$ of
  $(P,Q,\psi)$ relative to $J$ is minimal in $\category^{inj,grad}$ in
  the sense that if $(S,T,\sigma,B)$ is a surjective, injective
  and graded representation of $(P,Q,\psi)$ and
  $\eta:\mathcal{O}_{(P,Q,\psi)}(J)\longrightarrow B$ is a
  homomorphism such that $\eta\circ\iota_Q^J=T$, $\eta\circ\iota_P^J=S$ and
  $\eta\circ\iota_R^J=\sigma$, then $\eta$ is an isomorphism. \label{item:11}
\item $J$ is maximal. \label{item:14}
\end{enumerate}
\end{theor}

\begin{proof}
  If $B$ is a $\Z$-graded ring and
  $\eta:\mathcal{O}_{(P,Q,\psi)}(J)\longrightarrow B$ is a graded ring
  homomorphism
  such that $\eta\circ\iota_R^J$ is injective, and we let
  $T=\eta\circ\iota_Q^J$, $S=\eta\circ\iota_P^J$ and
  $\sigma=\eta\circ\iota_R^J$, then $(S,T,\sigma,B)$ is a
  surjective, injective and graded representation of $(P,Q,\psi)$. The
  equivalence of \eqref{item:9} and \eqref{item:11} easily follows
  from this.

  The equivalence of \eqref{item:11} and \eqref{item:14} follows from
  Remark \ref{remark:classi}.
\end{proof}

\begin{defi}
  Let $R$ be a ring, and let $(P,Q,\psi)$ be an $R$-system satisfying condition
  \textbf{(FS)}. A faithful $\psi$-compatible two-sided ideal $J$ of $R$ is called 
  \emph{uniquely maximal} if $J'\subseteq J$ for any $\psi$-compatible 
  two-sided ideal $J'$ of $R$.
\end{defi}

\begin{rema}
  Let $R$ be a ring, and let $(P,Q,\psi)$ be an $R$-system satisfying condition
  \textbf{(FS)}. It is clear that if $J$ is a uniquely maximal faithful $\psi$-compatible 
  two-sided of $R$, then it is the only maximal faithful $\psi$-compatible 
  two-sided of $R$. The standard argument using Zorn's Lemma shows that every faithful
  $\psi$-compatible two-sided ideal of $R$ is contained in a maximal faithful
  $\psi$-compatible two-sided ideal of $R$. Thus if there only is one 
  maximal faithful $\psi$-compatible 
  two-sided of $R$, then this ideal is automatically uniquely maximal.
\end{rema}

\begin{rema} \label{remark:CP}
Let $R$ be a ring and let $(P,Q,\psi)$ be an
$R$-system satisfying condition \textbf{(FS)}.
It follows from Remark \ref{remark:classi}
that if $J$ is a faithful $\psi$-compatible two-sided ideal of $R$, then
$(\iota^J_P,\iota^J_Q,\iota^J_R,\mathcal{O}_{(P,Q,\psi)}(J))$ is a
final object of $\category^{inj,grad}$ if and only if $J$ is uniquely maximal. 
If such a $J$ exists, then it would be natural to
define the Cuntz-Pimsner ring of the $R$-system $(P,Q,\psi)$ to be
$\mathcal{O}_{(P,Q,\psi)}(J)$ (and we will do that in Definition
\ref{def:CK}), however, as the following example shows, such a $J$
does not in general exist (in contrast to the $C^*$-algebraic case
where one always can use the analog of the ideal
$(\ker\Delta)^\perp\cap \Delta^{-1}(\mathcal{F}_P(Q))$ cf.
\cite{KS1}).
\end{rema}

\begin{exem}\label{example_two_maximal}
  Let $R=\mathbb{Z}\times \mathbb{R}\times \mathbb{Z}$ be a ring with
  multiplication defined by
  $$(x,y,z)\cdot(x',y',z'):=(xx',xy'+yx',xz'+zx')\,.$$
  Notice that $R$ is a unital ring with unit $(1,0,0)$.

  Let $\delta:R\longrightarrow R$ be a map defined as
  $\delta(x,y,z)=(x,y-z,0)$. We claim that $\delta$ is a ring
  homomorphism. Indeed, let $(x,y,z),(x',y',z')\in R$. Then we have
  \begin{align*}
    \delta(x,y,z)\delta(x',y',z')
    &=(x,y-z,0)(x',y'-z',0)
    =(xx',x(y'-z')+x'(y-z),0)\\
    &=(xx',xy'+yx'-(xz'+zx'),0)
    =\delta(xx',xy'+y'x,xz'+zx')\\
    &=\delta((x,y,z)(x',y',z')).
  \end{align*}

  Let $P=Q=\{(x,y,0):x\in \mathbb{Z},y\in\mathbb{R}\}\subseteq R$, and
  endow $P=Q$ with the following $R$-bimodule structure: Given $p\in P$, $q\in
  Q$ and $r\in R$ let
  \begin{align*}
    p\cdot r&=p\delta(r)&r\cdot p&=\delta(r)p\\
    q\cdot r&=q\delta(r)&r\cdot q&=\delta(r)q.
  \end{align*}

  Finally let $\psi:P\otimes_R Q\longrightarrow R$ be defined by
  $\psi(p\otimes q)=pq$. We will now check that the $R$-system $(P,Q,\psi)$
  satisfies property \textbf{(FS)}. Indeed, if $q\in Q$ then
  $$(1,0,0)\cdot\psi((1,0,0)\otimes q)=(1,0,0)\cdot q=q\,,$$
  and if $p\in P$ then

  $$\psi(p\otimes (1,0,0))\cdot (1,0,0)=p\cdot (1,0,0)=p\,.$$

  It easy to check that
  $$\Delta^{-1}(\mathcal{F}_{P}(Q)=R\qquad\text{and}\qquad \ker \Delta=\{(0,z,z):z\in \mathbb{Z}\}\,.$$

  Now we define
  $$J_1:=\{(0,y,0):y\in \mathbb{R}\}\qquad \text{and}\qquad J_2:=\{(0,0,z):z\in \mathbb{Z}\}\,.$$

  Now we will prove that both $J_1$ and $J_2$ are maximal faithful $\psi$-compatible 
  two-sided ideals of $R$. Let $J$ be a faithful $\psi$-compatible 
  two-sided ideals of $R$
  such that $J_1\subseteq J$ and assume
  that there exists $0\neq(x,y,z)\in J\setminus J_1$. Then $(x,0,z)\in
  J$, with either $x$ or $z$ are nonzero. If $x=0$, then $z\neq 0$ and
  then $(0,z,z)\in J\cap \ker\Delta$, but if $x\neq 0$ then
  $(0,0,1)(x,0,z)=(0,0,x)\in J$ and hence $0\neq (0,x,x)\in J\cap \ker
  \Delta$, a contradiction. Thus $J_1$ is maximal. We can do the same
  to prove that $J_2$ is also maximal.

  Notice that $J_1$ and $J_2$ are clearly non-isomorphic, however we can
  not deduce from this that their associated relative Cuntz-Pimsner
  rings are non-isomorphic.
\end{exem}

\section{Cuntz-Pimsner rings} \label{sec:cuntz-pimsner-rings}

Let $R$ be a ring, let $(P,Q,\psi)$ be an $R$-system satisfying
condition \textbf{(FS)}, and let $J$ be a uniquely maximal faithful 
$\psi$-compatible two-sided ideal of $R$.
In view of Remark \ref{remark:CP} it is natural to define
$\mathcal{O}_{(P,Q,\psi)}(J)$ to be \emph{the Cuntz-Pimsner ring} of
$(P,Q,\psi)$. We will do that now.

\begin{defi} \label{def:CK}
Let $R$ be a ring and let $(P,Q,\psi)$ be an $R$-system satisfying
condition \textbf{(FS)}. If there 
exists a  uniquely maximal faithful 
$\psi$-compatible two-sided ideal $J$ of $R$, then we
define \emph{the Cuntz-Pimsner ring} of $(P,Q,\psi)$ to be the ring
$$\mathcal{O}_{(P,Q,\psi)}:=\mathcal{O}_{(P,Q,\psi)}(J)$$
and we let
$$(\iota_P^{CP},\iota_Q^{CP},\iota_R^{CP},\mathcal{O}_{(P,Q,\psi)})$$
denote the covariant representation
$(\iota_P^J,\iota_Q^J,\iota_R^J,\mathcal{O}_{(P,Q,\psi)}(J))$ and
call it \emph{the Cuntz-Pimsner representation of $(P,Q,\psi)$}. We
let $p_r:=\iota_R^{CP}(r)$ for $r\in R$, $y_p:=\iota_P^{CP}(p)$ for
$p\in P$ and $x_q:=\iota_Q^{CP}(q)$ for $q\in Q$.
\end{defi}

It follows from Remark \ref{remark:CP} that $\mathcal{O}_{(P,Q,\psi)}$,
if it exists, is the (up to isomorphism) unique final object of
$\category^{inj,grad}$. It can also be described as the smallest
quotient of $\mathcal{T}_{(P,Q,\psi)}$ which preserves the $\Z$-grading
of $\mathcal{T}_{(P,Q,\psi)}$ and which leaves the embedded copy of $R$
intact.

It follows from Example \ref{example_two_maximal} that it is not
always the case that there exists a uniquely maximal faithful 
$\psi$-compatible two-sided ideal of $R$. We will now
describe a condition which will guarantee the existence of such an ideal. 
This condition is satisfied by many interesting examples, see Example
\ref{examples_cuntz:cross} -- \ref{examples_cuntz:graph_cuntz}.

If $J$ is a two-sided ideal of a ring $R$, then we let $J^\perp$ denote the
two-sided ideal $\{x\in R: \forall y\in J:xy=yx=0\}$. The following lemma
is then obvious.

\begin{lem} \label{lemma:perp}
  Let $R$ be a ring and let $(P,Q,\psi)$ be an $R$-system which
  satisfies condition \textbf{(FS)}. If
  $(\Delta^{-1}(\mathcal{F}_P(Q))\cap(\ker\Delta)^\perp)\cap\ker\Delta=\{0\}$,
  then $J=:\Delta^{-1}(\mathcal{F}_P(Q))\cap(\ker\Delta)^\perp$ is a
  uniquely maximal faithful $\psi$-compatible two-sided ideal of $R$. 
  Thus the Cuntz-Pimsner ring of $(P,Q,\psi)$ is defined in this case.
\end{lem}

A ring $R$ is said to be \textit{semiprime} if whenever $I$ is a
two-sided ideal of $R$ such that $I^2=\{0\}$,
then $I=\{0\}$. A two-sided ideal $I$ is said to be \textit{semiprime} if
whenever there exists a two-sided ideal $J$ with $J^2\subseteq I$, then
$J\subseteq I$. Equivalently $I$ is a semiprime ideal if and only
if $R/I$ is a semiprime ring. Observe that in particular
every $C^*$-algebra $A$ is semiprime and every closed ideal $I$ of
$A$ is also semiprime (since it is a $C^*$-algebra itself).

\begin{lem} \label{lemma:semiprime}
  Let $R$ be a ring which is semiprime, and let $(P,Q,\psi)$ be an
  $R$-system which satisfies condition \textbf{(FS)}. Then
  $(\ker\Delta)^\perp\cap\ker\Delta=\{0\}$.
\end{lem}

\begin{proof}
It is clear that $(\ker\Delta)^\perp\cap\ker\Delta$ is a two-sided
ideal of $R$ satisfying
$((\ker\Delta)^\perp\cap\ker\Delta)^2=\{0\}$. Thus
$(\ker\Delta)^\perp\cap\ker\Delta=\{0\}$.
\end{proof}

Thus when $R$ is semiprime, then $\Delta^{-1}(\mathcal{F}_P(Q))\cap(\ker\Delta)^\perp$ 
is a uniquely maximal faithful $\psi$-compatible two-sided ideal of $R$ 
for every $R$-system $(P,Q,\psi)$ and the Cuntz-Pimsner ring $\mathcal{O}_{(P,Q,\psi)}$
is defined.

Before we look at some examples where the Cuntz-Pimsner ring is
defined, we notice that it directly follows from Theorem
\ref{uniqueness}  that if the Cuntz-Pimsner ring of an $R$-system is
defined, then it satisfies the Graded Uniqueness Theorem.

\begin{corol}[The Graded Uniqueness Theorem]\label{uniqueness_semiprime}
  Let $R$ be a ring and let $(P,Q,\psi)$ be an
  $R$-system which satisfies condition \textbf{(FS)}, and assume that
  there exists a uniquely maximal faithful $\psi$-compatible two-sided ideal of $R$.
  If $A$ is a $\mathbb{Z}$-graded ring and
  $\eta:\mathcal{O}_{(P,Q,\psi)}\longrightarrow A$ is a graded
  ring homomorphism with $\eta(p_r)\neq 0$ for every $r\in
  R\setminus\{0\}$, then $\eta$ is injective.
\end{corol}

\begin{exem}\label{examples_cuntz:cross}
  Let us return to the Example
  \ref{examples:item:1}. We saw that if $R$ is a ring, $\varphi\in\aut(R)$, $P=R_\varphi$,
  $Q=R_{\varphi\inv}$ and
  \begin{eqnarray*}
    \psi:P\otimes_R Q & \longrightarrow R \\
    p\otimes q & \longmapsto p\varphi(q),
  \end{eqnarray*}
  then $(P,Q,\psi)$ is a $R$-system.

  Assume that $R$ has local units. If $q_1,q_2,\dots,q_n\in Q$ and
  $p_1,p_2,\dots,p_m\in P$ then there exists an idempotent $e\in R$ such
  that $eq_i=q_i$ for all $i\in\{1,2,\dots,n\}$ and $p_je=p_j$ for all
  $j\in\{1,2,\dots,m\}$ (we are here viewing the $q_i$'s and the
  $p_j$'s as elements of $R$ and using the multiplication of $R$). We
  then have that
  $\theta_{e,\varphi(e)}(q_i)=e\varphi^{-1}(\varphi(e)\varphi(q_i))=eeq_i=q_i$
  for all $i\in\{1,2,\dots,n\}$ and
  $\theta_{e,\varphi^{-1}(e)}(p_j)=p_j\varphi(\varphi^{-1}(e))e=p_jee=p_j$ for all
  $j\in\{1,2,\dots,m\}$. Thus $(P,Q,\psi)$ satisfies condition
  \textbf{(FS)}. Observe that we in this case have that $\Delta^{-1}(\mathcal{F}_P(Q))=R$
  because $\Delta(r)=\theta_{u,\varphi(r)}$ for every $r\in R$
  and $u\in R$ with $ur=ru=r$. Notice also that $\Delta$ is injective,
  so $R$ is a uniquely maximal faithful $\psi$-compatible two-sided ideal. Thus the
  Cuntz-Pimsner ring of the $R$-system $(P,Q,\psi)$ exists and is
  equal to $\mathcal{O}_{(P,Q,\psi)}(R)$.

  We saw in Example \ref{examples:item:1} that if
  $(S,T,\sigma,B)$ is a covariant representation of $(P,Q,\psi)$ and
  we for every $r\in R$ and $n\in\No$ let $(r,n)=S^n(r)$,
  $[r,-n]=T^n(r)$ and $[r,0]=\sigma(r)$, then
  $[r_1,k]+[r_2,k]=[r_1+r_2,k]$ for $r_1,r_2\in R$ and $k\in\Z$ and
  $[r_1,k_1][r_2,k_2]=[r_1\varphi^{k_1}(r_2),k_1+k_2]$ for
  $r_1,r_2\in R$ and $k_1,k_2\in\Z$ if $k_1$ and $k_2$ both are
  non-positive, or both are non-negative, or if $k_1$ is
  non-negative and $k_2$ is non-positive. If in addition
  $(S,T,\sigma,B)$ is Cuntz-Pimsner invariant relative to $R$, then
  we have for $r_1,r_2,u_1,u_2\in R$ where $r_2u_1=r_2$ and
  $u_2r_1=r_1$, and $n_1,n_2\in\No$ that
  \begin{gather*}
    \begin{split}
      [r_1,-n_1][r_2,n_1]&=T^{n_1}(r_1)S^{n_1}(r_2)
      =\pi_{S^{n_1},T^{n_1}}(\theta_{r_1,r_2})\\
      &=\sigma(r_1\varphi^{-n_1}(r_2))=[r_1\varphi^{-n_1}(r_2),0],
    \end{split}\\
    \begin{split}
      [r_1,-n_1][r_2,n_1+n_2]&=[r_1,-n_1][r_2,n_1][\varphi^{-n_1}(u_1),n_2]\\
      &=[r_1\varphi^{-n_1}(r_2),0][\varphi^{-n_1}(u_1),n_2]\\
      &=[r_1\varphi^{-n_1}(r_2)\varphi^{-n_1}(u_1),n_2]=[r_1\varphi^{-n_1}(r_2),n_2],
    \end{split}\\
    \begin{split}
      [r_1,-n_1-n_2][r_2,n_1]&=[u_2,-n_2][\varphi^{n_2}(r_1),-n_1][r_2,n_1]\\
      &=[u_2,-n_2][\varphi^{n_2}(r_1)\varphi^{n_1}(r_2),0]\\
      &=[u_2r_1\varphi^{-n_1-n_2}(r_2),-n_2]=[r_1\varphi^{-n_1-n_2}(r_2),-n_2].
    \end{split}
  \end{gather*}
  Thus $[r_1,k_1][r_2,k_2]=[r_1\varphi^{k_1}(r_2),k_1+k_2]$ for
  $r_1,r_2\in R$ and $k_1,k_2\in\Z$.

  If on the other hand we have a ring $B$ which contains a set of
  elements $\{[r,k]: r\in R,\ k\in\Z\}$ satisfying
  $[r_1,k]+[r_2,k]=[r_1+r_2,k]$ and
  $[r_1,k_1][r_2,k_2]=[r_1\varphi^{k_1}(r_2),k_1+k_2]$, and we
  define $\sigma:R\longrightarrow B$ by $\sigma(r)=[r,0]$,
  $S:P\longrightarrow B$ by $S(p)=[p,1]$, and $T:Q\longrightarrow B$
  by $T(q)=[q,-1]$, then $(S,T,\sigma,B)$ is a covariant
  representation of $(P,Q,\psi)$ which is Cuntz-Pimsner invariant
  relative to $R$.

  Thus $\mathcal{O}_{(P,Q,\psi)}$ is the universal ring generated by
  elements  $\{[r,k]: r\in R,\ k\in\Z\}$ satisfying
  $[r_1,k]+[r_2,k]=[r_1+r_2,k]$ and
  $[r_1,k_1][r_2,k_2]=[r_1\varphi^{k_1}(r_2),k_1+k_2]$; i.e.,
  $\mathcal{O}_{(P,Q,\psi)}$ is isomorphic to the crossed product
  $R\times_\varphi\Z$.

  We will return to this example in Example \ref{examples: ideals}.
\end{exem}

\begin{exem} \label{exam:endo}
  Let $R$ be a ring and let $\alpha:R\longrightarrow R$ be a ring
  homomorphism. Let $P:=\spa\{r_1\alpha(r_2)\mid r_1,r_2\in R\}$ be the
  $R$-module with left action defined by $r\cdot p=rp$ and right
  action defined by $p\cdot r=p\alpha(r)$ for $r\in R$ and $p\in P$,
  and let $Q:=\spa\{\alpha(r_1)r_2\mid r_1,r_2\in R\}$ be the
  $R$-module with left action defined by $r\cdot q=\alpha(r)q$ and
  right action defined by $q\cdot r=qr$ for $r\in R$ and $q\in
  Q$. Finally let $\psi:P\otimes Q\longrightarrow R$ be the bimodule homomorphism
  defined by $\psi(p\otimes q)=pq$. Then $(P,Q,\psi)$ is an
  $R$-system.

  If $(S,T,\sigma,B)$ is a covariant representation of $(P,Q,\psi)$,
  then $S(p)\sigma(r)=S(p\alpha(r))$, $\sigma(r)S(p)=S(rp)$,
  $T(q)\sigma(r)=T(qr)$, $\sigma(r)T(q)=T(\alpha(r)q)$ and
  $S(p)T(q)=\sigma(pq)$ for $p\in P$, $q\in Q$ and $r\in R$ where we
  view $p$ and $q$ as elements of $R$ and use the multiplication of
  $R$.

  It is not difficult to show that if $R$ has local units, then
  $(P,Q,\psi)$ satisfies condition \textbf{(FS)},
  $\Delta^{-1}(\mathcal{F}_P(Q))=R$ and that $\ker\Delta=\{0\}$. Thus
  the Cuntz-Pimsner ring of $(P,Q,\psi)$ is defined in this case and is
  equal to $\mathcal{O}_{(P,Q,\psi)}(R)$. If in addition $\alpha$ is
  injective and $\alpha(r_1)r_2\alpha(r_3)\in\alpha(R)$ for all
  $r_1,r_2,r_3\in R$, then a covariant representation $(S,T,\sigma,B)$
  of $(P,Q,\psi)$ is Cuntz-Pimsner invariant relative to $R$ if and
  only if $T(q)S(p)=\sigma(\alpha^{-1}(qp))$ for all $p\in P$ and
  $q\in Q$.

  It is not difficult to see that if $\alpha$ is an automorphism and
  $R$ has local units, then $\mathcal{O}_{(P,Q,\psi)}$ is isomorphic
  to the crossed product $R\times_\alpha\Z$, cf. Example \ref{examples_cuntz:cross}.
\end{exem}

\begin{exem}\label{examples_cuntz:skew}
Given a unital ring $R$ and a ring isomorphism
$\alpha:R\longrightarrow eRe$ where $e$ is an idempotent of $R$. Ara,
Gonz\'alez-Barroso, Goodearl and Pardo have in \cite{AGGP} defined
the \emph{fractional skew monoid ring} of the system $(R,\alpha)$
to be the universal unital ring $R[t_+,t_-;\alpha]$ generated by elements
$t_+$, $t_-$ and $\{\phi(r)\mid r\in R\}$ satisfying that $\phi:R\longrightarrow
R[t_+,t_-;\alpha]$ is a unital ring homomorphism and
that the relations
$$t_-t_+=1\qquad,\qquad t_+t_-=\phi(e)\qquad,\qquad
rt_-=t_-\alpha(r) \qquad\text{and}\qquad
t_+r=\alpha(r)t_+$$ hold for all $r\in R$.
This construction is an exact algebraic analog of the construction
of the crossed product of a $C^*$-algebra by an endomorphism
introduced by Paschke \cite{PA}. In fact, if $A$ is a $C^*$-algebra
and the corner isomorphism $\alpha$ is a $*$-homomorphism, then
Paschke's $C^*$-crossed product, which he denotes $A\ltimes_\alpha
\mathbb{N}$, is just the completion of $A[t_+,t_-;\alpha]$ in a
suitable norm. The Cuntz-Krieger rings, crossed products by
automorphisms and Leavitt path algebras of finite graphs without
sinks are examples of fractional skew monoid rings among many
others (see \cite{AGGP}). As an important advance in the study of
this class of rings, in \cite[Theorem 5.3]{AGGP} conditions for
$R[t_+,t_-;\alpha]$ being a simple and purely infinite ring are
given, and in \cite{AB} the $K_1$ of fractional skew monoid
rings is computed.

We will now show that the fractional skew monoid ring $R[t_+,t_-;\alpha]$ is
isomorphic, as a $\Z$-graded ring, to $\mathcal{O}_{(P,Q,\psi)}$
where $(P,Q,\psi)$ is the $R$-system considered in Example
\ref{exam:endo}. First we notice that if $r_1,r_2,r_3\in R$, then
$\alpha(r_1)r_2\alpha(r_3)\in eReReRe\subseteq eRe=\alpha(R)$.
Define $S:P\longrightarrow R[t_+,t_-;\alpha]$ and
$T:Q\longrightarrow R[t_+,t_-;\alpha]$ by
$S(p)=\phi(p)t_+$ and $T(q)=t_-\phi(q)$. It is then easy to check
that $(S,T,\phi,R[t_+,t_-;\alpha])$ is a surjective covariant
representation of $(P,Q,\psi)$ which is Cuntz-Pimsner invariant
relative to $R$, cf. Example \ref{exam:endo}. Thus it follows from
Theorem \ref{univ_cuntz} that there exists a ring homomorphism
$\eta:\mathcal{O}_{(P,Q,\psi)} \longrightarrow R[t_+,t_-;\alpha]$ such that
$\eta(p_r)=\phi(r)$, $\eta(y_p)=\phi(p)t_+$ and
$\eta(x_q)=t_-\phi(q)$ for $r\in R$, $p\in P$ and $q\in Q$. It
follows from \cite[Proposition 1.6 and Corollary 1.11]{AGGP} that
$\eta$ is graded and that $\eta(p_r)\ne 0$ for $r\ne 0$, so $\eta$
is injective and thus an isomorphism according to Corollary
\ref{uniqueness_semiprime}.
\end{exem}

\begin{exem}\label{examples_cuntz:graph_cuntz}
Let us return to the Example \ref{examples:graph-toeplitz}. Given
$q=\left(\sum_{e\in E^1} \lambda_e\textbf{1}_{e}\right) \in Q$ we let
$$\textrm{Supp}(\sum_{e\in E^1} \lambda_e\textbf{1}_{e} ):=\{e\in E^1:
\lambda_{e}\neq 0\}\,.$$ Notice that $|\textrm{Supp}(q)|<\infty$.
Given $q_1,\ldots,q_n\in Q$ we have that the homomorphism
$$\Theta=\sum_{e\in
\textrm{Supp}(q_1)\cup\cdots\cup \textrm{Supp}(q_n)
}\theta_{\textbf{1}_{e},\textbf{1}_{\overline{e}}}\in
\mathcal{F}_{P}(Q)
$$
satisfies $\Theta(q_i)=q_i$ for every $i\in\{1,2,\dots,n\}$.
Similarly, we have that there for $p_1,p_2,\dots,p_n\in P$ exists a
homomorphism $\Delta\in\mathcal{F}_Q(P)$ such that $\Delta(p_i)=p_i$
for every $i\in\{1,2,\dots,n\}$. Thus the $R$-system $(P,Q,\psi)$
satisfies the condition \textbf{(FS)}.

Now it is easy to see that
\begin{gather*}
\Delta^{-1}(\mathcal{F}_P(Q))=\textrm{span}_F\{\textbf{1}_v:
|s^{-1}(v)|<\infty\}\,,\\
\ker\Delta=\text{span}_F\{\textbf{1}_v: |s^{-1}(v)|=0\}\,.
\end{gather*}
It follows that $(\ker\Delta)^\perp=\text{span}_F\{\textbf{1}_v:
|s^{-1}(v)|>0\}$, and thus that
$(\Delta^{-1}(\mathcal{F}_P(Q))\cap(\ker\Delta)^\perp)\cap\ker\Delta=\{0\}$. Hence
the Cuntz-Pimsner ring of $(P,Q,\psi)$ is defined in this case and is
equal to $\mathcal{O}_{(P,Q,\psi)}(\Delta^{-1}(\mathcal{F}_P(Q))\cap(\ker\Delta)^\perp)$. 

We saw in Example \ref{examples:graph-toeplitz} that if
$(S,T,\sigma,B)$ be a covariant representation of $(P,Q,\psi)$ and
we let $p_{v}:=\sigma(\textbf{1}_v)$ for $v\in E^0$, and
$x_{e}=T(\textbf{1}_e)$ and $y_{e}=S(\textbf{1}_{\overline{e}})$ for
$e\in E^1$, then $\mathcal{R}\langle S,T,\sigma\rangle$ becomes a
$F$-algebra when we equip it with an $F$-multiplication of $F$
defined by $\lambda \sigma(r)=\sigma(\lambda r)$, $\lambda
S(p)=S(\lambda p)$ and $\lambda T(q)=T(\lambda q)$ for $\lambda\in
F$, $r\in R$, $p\in P$ and $q\in Q$. Then $\{p_{v}\}_{v\in E^0}$ is
a family of pairwise orthogonal idempotents such that we for all
$e,f\in E^1$ have that $p_{s(e)}x_{e}=x_{e}=x_{e}p_{r(e)}$,
$p_{r(e)}y_{e}=y_{e}=y_{e}p_{s(e)}$, and
$y_{e}x_{f}=\delta_{e,f}p_{r(e)}$. If in addition $(S,T,\sigma,B)$
is Cuntz-Pimsner invariant relative to
$\Delta^{-1}(\mathcal{F}_P(Q))\cap(\ker\Delta)^\perp
=\textrm{span}_F\{\textbf{1}_v: 0<|s^{-1}(v)|<\infty\}$, then we
have for $v\in E^0$ with $0<|s^{-1}(v)|<\infty$ that
\begin{equation*}
p_{v}=\sigma(\textbf{1}_v)=\pi_{T,S}\bigl(\Delta(\textbf{1}_v)\bigr)
=\pi_{T,S}\left(\sum_{e\in s\inv(v)}
\theta_{\textbf{1}_e,\textbf{1}_{\overline{e}}}\right) =\sum_{e\in
s\inv(v)}T(\textbf{1}_e)S(\textbf{1}_{\overline{e}}) =\sum_{e\in
s\inv(v)}x_{e}y_{e}\,.
\end{equation*}

On the other hand, let  $B$ be an $F$-algebra which contains a
family $\{p_{v}\}_{v\in E^0}$ of pairwise orthogonal idempotents and
families $\{x_{e}\}_{e\in E^1}$ and $\{y_{e}\}_{e\in E^1}$
satisfying $p_{s(e)}x_{e}=x_{e}=x_{e}p_{r(e)}$,
$p_{r(e)}y_{e}=y_{e}=y_{e}p_{s(e)}$, and
$y_{e}x_{f}=\delta_{e,f}p_{r(e)}$ for all $e,f\in E^1$. Then for
$r=\sum_{v\in E^0}s_v\textbf{1}_v\in R$ let $\sigma(r):=\sum_{v\in
E^0}s_vp_{v}$, for $p=\sum_{e\in E^1}
\lambda_e\textbf{1}_{\overline{e}}\in P$ let $S(p):=\sum_{e\in
E^1}p_ey_{e}$, and for $q=\sum_{e\in E^1}\lambda_e\textbf{1}_{e}\in
Q$ let $T(q):=\sum_{e\in E^1}q_ex_{e}$, we have that
$(S,T,\sigma,B)$ is a covariant representation of $(P,Q,\psi)$ which
is Cuntz-Pimsner invariant relative to
$\Delta^{-1}(\mathcal{F}_P(Q))\cap(\ker\Delta)^\perp$.

Thus $\mathcal{O}_{(P,Q,\psi)}$ is the universal $F$-algebra
generated by a set $\{p_{v}: v\in E^0\}$ of pairwise orthogonal
idempotents, together with a set $\{x_{e},y_{e}: e\in E^1\}$ of
elements satisfying
\begin{enumerate}
\item $p_{s(e)}x_{e}=x_{e}=x_{e}p_{r(e)}$ for
  $e\in E^1$,
\item $p_{r(e)}y_{e}=y_{e}=y_{e}p_{s(e)}$
  for $e\in E^1$,
\item $y_{e}x_{f}=\delta_{e,f}p_{r(e)}$ for $e,f\in E^1$,
\item $p_{v}=\sum_{e\in s\inv(v)}x_{e}y_{e}$ for $v\in
  E^0$ with $0<|s^{-1}(v)|<\infty$.
\end{enumerate}
I.e., $\mathcal{O}_{(P,Q,\psi)}$ is isomorphic to the Leavitt path $L_F(E)$
algebra associated with $E$,
cf. \cite{AA1},\cite{AA2},\cite{AALP},\cite{AMP}\&\cite{TF}. Thus we
recover from Corollary \ref{uniqueness_semiprime} the Graded
Uniqueness Theorem \cite[Theorem 4.8]{TF} for Leavitt path algebras.

We will return to this example in Example \ref{examples: ideals: graph ideals}.
\end{exem}

\section{The Algebraic Gauge-invariant Theorem}
We saw in Example \ref{examples_cuntz:cross} that our Graded
Uniqueness Theorem (Corollary \ref{uniqueness_semiprime}) is a
generalization of the Graded Uniqueness Theorem for Leavitt path
algebras (\cite[Theorem 4.8]{TF}). We will now generalize the
Algebraic Gauge-Invariant Uniqueness Theorem for row finite graphs
(\cite[Theorem 1.8]{AALP}) to Cuntz-Pimsner rings and thereby to all
directed graphs.

\begin{prop}[{Cf. \cite[Proposition 1.3]{FMR} and \cite[Remark 1.2(2)]{PI}}]
Let $R$ be an (associative) $F$-algebra where $F$ is a field, and let
$(P,Q,\psi)$ be an $R$-system satisfying condition \textbf{(FS)}
and let $J$ be a $\psi$-compatible two-sided ideal of $R$. 
Then there exists for every $t\in
F^*$ ($F^*$ denotes the multiplication group of $F$) a unique
automorphism $\tau^{J}_t$ on $\mathcal{O}_{(P,Q,\psi)}(J)$
satisfying $\tau^J_t(\iota_R^J(r))=\iota_R^J(r)$,
$\tau^J_t(\iota_P^J(p))=t\iota_P^J(p)$ and
$\tau^J_t(\iota_Q^J(q))=t\inv\iota_Q^J(q)$
for $r\in R$, $p\in P$ and $q\in Q$.

The action
$$\begin{array}{rl}
\tau^{J}: F^* & \longrightarrow \text{Aut}_F(\mathcal{O}_{(P,Q,\psi)}(J)) \\
t & \longmapsto \tau^{J}_t
\end{array}
$$
is called \emph{the gauge action of $F$ on
$\mathcal{O}_{(P,Q,\psi)}(J)$}.
\end{prop}

\begin{proof}
  Since $\mathcal{O}_{(P,Q,\psi)}(J)$ is generated by
  $\{\iota_R^J(r):r\in R\}\cup \{\iota_P^J(p):p\in P\}\cup
  \{\iota_Q^J(q):q\in Q\}$, it follows that
  a ring homomorphism defined on $\mathcal{O}_{(P,Q,\psi)}(J)$ is
  uniquely determined by its values on $\{\iota_R^J(r):r\in R\}\cup
  \{\iota_P^J(p):p\in P\}\cup \{\iota_Q^J(q):q\in Q\}$.

  Let $t\in F^*$. For $r\in R$, $p\in P$
  and $q\in Q$ let $\sigma(r)=\iota_R^J(r)$, $S(p)=t\iota_P^J(p)$ and
  $T(q)=t\inv \iota_Q^J(q)$.
  Then $(S,T,\sigma,\mathcal{O}_{(P,Q,\psi)}(J))$ is a covariant
  representation of $(P,Q,\psi)$ which is Cuntz-Pimsner invariant
  relative to $J$. Thus there exists a homomorphism
  $\tau^J_t:\mathcal{O}_{(P,Q,\psi)}(J)\longrightarrow
  \mathcal{O}_{(P,Q,\psi)}(J)$ such that $\tau^J_t(\iota_R^J(r))=\iota_R^J(r)$,
  $\tau^J_t(\iota_P^J(p))=t\iota_P^J(p)$ and
  $\tau^J_t(\iota_Q^J(q))=t\inv \iota_Q^J(q)$ for $r\in R$,
  $p\in P$ and $q\in Q$.

  If $t_1,t_2\in F^*$ and  $r\in R$, $p\in P$ and $q\in Q$, then
  $\tau^J_{t_1}\circ\tau^J_{t_2}(\iota_R^J(r))=\tau^J_{t_1t_2}(\iota_R^J(r))$,
  $\tau^J_{t_1}\circ\tau^J_{t_2}(\iota_P^J(p))=\tau^J_{t_1t_2}(\iota_P^J(p))$, and
  $\tau^J_{t_1}\circ\tau^J_{t_2}(\iota_Q^J(q))=\tau^J_{t_1t_2}(\iota_Q^J(q))$, so
  $\tau^J_{t_1}\circ\tau^J_{t_2}=\tau^J_{t_1t_2}$. We have in
  particular that
  $\tau^J_t\circ\tau^J_{t\inv}=\id_{\mathcal{O}_{(P,Q,\psi)}(J)}$, so
  $\tau^J_t$ is an automorphism.
\end{proof}

\begin{theor}
  Let $F$ be an infinite field, $R$ an (associative) $F$-algebra, and
  let $(P,Q,\psi)$ be an $R$-system satisfying condition
  \textbf{(FS)}. Assume that $J$ is a maximal faithful $\psi$-compatible
  two-sided ideal of $R$, and let $A$ be
  an $F$-algebra. Suppose that
  $$\phi:\mathcal{O}_{(P,Q,\psi)}(J)\longrightarrow A$$
  is a $F$-algebra homomorphism such that $\phi(\iota_R^J(r))\neq 0$ for
  every $r\in R\setminus \{0\}$. If there exists a group action
  $\sigma:F^*\longrightarrow \text{Aut}_F(A)$ such that $\phi\circ
  \tau^{J}_t=\sigma_t\circ \phi$ for every $t\in F^*$, then
  $\phi$ is injective.
\end{theor}

\begin{proof}
  By Theorem \ref{uniqueness} it is enough to check that
  $\oplus_{n\in\Z}\phi(\rho_J(\mathcal{T}_{(P,Q,\psi)}^{(n)})$ is
  a grading of $B$. We will do that by showing that $\ker\phi$ is a graded ideal.
  Assume that $\phi(z_{n_1}+\cdots+z_{n_r})=0$,
  $n_1,\ldots,n_r\in \mathbb{Z}$, $n_i\ne n_j$ for $i\ne j$ and
  $z_{n_i}\in \rho_J(\mathcal{T}^{(n_i)}_{(P,Q,\psi)})$ for every
  $i=1,\ldots,r$.
  We then have for $t\in F^*$ that
  \begin{equation*}
    0=\sigma_t\bigl(\phi(z_{n_1}+\cdots+z_{n_r})\bigr)
    =\phi\bigl(\tau^J_t(z_{n_1}+\cdots+z_{n_r})\bigr)
    =\phi(t^{n_1}z_{n_1}+\cdots+t^{n_r}z_{n_r})\,.
  \end{equation*}
  On the other hand we have that
  $0=t^{n_r}\phi(z_{n_1}+\cdots+z_{n_r})=\phi(t^{n_r}z_{n_1}+\cdots+t^{n_r}z_{n_r})$.
  It follows that
  $$0=\phi((t^{n_r}-t^{n_1})z_{n_1}+\cdots+(t^{n_r}-t^{n_{r-1}})z_{n_{r-1}})\,,$$
  and since $F$ is an infinite field we have that
  $t^{n_r}-t^{n_i}\neq 0$ for every $i=1,\ldots,r-1$. Repeating this
  process $r-1$  times we get that $\phi(z_{n_1})=0$ as desired.
  Repeating the same argument we get that $\phi(z_{n_i})=0$ for
  every $i=1,\ldots,r$.
  This shows that $\ker\phi$ is a graded ideal and thus that
  $\oplus_{n\in\Z}\phi(\rho_J(\mathcal{T}_{(P,Q,\psi)}^{(n)})$ is
  a grading of $B$.
\end{proof}

If $F$ is a field, $R$ is an $F$-algebra, $(P,Q,\psi)$ is an
$R$-system satisfying condition \textbf{(FS)}, and $J$ is a
uniquely maximal faithful $\psi$-compatible two-sided ideal of $R$, 
then we denote by $\tau^{CP}$ the gauge
action $\tau^J$ of $\mathcal{O}_{(P,Q,\psi)}=\mathcal{O}_{(P,Q,\psi)}(J)$. 
We then get as a corollary to the previous theorem the following Gauge-Invariant
Uniqueness Theorem for Cuntz-Pimsner rings.

\begin{corol}[The Gauge-Invariant Uniqueness Theorem for Cuntz-Pimsner
  Rings, {cf. \cite[Theorem 4.1]{FMR}}]  \label{GIUTFCP}
  Let $F$ be an infinite field, $R$ an (associative) $F$-algebra and
  let $(P,Q,\psi)$ be an $R$-system satisfying condition
  \textbf{(FS)}. Assume that
  there exists a uniquely maximal faithful $\psi$-compatible two-sided ideal of $R$.
  Let $A$ be an $F$-algebra. Suppose that
  $$\phi:\mathcal{O}_{(P,Q,\psi)}\longrightarrow A$$
  is an $F$-algebra homomorphism such that $\phi(p_r)\neq 0$ for every
  $r\in R\setminus \{0\}$. If there exists a group action
  $\sigma:F^*\longrightarrow \text{Aut}_F(A)$ such that $\phi\circ
  \tau^{CP}_t=\sigma_t\circ \phi$ for every $t\in F^*$, then $\phi$ is
  injective.
\end{corol}

When we specialize to directed graphs, we get a generalization of the
Algebraic Gauge-Invariant Uniqueness Theorem \cite[Theorem 1.8.]{AALP}
from row finite graphs to all directed graphs.

\begin{corol} \label{cor:GIUTLA}
Let $E$ be a directed graph, let $F$ be an infinite field and let
$A$ be an $F$-algebra. Suppose that
$$\phi:L_F(E)\longrightarrow A$$
is a $F$-algebra homomorphism such that $\phi(p_v)\neq 0$ for every
$v\in E^0$. If there exists a group action
$\sigma:F^*\longrightarrow \text{Aut}_F(A)$ such that $\phi\circ
\tau^{E}_t=\sigma_t\circ \phi$ for every $t\in F^*$, then $\phi$ is
injective.
\end{corol}

\begin{proof}
Follows from Example \ref{examples_cuntz:graph_cuntz} and Corollary
\ref{GIUTFCP}.
\end{proof}

\section{Graded covariant representations}

In Section \ref{sec:relat-cuntz-pimsn} we classified all surjective,
injective and graded covariant representations of an $R$-system
satisfying condition \textbf{(FS)}. We will in this section extend
this classification to all surjective and graded covariant
representations. As a corollary we get a description of all graded
two-sided ideals of a relative Cuntz-Pimsner algebra (and therefore of the
Toeplitz ring and the Cuntz-Pimsner ring whenever it is defined)
of an $R$-system satisfying condition \textbf{(FS)}.

We will proceed as in Section \ref{sec:relat-cuntz-pimsn} and first
describe a family of surjective and graded covariant
representations of a given $R$-system which satisfies condition
\textbf{(FS)}, and then show that this family contains up to
isomorphism all surjective and graded covariant representations. This
approach is inspired by the work of Katsura in \cite{KS} (notice
however that our definition of a $T$-pair (see Definition
\ref{def:t-pair}) is different from Katsura's definition).

At the end of the section we will see how our
description of the graded two-sided ideals of a Cuntz-Pimsner ring agrees with
Tomforde's characterization of the
graded ideals of a Leavitt path algebra. We will also show (cf.
Proposition \ref{theor_2}) that if
the $R$-system $(P,Q,\psi)$ satisfies condition \textbf{(FS)}, then any
quotient of a relative Cuntz-Pimsner ring of $(P,Q,\psi)$ by a
graded two-sided ideal is again a relative Cuntz-Pimsner ring (but
of a different system).

\subsection{The classification of graded covariant representations of an $R$-system}
\label{sec:grad-covar-repr}

We begin with some definitions and some notation.

\begin{defi} \label{def:invariant}
Let $R$ be a ring and let $(P,Q,\psi)$ be an $R$-system. A two-sided
ideal $I$ of $R$ is said to be \textit{$\psi$-invariant} if
$\psi(p\otimes x q)\in I$ for every $p\in P$, $q\in Q$ an $x\in I$.

If $I$ is a two-sided ideal of $R$, then $QI:=\spa\{qx: q\in Q,\ x\in
I\}$ and $IQ:=\spa\{xq: q\in Q,\ x\in I\}$ are $I$-bimodules.
Similarly we define $IP:=\spa\{xp: p\in P,\ x\in I\}$ and
$PI:=\spa\{px: p\in P,\ x\in I\}$ which are also $I$-bimodules.
\end{defi}

\begin{rema}\label{rema31}
Observe that if $R$ is a ring, $(P,Q,\psi)$ is an $R$-system which
satisfies condition \textbf{(FS)}, and $I$ is $\psi$-invariant
two-sided ideal of $R$,
then $IQ\subseteq QI$ and
$PI\subseteq IP$. Indeed, let $x\in I$, then by the \textbf{(FS)}
condition there exists $\Theta=\sum^n_{i=1} \theta_{q_i,p_i}\in
\mathcal{F}_P(Q)$ such that $xq=\Theta(xq)=\sum^n_{i=1}
\theta_{q_i,p_i}(xq)=\sum^n_{i=1} q_i\psi(p_i\otimes xq)\in QI$
since $\psi(p_i\otimes xq)\in I$ for every $i\in \{1,\ldots,n\}$.
Similarly one can prove that $PI\subseteq IP$.
\end{rema}

\begin{defi}
Let $R$ be a ring and let $(P,Q,\psi)$ be an $R$-system satisfying
condition \textbf{(FS)}. For a two-sided ideal $I$ of $R$ we define
$R_I:=R/I$, $Q_I:=Q/QI$ and ${}_IP:=P/IP$. We let $\wp_I$ be their
respective projections.

It follows from Remark \ref{rema31} that if $I$ is a
$\psi$-invariant two-sided ideal of $R$, then $Q_I$ and ${}_IP$ are
$R_I$-bimodules. We can in this case define a $R_I$-bimodule
homomorphism $\psi_I:{}_IP\otimes Q_I\longrightarrow R_I$ by
$\psi_I(\wp_I(p)\otimes \wp_I(q))=\wp_I(\psi(p\otimes q))$.

Observe that we can also define a projection $\wp_I:\mathcal{L}_P(Q)
\longrightarrow \mathcal{L}_{{}_IP}(Q_I)$ such that
$\wp_I(T)(\wp_I(q))=\wp_I(T(q))$ for every $T\in\mathcal{L}_P(Q)$
and $q\in Q$, and then we have that
$\wp_I(\mathcal{F}_{P}(Q))=\mathcal{F}_{{}_IP}(Q_I)$. We also define
a ring homomorphism $\Delta_I:R_I\longrightarrow \text{End}(Q_I)$ by
$\Delta_I(\wp_I(r))\wp_I(q)=\wp_I(rq)$ for $r\in R$ and $q\in Q$. We
then have that
$\Delta_I(\wp_I(r))=\wp_I(\Delta(r))$ for every $r\in R$.
\end{defi}

We then have the following straightforward lemma:

\begin{lem} \label{lemma:rifs}
Let $R$ be a ring, let $(P,Q,\psi)$ be an $R$-system satisfying
condition \textbf{(FS)},
and let $I$ be a $\psi$-invariant two-sided ideal of $R$. Then the $R_I$-system
$({}_IP,Q_I,\psi_I)$ satisfies condition \textbf{(FS)}.
\end{lem}

\begin{defi}[{Cf. \cite[Definition 5.6]{KS}}] \label{def:t-pair}
Let $R$ be a ring and let $(P,Q,\psi)$ be an $R$-system satisfying
condition \textbf{(FS)}. A pair $\omega=(I,J)$ of two-sided ideals
of $R$ such that $I\subseteq J$ is said to be a \textit{$T$-pair} of
$(P,Q,\psi)$ if $I$ is a $\psi$-invariant ideal and
$J_I:=\wp_I(J)$ is a faithful $\psi_I$-compatible two-sided ideal of $R_I$.
\end{defi}

\noindent Notice that since $I\subseteq J$, we have that
$\wp_I^{-1}(J_I)=J$.

Let $R$ be a ring, let $(P,Q,\psi)$ be an $R$-system satisfying
condition \textbf{(FS)}, and let $\omega=(I,J)$ be a $T$-pair. Then we
define the following maps
$$\iota_R^\omega:=\iota^{J_I}_{R_I}\circ \wp_I:R\longrightarrow
\mathcal{O}_{({}_IP,Q_I,\psi_I)}(J_I)\,,$$
$$\iota_Q^\omega:=\iota^{J_I}_{Q_I}\circ \wp_I:Q\longrightarrow
\mathcal{O}_{({}_IP,Q_I,\psi_I)}(J_I)\,,$$
$$\iota_P^\omega:=\iota^{J_I}_{{}_IP}\circ \wp_I:P\longrightarrow
\mathcal{O}_{({}_IP,Q_I,\psi_I)}(J_I)\,,$$ where
$(\iota^{J_I}_{{}_IP},\iota^{J_I}_{Q_I},\iota^{J_I}_{R_I},
\mathcal{O}_{({}_IP,Q_I,\psi_I)}(J_I))$ is the universal
Cuntz-Pimsner invariant representation of $({}_IP,Q_I,\psi_I)$
relative to $J_I$. It is easy to check
$(\iota_P^\omega,\iota_Q^\omega,\iota_R^\omega,\mathcal{O}_{({}_IP,Q_I,\psi_I)}(J_I))$
is a surjective and graded covariant representation of $(P,Q,\psi)$.
We will in this section show that the family
$\{(\iota_P^\omega,\iota_Q^\omega,\iota_R^\omega,\mathcal{O}_{({}_IP,Q_I,\psi_I)}(J_I))\mid
\omega\text{ is a $T$-pair of }(P,Q,\psi)\}$ up to isomorphism
contains all surjective and graded covariant representations of
$(P,Q,\psi)$.

\begin{defi} \label{def:I}
Let $R$ be a ring, let $(P,Q,\psi)$ be an $R$-system that satisfies
condition \textbf{(FS)} and let $(S,T,\sigma,B)$ be a covariant
representation of $(P,Q,\psi)$. Then we define $I_{(S,T,\sigma,B)}$
as the two-sided ideal $\ker\sigma$ of $R$.
\end{defi}

\begin{lem}\label{lemma_31}
Let $R$ be a ring and let $(P,Q,\psi)$ be an $R$-system satisfying
condition \textbf{(FS)}. If $(S,T,\sigma,B)$ is a covariant
representation of $(P,Q,\psi)$, then $\ker T=QI_{(S,T,\sigma,B)}$
and $\ker S=I_{(S,T,\sigma,B)}P$.
\end{lem}
\begin{proof}
Clearly $QI_{(S,T,\sigma,B)}\subseteq \ker T$. Now let $q\in \ker
T$, then for every $p\in P$ we have $0=S(p)T(q)=\sigma(\psi(p\otimes
q))$ and hence $\psi(p\otimes q)\in \ker \sigma=I_{(S,T,\sigma,B)}$
for every $p\in P$. By condition \textbf{(FS)} there exists
$\Theta=\sum^n_{i=1}\theta_{q_i,p_i}$ such that $\Theta(q)=q$ and
therefore
$q=\Theta(q)=\sum^n_{i=1}\theta_{q_i,p_i}(q)=\sum^n_{i=1}q_i\psi(p_i\otimes
q)\in QI_{(S,T,\sigma,B)}$ as desired.

That $\ker S=I_{(S,T,\sigma,B)}P$ can be proved in a similar way.
\end{proof}

\begin{prop}\label{prop_31}
Let $R$ be a ring and let $(P,Q,\psi)$ be an $R$-system satisfying
condition \textbf{(FS)}. Let $(S,T,\sigma,B)$ be a covariant
representation of $(P,Q,\psi)$, and let $I_{(S,T,\sigma,B)}$ be as defined
in Definition \ref{def:I}, and let $J_{(S,T,\sigma,B)}$ be as defined
in Definition \ref{def:J} . Then the pair
$\omega_{(S,T,\sigma,B)}:=(I_{(S,T,\sigma,B)},J_{(S,T,\sigma,B)})$
is a $T$-pair of $(P,Q,\psi)$.
\end{prop}

\begin{proof}
We let $I:=I_{(S,T,\sigma,B)}$ and $J:=J_{(S,T,\sigma,B)}$. It is
clear that $I$ is a two-sided ideal of $R$, and it follows from
Lemma \ref{lemma_6} that also $J$ is a two-sided ideal of $R$. It is
clear that $I\subseteq J$.

First we prove that $I$ is $\psi$-invariant. Indeed, let $x\in I$, $p\in P$
and $q\in Q$. Then $\sigma(\psi(p\otimes x
q))=S(p)\sigma(x)T(q)=0$, so $\psi(p\otimes x q)\in \ker\sigma=I$.

Now let $x\in J=\sigma^{-1}(\pi_{T,S}(\mathcal{F}_{P}(Q)))$. Then
there exists $\Theta\in \mathcal{F}_{P}(Q)$ with
$\sigma(x)=\pi_{T,S}(\Theta)$. Thus we have for every $q\in Q$ that
$$T(xq)=\sigma(x)T(q)=\pi_{T,S}(\Theta)T(q)=T(\Theta(q))\,,$$
and it follows from Lemma \ref{lemma_31} that $xq-\Theta(q)\in \ker
T=QI$. Hence $\wp_I(xq)-\wp_I(\Theta(q))=0$, so
$\wp_I(x)\wp_I(q)=\wp_I(\Theta)(\wp_I(q))$. Since $\wp_I(\Theta)\in
\mathcal{F}_{{}_IP}(Q_I)$, it follows that $\Delta_I(\wp_I(x))\in
\mathcal{F}_{{}_IP}(Q_I)$.

Now we check that $J_I\cap \ker\Delta_I=0$. Let $x\in J$ and assume
that $\wp_I(x)\in \ker\Delta_I$. Then $xq\in QI$ for every $q\in Q$.
But since $x\in J$,  there exists
$\Theta=\sum^n_{i=1}\theta_{q_i,p_i}\in\mathcal{F}_{P}(Q)$ such that
$\sigma(x)=\pi_{T,S}(\Theta)=\sum^n_{i=1}T(q_i)S(p_i)$. It then
follows from Lemma \ref{lemma_31} that
$xq-\sum^n_{i=1}q_i\psi(p_i\otimes q)\in \ker T=QI$, so
$\sum^n_{i=1}q_i\psi(p_i\otimes q)\in QI$ for every $q\in Q$. Now by
condition \textbf{(FS)} there exist
$\Theta_1=\sum^m_{j=1}\theta_{a_j,b_j}\in \mathcal{F}_{P}(Q)$ and
$\Theta_2=\sum^l_{k=1}\theta_{c_k,d_k}\in \mathcal{F}_{Q}(P)$ such
that $\Theta_1(q_i)=q_i$ and $\Theta_2(p_i)=p_i$ for every
$i=1,\ldots,n$. Then we have
\begin{equation*}
\begin{split}
\sigma(x)
&=\sum^n_{i=1}T(q_i)S(p_i)=\sum^n_{i=1}T(\Theta_1(q_i))S(\Theta_2(p_i))\\
&=\sum^n_{i=1}T\left(\sum^m_{j=1}\theta_{a_j,b_j}(q_i)\right)S\left(\sum^l_{k=1}\theta_{c_k,d_k}(p_i)\right)\\
&=\sum^n_{i=1}T\left(\sum^m_{j=1}a_j\psi(b_j\otimes q_i)\right)S\left(\sum^l_{k=1}\psi(p_i\otimes d_k)c_k)\right)\\
&=\sum^n_{i=1}\sum^m_{j=1}\sum^l_{k=1}T(a_j)\sigma(\psi(b_j\otimes q_i)\psi(p_i\otimes d_k))S(c_k)\\
&=\sum^m_{j=1}\sum^l_{k=1}T(a_j)\sigma\left(\psi(b_j\otimes \sum^n_{i=1}
q_i\psi(p_i\otimes d_k))\right)S(c_k)\\
&=\sum^m_{j=1}\sum^l_{k=1}T(a_j)\sigma(\psi(b_j\otimes
\Theta(d_k)))S(c_k),
\end{split}
\end{equation*}
but $\Theta(d_k)\in QI$ for every $k=1,\ldots,l$, and hence
$\psi(b_j\otimes \Theta(d_k))\in I$. So $\sigma(\psi(b_j\otimes
\Theta(d_k)))=0$, from which it follows that
$0=\sum^n_{i=1}T(q_i)S(p_i)=\sigma(x)$, and therefore $x\in
\ker\sigma=I$. Thus $\wp_I(x)=0$.
\end{proof}

\begin{prop}\label{prop_32}
Let $R$ be a ring and let $(P,Q,\psi)$ an $R$-system satisfying
condition \textbf{(FS)}. If $\omega=(I,J)$ is a $T$-pair, then
$\omega=\omega_{(\iota_P^\omega,\iota_Q^\omega,\iota_R^\omega,\mathcal{O}_{({}_IP,Q_I,\psi_I)}(J_I))}$.
\end{prop}

\begin{proof}
First notice that
$I_{(\iota_P^\omega,\iota_Q^\omega,\iota_R^\omega,\mathcal{O}_{({}_IP,Q_I,\psi_I)}(J_I))}=\ker\iota_R^\omega
=\ker(\iota^{J_I}_{R_I}\circ \wp_I)=\ker \wp_I=I$ by injectivity of
$\iota^{J_I}_{R_I}$.

Let $x\in J$. Then we have that $\wp_I(x)\in J_I$ and thus that
$$\iota_R^\omega(x)=\iota^{J_I}_{R_I}(\wp_I(x))=\pi^{J_I}(\Delta_I(\wp_I(x)))
\in\pi^{J_I}(\mathcal{F}_{{}_IP}(Q_I))=\pi^{J_I}(\wp_I(\mathcal{F}_{P}(Q)))
=\pi_{\iota_Q^\omega,\iota_P^\omega}(\mathcal{F}_{P}(Q))\,,$$ and
therefore $x\in
(\iota_R^\omega)^{-1}(\pi_{\iota_Q^\omega,\iota_P^\omega}(\mathcal{F}_{P}(Q)))$.
This shows that $J\subseteq
J_{(\iota_P^\omega,\iota_Q^\omega,\iota_R^\omega,\mathcal{O}_{({}_IP,Q_I,\psi_I)}(J_I))}$.

Assume now that $x\in
J_{(\iota_P^\omega,\iota_Q^\omega,\iota_R^\omega,\mathcal{O}_{({}_IP,Q_I,\psi_I)}(J_I))}$.
Then we have $$\iota^{J_I}_{R_I}(\wp_I(x))=\iota_R^\omega(x)\in
\pi_{\iota_Q^\omega,\iota_P^\omega}(\mathcal{F}_{P}(Q))
=\pi^{J_I}(\wp_I(\mathcal{F}_{P}(Q)))=\pi^{J_I}(\mathcal{F}_{{}_IP}(Q_I)).$$
Since $J_I\subseteq \Delta^{-1}_I(\mathcal{F}_{{}_IP}(Q_I))$ and
$J_I\cap \ker\Delta_I=0$, it follows from Proposition \ref{prop_3}
that $\wp_I(x)\in J_I$. Thus $x\in J$ which shows that
$J_{(\iota_P^\omega,\iota_Q^\omega,\iota_R^\omega,\mathcal{O}_{({}_IP,Q_I,\psi_I)}(J_I))}
\subseteq J$.
\end{proof}

\begin{lem}\label{lemma_33}
  Let $R$ be a ring and $(P,Q,\psi)$ an $R$-system. Let
  $(S,T,\sigma,B)$ be a covariant representation and let $I$ be a $\psi$-invariant ideal of
  $R$. Then we have:
  \begin{enumerate}
  \item If there is a covariant representation
    $(S_I,T_I,\sigma_I,B)$ of $({}_IP,Q_I,\psi_I)$ such that
    $T=T_I\circ \wp_I$, $S=S_I\circ \wp_I$ and $\sigma=\sigma_I\circ \wp_I$,
    then $I\subseteq I_{(S,T,\sigma,B)}$.
  \item If $I\subseteq I_{(S,T,\sigma,B)}$, then there exists a unique covariant representation
    $(S_I,T_I,\sigma_I,B)$ of $({}_IP,Q_I,\psi_I)$ such that
    $T=T_I\circ \wp_I$, $S=S_I\circ \wp_I$ and $\sigma=\sigma_I\circ \wp_I$.
  \item If $I\subseteq I_{(S,T,\sigma,B)}$, then the covariant representation
    $(S_I,T_I,\sigma_I,B)$ is injective if and only if $I=I_{(S,T,\sigma,B)}$.
  \item If $I\subseteq I_{(S,T,\sigma,B)}$, then the covariant representation
    $(S_I,T_I,\sigma_I,B)$ is surjective and graded if and only if $(S,T,\sigma,B)$ is.
  \item If $I\subseteq I_{(S,T,\sigma,B)}$ and $(I,J)$ is a $T$-pair of $(P,Q,\psi)$, then the
    covariant representation $(S_I,T_I,\sigma_I,B)$ is Cuntz-Pimsner
    invariant relative to $J_I$ if and only if $J\subseteq J_{(S,T,\sigma,B)}$.
\end{enumerate}
\end{lem}

\begin{proof}
  If there is a covariant representation
  $(S_I,T_I,\sigma_I,B)$ of $({}_IP,Q_I,\psi_I)$ such that
  $T=T_I\circ \wp_I$, $S=S_I\circ \wp_I$ and $\sigma=\sigma_I\circ \wp_I$,
  then $I\subseteq I_{(S,T,\sigma,B)}$.

  Assume now that $I\subseteq I_{(S,T,\sigma,B)}$. It follows from
  Lemma \ref{lemma_31} that
  we can define maps $\sigma_I:R_I\longrightarrow B$ by letting
  $\sigma_I(r+I)=\sigma(r)$ for every $r\in R$, $T_I:Q_I\longrightarrow
  B$ by letting $T_I(q+QI)=T(q)$ for every $q\in Q$ and
  $S_I:{}_IP\longrightarrow B$ by letting $S_I(p+IP)=S(p)$ for every $p\in
  P$. it is then clear that $(S_I,T_I,\sigma_I,B)$ is a covariant representation
  of $({}_IP,Q_I,\psi_I)$ satisfying
  $T=T_I\circ \wp_I$, $S=S_I\circ \wp_I$ and $\sigma=\sigma_I\circ
  \wp_I$. It is also clear that $(S_I,T_I,\sigma_I,B)$ is the unique
  covariant representation
  of $({}_IP,Q_I,\psi_I)$ with this property. Finally it is straight
  forward to check that $(S_I,T_I,\sigma_I,B)$ is injective if and
  only if $I=I_{(S,T,\sigma,B)}$, that $(S_I,T_I,\sigma_I,B)$ is
  surjective and graded if and only if $(S,T,\sigma,B)$ is, and that
  $(S_I,T_I,\sigma_I,B)$ is Cuntz-Pimsner
    invariant relative to $J_I$ if and only if $J\subseteq J_{(S,T,\sigma,B)}$.
\end{proof}

\begin{theor}\label{theor_1}
  Let $R$ be a ring and $(P,Q,\psi)$ an $R$-system that satisfies
  condition \textbf{(FS)}. Let $(S,T,\sigma,B)$ be a covariant
  representation of and let $\omega=(I,J)$ be a $T$-pair of
  $(P,Q,\psi)$.
  Then we have:
  \begin{enumerate}
  \item If there is a ring homomorphism
    $\eta:\mathcal{O}_{({}_IP,Q_I,\psi_I)}(J_I)\longrightarrow
    B$ such that
    $\eta\circ\iota^\omega_R=\sigma$,
    $\eta\circ \iota^\omega_Q=T$ and
    $\eta\circ\iota^\omega_P=S$, then
    $I\subseteq I_{(S,T,\sigma,B)}$ and $J\subseteq
    J_{(S,T,\sigma,B)}$.
  \item \label{item:18} If $I\subseteq I_{(S,T,\sigma,B)}$ and $J\subseteq
    J_{(S,T,\sigma,B)}$, then there exists a unique
    ring homomorphism
    $\eta_{(S,T,\sigma,B)}^\omega:\mathcal{O}_{({}_IP,Q_I,\psi_I)}(J_I)\longrightarrow
    B$ such that
    $\eta_{(S,T,\sigma,B)}^\omega\circ\iota^\omega_R=\sigma$,
    $\eta_{(S,T,\sigma,B)}^\omega\circ \iota^\omega_Q=T$ and
    $\eta_{(S,T,\sigma,B)}^\omega\circ\iota^\omega_P=S$.
  \item If $I\subseteq I_{(S,T,\sigma,B)}$ and $J\subseteq
    J_{(S,T,\sigma,B)}$, then $\eta_{(S,T,\sigma,B)}^\omega$ is an isomorphism if and only if
    $(S,T,\sigma,B)$ is a surjective and graded representation and
    $\omega=\omega_{(S,T,\sigma,B)}$.
  \end{enumerate}
\end{theor}

\begin{proof}
  It is easy to check that if there exists a
  ring homomorphism
  $\eta:\mathcal{O}_{({}_IP,Q_I,\psi_I)}(J_I)\longrightarrow
    B$ such that
    $\eta\circ\iota^\omega_R=\sigma$,
    $\eta\circ \iota^\omega_Q=T$ and
    $\eta\circ\iota^\omega_P=S$, then
  $I\subseteq I_{(S,T,\sigma,B)}$ and $J\subseteq J_{(S,T,\sigma,B)}$.

  Assume now that $I\subseteq I_{(S,T,\sigma,B)}$ and $J\subseteq
  J_{(S,T,\sigma,B)}$. It follows from Lemma \ref{lemma_33} that there
  exists a covariant representation
  $(S_I,T_I,\sigma_I,B)$ of $({}_IP,Q_I,\psi_I)$ which is Cuntz-Pimsner
  invariant relative to $J_I$ such that
  $T=T_I\circ \wp_I$, $S=S_I\circ \wp_I$ and $\sigma=\sigma_I\circ
  \wp_I$. It then follows from Theorem \ref{univ_cuntz} that there exists
  a ring homomorphism $\eta_{(S,T,\sigma,B)}^\omega:\mathcal{O}_{({}_IP,Q_I,\psi_I)}(J_I)\longrightarrow
  B$ such that $\eta_{(S,T,\sigma,B)}^\omega\circ \iota^{J_I}_{R_I}=\sigma_I$,
  $\eta_{(S,T,\sigma,B)}^\omega\circ \iota^{J_I}_{Q_I}=T_I$ and
  $\eta_{(S,T,\sigma,B)}^\omega\circ \iota^{J_I}_{{}_IP}=S_I$. It
  follows that $\eta_{(S,T,\sigma,B)}^\omega\circ\iota^\omega_R=\sigma$,
  $\eta_{(S,T,\sigma,B)}^\omega\circ \iota^\omega_Q=T$ and
  $\eta_{(S,T,\sigma,B)}^\omega\circ\iota^\omega_P=S$. Since
  $\mathcal{O}_{({}_IP,Q_I,\psi_I)}(J_I)$ is generated by
  $\iota_R^\omega(R)$, $\iota_Q^\omega(Q)$ and $\iota_P^\omega(P)$,
  the uniqueness of $\eta_{(S,T,\sigma,B)}^\omega$ follows.

  Assume that $\eta_{(S,T,\sigma,B)}^\omega$ is an isomorphism. Then
  $(S,T,\sigma,B)$ is a surjective and graded, and
  $\omega_{(\iota_P^\omega,\iota_Q^\omega,\iota_R^\omega,\mathcal{O}_{({}_IP,Q_I,\psi_I)}(J_I))}
  =\omega_{(S,T,\sigma,B)}$. It therefore follows from
  Proposition \ref{prop_32} that
  $\omega=\omega_{(S,T,\sigma,B)}$.

  Finally assume that $(S,T,\sigma,B)$ is surjective and graded, and
  that $\omega=\omega_{(S,T,\sigma,B)}$. Then it follows from Lemma
  \ref{lemma_33} that $(S_I,T_I,\sigma_I,B)$ is surjective, injective
  and graded, and it is easy to check that
  $J_I=\wp_I(J_{(S,T,\sigma,B)})=J_{(S_I,T_I,\sigma_I,B)}$, and hence from Theorem
  \ref{prop:psi} we get that $\eta_{(S,T,\sigma,B)}^\omega$ is an isomorphism.
\end{proof}

We now have the promised classification of all surjective and graded covariant representations of a given $R$-system satisfying condition \textbf{(FS)}.

\begin{rema} \label{remark:classification-noninjective}
  Let $R$ be a ring and let $(P,Q,\psi)$ be an $R$-system satisfying
  condition \textbf{(FS)}. Then it follows from Proposition
  \ref{prop_31} and Theorem \ref{theor_1} that every surjective and
  graded covariant of $(P,Q,\psi)$ is isomorphic to
  $(\iota_P^\omega,\iota_Q^\omega,\iota_R^\omega,\mathcal{O}_{({}_IP,Q_I,\psi_I)}(J_I))$
  for some $T$-pair $\omega=(I,J)$ of $(P,Q,\psi)$. It also follows
  that if $\omega_1=(I_1,J_1)$ and $\omega_2=(I_2,J_2)$ are two
  $T$-pairs of $(P,Q,\psi)$, then there is a ring homomorphism
  $\phi:\mathcal{O}_{({}_{I_1}P,Q_{I_1},\psi_{I_1})}((J_1)_{I_1})\longrightarrow
  \mathcal{O}_{({}_{I_2}P,Q_{I_2},\psi_{I_2})}((J_2)_{I_2})$ such that
  $\phi\circ\iota^{\omega_1}_R=\iota^{\omega_2}_R$,
  $\phi\circ\iota^{\omega_1}_P=\iota^{\omega_2}_P$ and
  $\phi\circ\iota^{\omega_1}_Q=\iota^{\omega_2}_Q$ if and only if
  $I_1\subseteq I$ and $J_1\subseteq J_2$.
\end{rema}

Let $R$ be a ring and let $(P,Q,\psi)$ be an $R$-system satisfying
condition \textbf{(FS)}. If $(I,J)$ is a pair of two-sided ideals of $R$ such that $I\subseteq J$, the ideal $I$ is $\psi$-invariant and $\wp_I(J)\subseteq \Delta^{-1}_I(\mathcal{F}_{{}_IP}(Q_I))$, then $(\iota^{J_I}_{P_I}\circ \wp_I,\iota^{J_I}_{Q_I}\circ \wp_I,\iota^{J_I}_{R_I}\circ \wp_I,\mathcal{O}_{({}_IP,Q_I,\psi_I)}(J_I))$ is a surjective and graded covariant representation of $(P,Q,\psi)$, even though $\wp_I(J)\cap \ker\Delta_I\ne 0$, and it then follows from the previous remark that this representation is isomorphic to $(\iota_P^{\omega'},\iota_Q^{\omega'},\iota_R^{\omega'},\mathcal{O}_{({}_{I'}P,Q_{I'},\psi_{I'})}(J'_{I'}))$ for some $T$-pair $\omega'=(I',J')$. We will now describe this $T$-pair in terms of the pair $(I,J)$. We will begin with the case where $I=\{0\}$, but first a lemma:

\begin{lem} \label{lemma:compact-expand}
  Let $R$ be a ring and let $(P,Q,\psi)$ be an $R$-system satisfying
  condition \textbf{(FS)}. If $x\in\Delta^{-1}(\mathcal{F}_P(Q))$, $q\in Q^{\otimes n}$ and
  $p\in P^{\otimes n}$, then
  $\theta_{qx,p}\otimes 1_Q\in\mathcal{F}_{P^{\otimes n+1}}(Q^{\otimes n+1})$ and
  \begin{equation} \label{eq:7}
    \pi(\theta_{qx,p}\otimes 1_Q)=\iota_Q^n(q)\pi(\Delta(x))\iota_P^n(p).
  \end{equation}
\end{lem}

\begin{proof}
  Choose $q_1,q_2,\dots,q_k\in Q$ and $p_1,p_2,\dots,p_k\in P$ such that
 $\Delta(x)=\sum_{i=1}^k\theta_{q_i,p_i}$. Then we have for
 $q^n\in Q^{\otimes n}$ and $q^1\in Q$ that
 \begin{equation*}
   \begin{split}
     \theta_{qx,p}\otimes 1_Q(q^{\otimes n}\otimes q^1)
     &=q\otimes x\psi_n(p\otimes q^n)q^1
     =\sum_{i=1}^kq\otimes q_i\psi\bigl(p_i\otimes\psi_n(p\otimes q^n)q^1\bigr)\\
     &=\sum_{i=1}^kq\otimes q_i\psi_{n+1}\bigl((p_i\otimes p)\otimes (q^n\otimes q^1)\bigr)
     =\sum_{i=1}^k\theta_{q\otimes q_i,p_i\otimes p}(q^n\otimes q^1).
   \end{split}
 \end{equation*}
 It follows that $\theta_{qx,p}\otimes 1_Q=\sum_{i=1}^k\theta_{q\otimes q_i,p_i\otimes p}
 \in\mathcal{F}_{P^{\otimes n+1}}(Q^{\otimes n+1})$ and that
 \begin{equation*}
   \pi(\theta_{qx,p}\otimes 1_Q)=\sum_{i=1}^k\iota_Q^n(q)\iota_Q(q_i)\iota_P(p_i)\iota_P^n(p)
   =\iota_Q^n(q)\pi(\Delta(x))\iota_P^n(p).
 \end{equation*}
\end{proof}

Let $R$ be a ring and let $(P,Q,\psi)$ be an $R$-system. For every
$x\in R$ we define $\Delta^n(x)\in\mathcal{L}_{P^{\otimes
    n}}(Q^{\otimes n})$ inductively by letting $\Delta^1(x)=\Delta(x)$
and $\Delta^n(x)=\Delta^{n-1}(x)\otimes 1_Q$ for $n\ge 2$.

\begin{lem} \label{lemma:newt}
  Let $R$ be a ring, let $(P,Q,\psi)$ be an $R$-system satisfying
  condition \textbf{(FS)}, and let $J$ be a two-sided ideal of $R$
  such that $J\subseteq \Delta^{-1}(\mathcal{F}_P(Q))$. If we let
  $$I=\{x\in J\mid \forall
    m\in\N:\Delta^m(x)(Q^{\otimes m})\subseteq Q^{\otimes m}J
    \land \exists n\in\N:\Delta^n(x)=0\},$$
  then $I=I_{(\iota_P^J,\iota_Q^J,\iota_R^J,\mathcal{O}_{(P,Q,\psi)}(J))}$
  and $J=J_{(\iota_P^J,\iota_Q^J,\iota_R^J,\mathcal{O}_{(P,Q,\psi)}(J))}$.
\end{lem}

\begin{proof}
  Let $x\in
  I_{(\iota_P^J,\iota_Q^J,\iota_R^J,\mathcal{O}_{(P,Q,\psi)}(J))}$.
  Then $\iota_R(x)\in\mathcal{T}(J)$. It follows from Lemma
  \ref{lemma_5} that
  $\iota_R(x)=\projection_{(0,0)}(\iota_R(x))\in\iota_R(J)$ and that
  there is an $n\in\N$ such that
  $\iota_Q^n(xq)=\iota_R(x)\iota_Q^n(q)=0$ for every $q\in Q^{\otimes
    n}$. Since $\iota_R$ and $\iota_Q^n$ are injective (cf. Theorem
  \ref{theor:toeplitz} and Lemma \ref{lemma:inj}) it follows that
  $x\in J$ and that $\Delta^n(x)=0$. It also follows from Lemma
  \ref{lemma_5} that
  \begin{equation*}
    \iota_R(x)=\iota_R(x)-\pi(\Delta(x))+\sum_{i=1}^{n-1}\sum_{j=1}^{m_i}
    \iota_Q^i(q_j^i)\bigl(\iota_R(x_j^i)-\pi(\Delta(x_j^i))\bigr)\iota_P(p_j^i)
  \end{equation*}
  for some $x_j^i\in J$, $q_j^i\in Q^{\otimes i}$, $p_j^i\in
  P^{\otimes i}$. We will by induction show that
  \begin{equation} \label{eq:2}
    \Delta^i(x)=\sum_{j=1}^{m_i}\theta_{q_j^ix_j^i,p_j^i}
  \end{equation}
  for every $i\in\{1,2,\dots,n-1\}$.
  It will then follow that
  $\Delta^i(x)(q)=\sum_{j=1}^{m_i}q_j^ix_j^i\psi_j(p_j^i\otimes q)\in
  Q^{\otimes i}J$ for every $i\in\{1,2,\dots,n-1\}$ and every $q\in
  Q^{\otimes i}$, and thus that $x\in I$.

  For $i=1$ we have
  \begin{equation*}
    0=\projection_{(1,1)}(\iota_R(x))=-\pi(\Delta(x))+\sum_{j=1}^{m_1}
    \iota_Q(q_j^1)\iota_R(x_j^1)\iota_P(p_j^1).
  \end{equation*}
  Thus we have
  \begin{equation*}
    \pi(\Delta(x))=\sum_{j=1}^{m_1}
    \iota_Q(q_j^1)\iota_R(x_j^1)\iota_P(p_j^1)
    =\pi\left(\sum_{j=1}^{m_1}\theta_{q_j^1x_j^1,p_j^1}\right)
  \end{equation*}
  and since $\pi$ is injective (cf. Proposition \ref{lemma_2}), it
  follows that Equation \eqref{eq:2} holds for $i=1$.

  Let $k\in\{1,2,\dots,n-2\}$ and assume that Equation \eqref{eq:2} holds for
  $i=k$. We have that
   \begin{equation*}
    0=\projection_{(k+1,k+1)}(\iota_R(x))
    =-\sum_{j=1}^{m_k}\iota_q^k(q_j^k)\pi(\Delta(x_j^k))
    \iota_P^i(p_j^k)+\sum_{j=1}^{m_{k+1}}
    \iota_Q(q_j^{k+1})\iota_R(x_j^{k+1})\iota_P(p_j^{k+1}).
  \end{equation*}
  It follows that if $q_k\in Q^{\otimes k}$ and $q_1\in Q$, then we
  have that
  \begin{equation*}
    \begin{split}
      \iota_Q^{k+1}(\Delta^{k+1}(x)(q_k\otimes q_1)
      &=\pi(\Delta^k(x))\iota_Q^k(q_k)\iota_Q(q_1)
      =\sum_{j=1}^{m_k}\iota_Q(q_j^k)\iota_R(x_j^k)
      \iota_P^k(p_j^k)\iota_Q^k(q_k)\iota_Q(q_1)\\
      &=\sum_{j=1}^{m_k}\iota_Q^k(q_j^k)\pi(\Delta(x_j^k))\iota_P^k(p_j^k)
      \iota_Q^k(q_k)\iota_Q(q_1)\\
      &=\sum_{j=1}^{m_{k+1}}\iota_Q^{k+1}(q_j^{k+1})\iota_R(x_j^{k+1})
      \iota_P^{k+1}(p_j^{k+1})\iota_Q(q_k)\iota_Q(q_1)\\
      &=\iota_Q^{k+1}\left(\sum_{j=1}^{m_{k+1}}
        \theta_{q_j^{k+1}x_j^{k+1},p_j^{k+1}}(q_k\otimes q_1)\right),
    \end{split}
  \end{equation*}
  and since $\iota_Q^{k+1}$ is injective and $Q^{k+1}=\spa\{q_k\otimes
  q_1\mid q_k\in Q^{\otimes k},\ q_1\in Q\}$, it follows that Equation
  \eqref{eq:2} holds for $i=k+1$. Hence Equation \eqref{eq:2} holds for every
  $i\in\{1,2,\dots,n-1\}$. We have thus proved that
  $I_{(\iota_P^J,\iota_Q^J,\iota_R^J,\mathcal{O}_{(P,Q,\psi)}(J))}
  \subseteq I$.

  Let $x\in J$ and assume that $\Delta^m(x)(Q^{\otimes m})\subseteq
  Q^{\otimes m}J$ for all $m\in\N$ and that there is an $n\in\N$ such
  that $\Delta^n(x)=0$. We will by induction show that there for every
  $i\in\{1,2,\dots,n-1\}$ exist $x_j^i\in J$, $q_j^i\in Q^{\otimes
    i}$, $p_j^i\in P^{\otimes i}$ such that
  \begin{equation} \label{eq:5}
    \Delta^i(x)=\sum_{j=1}^{m_i}\theta_{q_j^ix_j^i,p_j^i}
  \end{equation}
  and such that $\Delta^{i+1}(x)\in\mathcal{F}_{P^{\otimes i+1}}(Q^{\otimes i+1})$ and
  \begin{equation}
    \label{eq:6}
    \pi(\Delta^{i+1}(x))=\sum_{j=1}^{m_i}\iota_Q^i(q_j^i)
    \pi(\Delta(x_j^i))\iota_P^i(p_j^i).
  \end{equation}
  It will then follow that we have
  \begin{equation*}
    \iota_R(x)=\iota_R(x)-\pi(\Delta(x))
    +\sum_{i=1}^{n-1}\sum_{j=1}^{m_i}\iota_Q^i(q_j^i)
    \bigl(\iota_R(x_j^i)-\pi(\Delta(x_j^i))\bigr)\iota_P^i(p_j^i)
    \in\mathcal{T}(J),
  \end{equation*}
  and thus that $x\in
  I_{(\iota_P^J,\iota_Q^J,\iota_R^J,\mathcal{O}_{(P,Q,\psi)}(J))}$.

  Choose $q_1,q_2,\dots,q_k\in Q$, $p_1,p_2,\dots,p_k\in P$ such that
  $\Delta(x)=\sum_{j=1}^k\theta_{q_j,p_j}$. It follows from condition
  \textbf{(FS)} that there exist $q'_1,q'_2,\dots,q'_h\in Q$ and
  $p'_1,p'_2,\dots,p'_h\in P$ such that
  $\sum_{l=1}^h\theta_{p'_l,q'_l}(p_j)=p_j$ for every
  $j\in\{1,2,\dots,k\}$. We then have that
  \begin{equation*}
    \Delta(x)=\sum_{j=1}^k\theta_{q_j,p_j}
    =\sum_{j=1}^k\theta_{q_j,\sum_{l=1}^h\psi(p_j\otimes q'_l)p'_l}
    =\sum_{l=1}^h\theta_{\Delta(x)q'_l,p'_l}.
  \end{equation*}
  Since $\Delta(x)q'_l\in QJ$ for each $l\in\{1,2,\dots,h\}$, it
  follows that there exist $x_j^1\in J$, $q_j^1\in Q$,
  $p_j^1\in P$ such that Equation \eqref{eq:5} holds for $i=1$. It
  then follows from Lemma \ref{lemma:compact-expand} that also Equation
  \eqref{eq:6} holds for $i=1$.

  Assume then that $k\in\{1,2,\dots,n-1\}$ and that there exist
  $x_j^k\in J$, $q_j^k\in Q^{\otimes k}$, $p_j^k\in P^{\otimes k}$ such
  that Equation \eqref{eq:6} holds for $i=k$. For each $j\in\{1,2,\dots,m_k\}$
  choose $q_{(j,1)}, q_{(j,2)},\dots,q_{(j,n_j)}\in Q$ and $p_{(j,1)}, p_{(j,2)},\dots,p_{(j,n_j)}\in P$
  such that $\Delta(x_j^k)=\sum_{h=1}^{n_j}\theta_{q_{(j,h)},p_{(j,h)}}$. If $q^k\in Q^{\otimes k}$
  and $q^1\in Q$, then we have
  \begin{equation*}
    \begin{split}
      \iota^{k+1}_Q\bigl(\Delta^{k+1}(x)(q^k\otimes q^1)\bigr)
      &=\sum_{j=1}^{m_k}\iota_Q^k(q_j^k)\iota_R(x_j^k)
      \iota_P^k(p_j^k)\iota_Q^k(q_k)\iota_Q(q^1)\\
      &=\sum_{j=1}^{m_k}\iota_Q^k(q_j^k)\iota_Q\bigl(x_j^k
      \psi(p_j^k\otimes q_k)q^1\bigr)\\
      &=\sum_{j=1}^{m_k}\iota_Q^k(q_j^k)\pi(\Delta(x_j^k))
      \iota_Q\bigl(\psi(p_j^k\otimes q_k)q^1\bigr)\\
      &=\sum_{j=1}^{m_k}\iota_Q^k(q_j^k)
      \left(\sum_{h=1}^{n_j}\iota_Q(q_{(j,h)})\iota_P(p_{(j,h)})\right)
      \iota_Q\bigl(\psi(p_j^k\otimes q_k)q^1\bigr)\\
      &=\iota_Q^{k+1}\left(\sum_{j=1}^{m_k}\sum_{h=1}^{n_j}
        \theta_{q_j^k\otimes q_{(j,h)}, p_{(j,h)}\otimes p_j^k}(q^k\otimes q^1)\right).
    \end{split}
  \end{equation*}
  It follows that $\Delta^{k+1}(x)=\sum_{j=1}^{m_k}\sum_{h=1}^{n_j}
  \theta_{q_j^k\otimes q_{(j,h)}, p_{(j,h)}\otimes p_j^k}$.
  By  condition
  \textbf{(FS)} there exist $q'_1,q'_2,\dots,q'_r\in Q^{\otimes k+1}$ and
  $p'_1,p'_2,\dots,p'_r\in P^{\otimes k+1}$ such that
  $\sum_{l=1}^r\theta_{p'_l,q'_l}(p_{(j,h)}\otimes p_j^k)=p_{(j,h)}\otimes p_j^k$ for every
  $j\in\{1,2,\dots,m_k\}$ and every $h\in\{1,2,\dots,n_j\}$. We then have that
  \begin{equation*}
    \Delta^{k+1}(x)=\sum_{j=1}^{m_k}\sum_{h=1}^{n_j}
    \theta_{q_j^k\otimes q_{(j,h)}, p_{(j,h)}\otimes p_j^k}
    =\sum_{j=1}^{m_k}\sum_{h=1}^{n_j}
    \theta_{q_j^k\otimes q_{(j,h)}, \sum_{l=1}^r\psi_{k+1}((p_{(j,h)}\otimes p_j^k)\otimes q'_l)p'_l}
    =\sum_{l=1}^r\theta_{\Delta^{k+1}(x)q'_l,p'_l}.
  \end{equation*}
  Since $\Delta^{k+1}(x)q'_l\in Q^{k+1}J$ for each $l\in\{1,2,\dots,r\}$, it
  follows that there exist $x_j^{k+1}\in J$, $q_j^{k+1}\in Q^{\otimes k+1}$,
  $p_j^{k+1}\in P^{\otimes k+1}$ such that Equation \eqref{eq:5} holds for $i=k+1$. It
  then follows from Lemma \ref{lemma:compact-expand} that also
  Equation \eqref{eq:6} holds for $i=k+1$.

  Thus there exist for every $i\in\{1,2,\dots,n-1\}$ elements
  $x_j^i\in J$, $q_j^i\in Q^{\otimes  i}$, $p_j^i\in P^{\otimes i}$ such that
  Equation \eqref{eq:5} and \eqref{eq:6} hold, and
  $x\in I_{(\iota_P^J,\iota_Q^J,\iota_R^J,\mathcal{O}_{(P,Q,\psi)}(J))}$. This shows that
  $I\subseteq I_{(\iota_P^J,\iota_Q^J,\iota_R^J,\mathcal{O}_{(P,Q,\psi)}(J))}$, and so we have proved that
  $I= I_{(\iota_P^J,\iota_Q^J,\iota_R^J,\mathcal{O}_{(P,Q,\psi)}(J))}$.

  We will now show that $J=J_{(\iota_P^J,\iota_Q^J,\iota_R^J,\mathcal{O}_{(P,Q,\psi)}(J))}$.
  If $x\in J$, then $\iota_R(x)-\pi(\Delta(x))\in\mathcal{T}(J)$, so
  $\iota_R^J(x)=\pi^J(\Delta(x))$ and $x\in J_{(\iota_P^J,\iota_Q^J,\iota_R^J,\mathcal{O}_{(P,Q,\psi)}(J))}$.
  In the other direction, if $x\in J_{(\iota_P^J,\iota_Q^J,\iota_R^J,\mathcal{O}_{(P,Q,\psi)}(J))}$,
  then it follows from Lemma \ref{lemma:compact} that $\iota_R^J(x)=\pi^J(\Delta(x))$
  and so $\iota_R(x)-\pi(\Delta(x))\in\mathcal{T}(J)$. It then follows from
  Lemma \ref{lemma_5} that $\iota_R(x)=\projection_{(0,0)}(\iota_R(x)-\pi(\Delta(x)))\in
  \iota_R(J)$, and since $\iota_R$ is injective, we have $x\in J$.
\end{proof}

\begin{prop} \label{prop:newt}
  Let $R$ be a ring, let $(P,Q,\psi)$ be an $R$-system satisfying
  condition \textbf{(FS)}.
  Let $(I,J)$ be a pair of two-sided ideals of $R$ such that $I\subseteq J$,
  the ideal $I$ is $\psi$-invariant and
  $\wp_I(J)\subseteq \Delta^{-1}_I(\mathcal{F}_{{}_IP}(Q_I))$.  If we let
  $$I'=\{x\in J\mid \forall
    m\in\N:\Delta_I^m(x)(Q_I^{\otimes m})\subseteq Q_I^{\otimes m}J_I
    \land \exists n\in\N:\Delta_I^n(x)=0\},$$
  then $I'=I_{(\iota^{J_I}_{P_I}\circ \wp_I,\iota^{J_I}_{Q_I}\circ \wp_I,\iota^{J_I}_{R_I}\circ \wp_I,\mathcal{O}_{({}_IP,Q_I,\psi_I)}(J_I))}$
  and $J=J_{(\iota^{J_I}_{P_I}\circ \wp_I,\iota^{J_I}_{Q_I}\circ \wp_I,\iota^{J_I}_{R_I}\circ \wp_I,\mathcal{O}_{({}_IP,Q_I,\psi_I)}(J_I))}$.
\end{prop}

\begin{proof}
  It is clear that we have
  $$I_{(\iota^{J_I}_{P_I}\circ \wp_I,\iota^{J_I}_{Q_I}\circ \wp_I,\iota^{J_I}_{R_I}\circ \wp_I,\mathcal{O}_{({}_IP,Q_I,\psi_I)}(J_I))}
  =\wp_I^{-1}(I_{(\iota_{{}_IP}^{J_I},\iota_{Q_I}^{J_I},\iota_{R_I}^{J_I},
    \mathcal{O}_{({}_IP,Q_I,\psi_I)}(J_I))})$$
  $$J_{(\iota^{J_I}_{P_I}\circ \wp_I,\iota^{J_I}_{Q_I}\circ \wp_I,\iota^{J_I}_{R_I}\circ \wp_I,\mathcal{O}_{({}_IP,Q_I,\psi_I)}(J_I))}
  =\wp_I^{-1}(J_{(\iota_{{}_IP}^{J_I},\iota_{Q_I}^{J_I},\iota_{R_I}^{J_I},
    \mathcal{O}_{({}_IP,Q_I,\psi_I)}(J_I))}),$$
  and the result then follows from Lemma \ref{lemma:newt}.
\end{proof}

\subsection{Products and coproducts in $\category$}
\label{sec:prod-copr-categ}

We will show that if $R$ is a ring and $(P,Q,\psi)$ is an $R$-system, then $\category$ has products and coproducts, and we will, in the case where $(P,Q,\psi)$ satisfies condition \textbf{(FS)}, show how the product and coproduct are related to $T$-pairs of $(P,Q,\psi)$.

\begin{prop} \label{prop:diamond}
  Let $R$ be a ring, let $(P,Q,\psi)$ be an $R$-system and let
  $((S_\lambda,T_\lambda,\sigma_\lambda,B_\lambda))_{\lambda\in \Lambda}$
  be a family of surjective covariant representations of $(P,Q,\psi)$.

  Then the product of
  $((S_\lambda,T_\lambda,\sigma_\lambda,B_\lambda))_{\lambda\in \Lambda}$
  in $\category$ exists; i.e., there exists a surjective covariant
  representation
  $(S_{\prod_{\lambda\in \Lambda}(S_\lambda,T_\lambda,\sigma_\lambda,B_\lambda)},
  T_{\prod_{\lambda\in \Lambda}(S_\lambda,T_\lambda,\sigma_\lambda,B_\lambda)},
  \sigma_{\prod_{\lambda\in \Lambda}(S_\lambda,T_\lambda,\sigma_\lambda,B_\lambda)},
  B_{\prod_{\lambda\in \Lambda}(S_\lambda,T_\lambda,\sigma_\lambda,B_\lambda)})$
  of $(P,Q,\psi)$ and a family $(\phi_\lambda)_{\lambda\in \Lambda}$ of
  ring homomorphisms
  $\phi_\lambda:B_{\prod_{j\in \Lambda}(S_j,T_j,\sigma_j,B_j)}\longrightarrow B_\lambda$ satisfying
  $\phi_\lambda\circ S_{\prod_{j\in \Lambda}(S_j,T_j,\sigma_j,B_j)}=S_\lambda$,
  $\phi_\lambda\circ T_{\prod_{j\in \Lambda}(S_j,T_j,\sigma_j,B_j)}=T_\lambda$ and
  $\phi_\lambda\circ \sigma_{\prod_{j\in \Lambda}(S_j,T_j,\sigma_j,B_j)}=\sigma_\lambda$ for all
  $\lambda\in \Lambda$, with the following property:
  \begin{enumerate}[label=\textbf{(PR)}]
  \item If $(S,T,\sigma,B)$ is a surjective covariant representation
    of $(P,Q,\psi)$ and there for each $\lambda\in \Lambda$ exists a ring homomorphism
    $\psi_\lambda:B\longrightarrow B_\lambda$ such that
    $\psi_\lambda\circ T=T_\lambda$, $\psi_\lambda\circ
    S=S_\lambda$ and $\psi_\lambda\circ\sigma=\sigma_\lambda$, then
    there exists a unique ring homomorphism
    $\tau:B\longrightarrow B_{\prod_{\lambda\in \Lambda}(S_\lambda,T_\lambda,\sigma_\lambda,B_\lambda)}$ such that
    $\tau\circ S=S_{\prod_{\lambda\in \Lambda}(S_\lambda,T_\lambda,\sigma_\lambda,B_\lambda)}$,
    $\tau\circ T=T_{\prod_{\lambda\in \Lambda}(S_\lambda,T_\lambda,\sigma_\lambda,B_\lambda)}$ and
    $\tau\circ \sigma
    =\sigma_{\prod_{\lambda\in \Lambda}(S_\lambda,T_\lambda,\sigma_\lambda,B_\lambda)}$, and such that
    $\phi_\lambda\circ\tau=\psi_\lambda$ for each $\lambda\in \Lambda$. \label{item:12}
  \end{enumerate}
  We furthermore have that the surjective covariant representation
  $$(S_{\prod_{\lambda\in \Lambda}(S_\lambda,T_\lambda,\sigma_\lambda,B_\lambda)},
  T_{\prod_{\lambda\in \Lambda}(S_\lambda,T_\lambda,\sigma_\lambda,B_\lambda)},
  \sigma_{\prod_{\lambda\in \Lambda}(S_\lambda,T_\lambda,\sigma_\lambda,B_\lambda)},
  B_{\prod_{\lambda\in \Lambda}(S_\lambda,T_\lambda,\sigma_\lambda,B_\lambda)})$$
  and the family $(\phi_\lambda)_{\lambda\in \Lambda}$
  are, up to isomorphism, the unique pair which possesses property \ref{item:12}; in fact if
  $(S,T,\sigma,B)$ is a surjective covariant representation
  of $(P,Q,\psi)$ and $(\psi_\lambda)_{\lambda\in \Lambda}$ is a family of  ring
  homomorphisms $\psi_\lambda:B\longrightarrow B_\lambda$ satisfying
  $\psi_\lambda\circ S=S_\lambda$, $\psi_\lambda\circ T=T_\lambda$ and
  $\psi_\lambda\circ\sigma=\sigma_\lambda$ for each $\lambda\in \Lambda$, and
  $\varphi:B_{\prod_{\lambda\in \Lambda}(S_\lambda,T_\lambda,\sigma_\lambda,B_\lambda)}\longrightarrow B$ is a ring
  homomorphism such that
  $\varphi\circ S_{\prod_{\lambda\in \Lambda}(S_\lambda,T_\lambda,\sigma_\lambda,B_\lambda)}=S$,
  $\varphi\circ T_{\prod_{\lambda\in \Lambda}(S_\lambda,T_\lambda,\sigma_\lambda,B_\lambda)}=T$ and
  $\varphi\circ \sigma_{\prod_{\lambda\in \Lambda}(S_\lambda,T_\lambda,\sigma_\lambda,B_\lambda)}
  =\sigma$,
  then $\varphi$ is an isomorphism.

  Moreover, $x\in B_{\prod_{\lambda\in \Lambda}(S_\lambda,T_\lambda,\sigma_\lambda,B_\lambda)}$
  is zero if and only if $\phi_\lambda(x)=0$ for all $\lambda\in \Lambda$.
\end{prop}

\begin{proof}
  Let $H=\cap_{\lambda\in \Lambda}\ker\eta_{(S_\lambda,T_\lambda,\sigma_\lambda,B_\lambda)}$
  where for each $\lambda\in \Lambda$ the homomorphism
  $\eta_{(S_\lambda,T_\lambda,\sigma_\lambda,B_\lambda)}:
  \mathcal{T}_{(P,Q,\psi)}\longrightarrow B_\lambda$ is
  the homomorphism given by Theorem \ref{theor:toeplitz}. If the
  family $((S_\lambda,T_\lambda,\sigma_\lambda,B_\lambda))_{\lambda\in \Lambda}$
  is empty, then we let
  $H=\mathcal{T}_{(P,Q,\psi)}$. Let
  $\wp_H:\mathcal{T}_{(P,Q,\psi)}\longrightarrow\mathcal{T}_{(P,Q,\psi)}/H$ be the
  corresponding quotient map, and let
  $S_{\prod_{\lambda\in \Lambda}(S_\lambda,T_\lambda,\sigma_\lambda,B_\lambda)}=\wp_H\circ \iota_P$,
  $T_{\prod_{\lambda\in \Lambda}(S_\lambda,T_\lambda,\sigma_\lambda,B_\lambda)}=\wp_H\circ \iota_Q$,
  $\sigma_{\prod_{\lambda\in \Lambda}(S_\lambda,T_\lambda,\sigma_\lambda,B_\lambda)}
  =\wp_H\circ \iota_R$ and
  $B_{\prod_{\lambda\in \Lambda}(S_\lambda,T_\lambda,\sigma_\lambda,B_\lambda)}
  =\mathcal{T}_{(P,Q,\psi)}/H$.
  We then have that $$(S_{\prod_{\lambda\in \Lambda}(S_\lambda,T_\lambda,\sigma_\lambda,B_\lambda)},
  T_{\prod_{\lambda\in \Lambda}(S_\lambda,T_\lambda,\sigma_\lambda,B_\lambda)},
  \sigma_{\prod_{\lambda\in \Lambda}(S_\lambda,T_\lambda,\sigma_\lambda,B_\lambda)},
  B_{\prod_{\lambda\in\Lambda}(S_\lambda,T_\lambda,\sigma_\lambda,B_\lambda)})$$
  is a surjective covariant representation of $(P,Q,\psi)$. We also have that there for each
  $\lambda\in \Lambda$ is a ring homomorphism
  $\phi_\lambda:B_{\prod_{j\in \Lambda}(S_j,T_j,\sigma_j,B_j)}\longrightarrow B_\lambda$ satisfying
  $\phi_\lambda\circ S_{\prod_{j\in \Lambda}(S_j,T_j,\sigma_j,B_j)}=S_\lambda$,
  $\phi_\lambda\circ T_{\prod_{j\in \Lambda}(S_j,T_j,\sigma_j,B_j)}=T_\lambda$ and
  $\phi_\lambda\circ \sigma_{\prod_{j\in
      \Lambda}(S_j,T_j,\sigma_j,B_j)}=\sigma_\lambda$, and we have that $x\in
  B_{\prod_{\lambda\in \Lambda}(S_\lambda,T_\lambda,\sigma_\lambda,B_\lambda)}$ is zero if and only if
  $\phi_\lambda(x)=0$ for all $\lambda\in \Lambda$.

  If $(T,S,\sigma,B)$ is a surjective covariant representation
  of $(P,Q,\psi)$ and there for each $\lambda\in \Lambda$ exists a ring homomorphism
  $\psi_\lambda:B\longrightarrow B_\lambda$ such that
  $\psi_\lambda\circ S=S_\lambda$, $\psi_\lambda\circ
  T=T_\lambda$ and $\psi_\lambda\circ\sigma=\sigma_\lambda$, then
  $\ker\eta_{(S,T,\sigma,B)}\subseteq H$ where
  $\eta_{(S,T,\sigma,B)}:\mathcal{T}_{(P,Q,\psi)}\longrightarrow B$ is
  the homomorphism given by Theorem \ref{theor:toeplitz}, and it
  follows that there is a unique ring
  homomorphism
  $\tau:B\longrightarrow
  B_{\prod_{\lambda\in \Lambda}(S_\lambda,T_\lambda,\sigma_\lambda,B_\lambda)}$ such that
  $\tau\circ S=S_{\prod_{\lambda\in \Lambda}(S_\lambda,T_\lambda,\sigma_\lambda,B_\lambda)}$,
  $\tau\circ T=T_{\prod_{\lambda\in \Lambda}(S_\lambda,T_\lambda,\sigma_\lambda,B_\lambda)}$ and
  $\tau\circ \sigma
  =\sigma_{\prod_{\lambda\in \Lambda}(S_\lambda,T_\lambda,\sigma_\lambda,B_\lambda)}$, and such that
  $\phi_\lambda\circ\tau=\psi_\lambda$ for each $\lambda\in \Lambda$.
  If there in addition is a ring homomorphism
  $\varphi:B_{\prod_{\lambda\in
      \Lambda}(S_\lambda,T_\lambda,\sigma_\lambda,B_\lambda)}\longrightarrow B$
  such that
  $\varphi\circ S_{\prod_{\lambda\in \Lambda}(S_\lambda,T_\lambda,\sigma_\lambda,B_\lambda)}=S$,
  $\varphi\circ T_{\prod_{\lambda\in \Lambda}(S_\lambda,T_\lambda,\sigma_\lambda,B_\lambda)}=T$ and
  $\varphi\circ \sigma_{\prod_{\lambda\in
      \Lambda}(S_\lambda,T_\lambda,\sigma_\lambda,B_\lambda)}
  =\sigma$,
  then $\tau$ is an inverse of $\varphi$, and it follows that
  $\varphi$ is an isomorphism.
\end{proof}

\begin{prop} \label{prop:coprod}
  Let $R$ be a ring, let $(P,Q,\psi)$ be an $R$-system and let
  $((S_\lambda,T_\lambda,\sigma_\lambda,B_\lambda))_{\lambda\in \Lambda}$ be a family of surjective
  covariant representations of $(P,Q,\psi)$.

  Then the coproduct of
  $((S_\lambda,T_\lambda,\sigma_\lambda,B_\lambda))_{\lambda\in \Lambda}$
  in $\category$ exists; i.e., there exists a surjective
  covariant representation
  $(S_{\coprod_{\lambda\in \Lambda}(S_\lambda,T_\lambda,\sigma_\lambda,B_\lambda)},
  T_{\coprod_{\lambda\in
      \Lambda}(S_\lambda,T_\lambda,\sigma_\lambda,B_\lambda)},
  \sigma_{\coprod_{\lambda\in
      \Lambda}(S_\lambda,T_\lambda,\sigma_\lambda,B_\lambda)},
  B_{\coprod_{\lambda\in
      \Lambda}(S_\lambda,T_\lambda,\sigma_\lambda,B_\lambda)})$ of $(P,Q,\psi)$
  and a family $(\phi_\lambda)_{\lambda\in \Lambda}$ of
  ring homomorphisms $\phi_\lambda:B_\lambda\longrightarrow B_{\coprod_{j\in
      \Lambda}(S_j,T_j,\sigma_j,B_j)}$ satisfying
  $\phi_\lambda\circ S_\lambda=S_{\coprod_{j\in \Lambda}(S_j,T_j,\sigma_j,B_j)}$,
  $\phi_\lambda\circ T_\lambda=T_{\coprod_{j\in \Lambda}(S_j,T_j,\sigma_j,B_j)}$ and
  $\phi_\lambda\circ \sigma_\lambda=\sigma_{\coprod_{j\in \Lambda}(S_j,T_j,\sigma_j,B_j)}$ for all
  $\lambda\in \Lambda$, with the following property:
  \begin{enumerate}[label=\textbf{(CO)}]
  \item If $(S,T,\sigma,B)$ is a surjective covariant representation
    of $(P,Q,\psi)$ and there for each $\lambda\in \Lambda$ exists a ring homomorphism
    $\psi_\lambda:B_\lambda\longrightarrow B$ such that
    $\psi_\lambda\circ S_\lambda=S$, $\psi_\lambda\circ
    T_\lambda=T$ and $\psi_\lambda\circ\sigma_\lambda=\sigma$,
    then there exists a unique ring homomorphism
    $\tau:B_{\coprod_{\lambda\in \Lambda}(S_\lambda,T_\lambda,\sigma_\lambda,B_\lambda)}\longrightarrow B$ such that
    $\tau\circ S_{\coprod_{\lambda\in \Lambda}(S_\lambda,T_\lambda,\sigma_\lambda,B_\lambda)}=S$,
    $\tau\circ T_{\coprod_{\lambda\in \Lambda}(S_\lambda,T_\lambda,\sigma_\lambda,B_\lambda)}=T$ and
    $\tau\circ \sigma_{\coprod_{\lambda\in
        \Lambda}(S_\lambda,T_\lambda,\sigma_\lambda,B_\lambda)}
    =\sigma$, and such that
    $\tau\circ\phi_\lambda=\psi_\lambda$ for each $\lambda\in \Lambda$. \label{item:15}
  \end{enumerate}
  We furthermore have that the surjective covariant representation
  $$(S_{\coprod_{\lambda\in \Lambda}(S_\lambda,T_\lambda,\sigma_\lambda,B_\lambda)},
  T_{\coprod_{\lambda\in \Lambda}(S_\lambda,T_\lambda,\sigma_\lambda,B_\lambda)},
  \sigma_{\coprod_{\lambda\in
      \Lambda}(S_\lambda,T_\lambda,\sigma_\lambda, B_\lambda)},
  B_{\coprod_{\lambda\in
      \Lambda}(S_\lambda,T_\lambda,\sigma_\lambda,B_\lambda)})$$
  and the family $(\phi_\lambda)_{\lambda\in \Lambda}$ are, up to isomorphism,
  the unique pair which possesses property \ref{item:15}; in fact if
  $(S,T,\sigma,B)$ is a surjective covariant representation
  of $(P,Q,\psi)$ and $(\psi_\lambda)_{\lambda\in \Lambda}$ is a family of ring
  homomorphisms $\psi_\lambda:B_\lambda\longrightarrow B$ satisfying
  $\psi_\lambda\circ S_\lambda=S$, $\psi_\lambda\circ T_\lambda=T$
  and $\psi_\lambda\circ\sigma_\lambda=\sigma$ for each $\lambda\in \Lambda$, and
  $\varphi:B\longrightarrow B_{\coprod_{\lambda\in \Lambda}(S_\lambda,T_\lambda,\sigma_\lambda,B_\lambda)}$ is a ring
  homomorphism such that
  $\varphi\circ S=S_{\coprod_{\lambda\in \Lambda}(S_\lambda,T_\lambda,\sigma_\lambda,B_\lambda)}$,
  $\varphi\circ T=T_{\coprod_{\lambda\in \Lambda}(S_\lambda,T_\lambda,\sigma_\lambda,B_\lambda)}$ and
  $\varphi\circ \sigma=\sigma_{\coprod_{\lambda\in \Lambda}(S_\lambda,T_\lambda,\sigma_\lambda,B_\lambda)}$,
  then $\varphi$ is an isomorphism.

  Moreover, if each $(S_\lambda,T_\lambda,\sigma_\lambda,B_\lambda)$ is graded, then the
  surjective covariant representation
  $(S_{\coprod_{\lambda\in \Lambda}(S_\lambda,T_\lambda,\sigma_\lambda,B_\lambda)},
  T_{\coprod_{\lambda\in \Lambda}(S_\lambda,T_\lambda,\sigma_\lambda,B_\lambda)},
  \sigma_{\coprod_{\lambda\in \Lambda}(S_\lambda,T_\lambda,\sigma_\lambda,B_\lambda)},
  B_{\coprod_{\lambda\in \Lambda}(S_\lambda,T_\lambda,\sigma_\lambda,B_\lambda)})$ is also graded.
\end{prop}

\begin{proof}
  Let $H$ be the smallest two-sided ideal of
  $\mathcal{T}_{(P,Q,\psi)}$ which contains $\cup_{\lambda\in
    \Lambda}\ker\eta_{(S_\lambda,T_\lambda,\sigma_\lambda,B_\lambda)}$ where for
  each $\lambda\in \Lambda$ the homomorphism
  $\eta_{(S_\lambda,T_\lambda,\sigma_\lambda,B_\lambda)}:\mathcal{T}_{(P,Q,\psi)}\longrightarrow B_\lambda$ is
  the homomorphism given by Theorem \ref{theor:toeplitz}. Let
  $\wp_H:\mathcal{T}_{(P,Q,\psi)}\longrightarrow\mathcal{T}_{(P,Q,\psi)}/H$ be the
  corresponding quotient map, and let
  $S_{\coprod_{\lambda\in \Lambda}(S_\lambda,T_\lambda,\sigma_\lambda,B_\lambda)}=\wp_H\circ \iota_P$,
  $T_{\coprod_{\lambda\in \Lambda}(S_\lambda,T_\lambda,\sigma_\lambda,B_\lambda)}=\wp_H\circ \iota_Q$,
  $\sigma_{\coprod_{\lambda\in
      \Lambda}(S_\lambda,T_\lambda,\sigma_\lambda,B_\lambda)}
  =\wp_H\circ \iota_R$ and
  $B_{\coprod_{\lambda\in
      \Lambda}(S_\lambda,T_\lambda,\sigma_\lambda,B_\lambda)}=\mathcal{T}_{(P,Q,\psi)}/H$.
  We then have that
  $$(S_{\coprod_{\lambda\in \Lambda}(S_\lambda,T_\lambda,\sigma_\lambda,B_\lambda)},
  T_{\coprod_{\lambda\in
      \Lambda}(S_\lambda,T_\lambda,\sigma_\lambda,B_\lambda)},
  \sigma_{\coprod_{\lambda\in
      \Lambda}(S_\lambda,T_\lambda,\sigma_\lambda,B_\lambda)},
  B_{\coprod_{\lambda\in
      \Lambda}(S_\lambda,T_\lambda,\sigma_\lambda,B_\lambda)})$$
  is a surjective covariant representation of $(P,Q,\psi)$.
  We also have that there for each
  $\lambda\in \Lambda$ is a ring homomorphism
  $\phi_\lambda:B_\lambda\longrightarrow B_{\coprod_{j\in \Lambda}(S_j,T_j,\sigma_j,B_j)}$ satisfying
  $\phi_\lambda\circ S_\lambda=S_{\coprod_{j\in \Lambda}(S_j,T_j,\sigma_j,B_j)}$,
  $\phi_\lambda\circ T_\lambda=T_{\coprod_{j\in \Lambda}(S_j,T_j,\sigma_j,B_j)}$ and
  $\phi_\lambda\circ \sigma_\lambda=\sigma_{\prod_{j\in
      \Lambda}(S_j,T_j,\sigma_j,B_j)}$.

  If $(S_\lambda,T_\lambda,\sigma_\lambda,B_\lambda)$ is graded, then
  $\ker\eta_{(S_\lambda,T_\lambda,\sigma_\lambda,B_\lambda)}$ is a
  graded two-sided ideal of
  $\mathcal{T}_{(P,Q,\psi)}$. It easily follows that if each
  $(S_\lambda,T_\lambda,\sigma_\lambda,B_\lambda)$ is graded, then $H$
  is a two-sided graded ideal of
  $\mathcal{T}_{(P,Q,\psi)}$, and thus that
  $(S_{\coprod_{\lambda\in \Lambda}(S_\lambda,T_\lambda,\sigma_\lambda,B_\lambda)},
  T_{\coprod_{\lambda\in
      \Lambda}(S_\lambda,T_\lambda,\sigma_\lambda,B_\lambda)},
  \sigma_{\coprod_{\lambda\in
      \Lambda}(S_\lambda,T_\lambda,\sigma_\lambda,B_\lambda)},
  B_{\coprod_{\lambda\in
      \Lambda}(S_\lambda,T_\lambda,\sigma_\lambda,B_\lambda)})$ is also graded.

  If $(S,T,\sigma,B)$ is a surjective covariant representation
  of $(P,Q,\psi)$ and there for each $\lambda\in \Lambda$ exists a ring homomorphism
  $\psi_\lambda:B_\lambda\longrightarrow B$ such that $\psi_\lambda\circ S_\lambda=S$, $\psi_\lambda\circ
  T_\lambda=T$ and $\psi_\lambda\circ\sigma_\lambda=\sigma$, then
  $H\subseteq \ker\eta_{(S,T,\sigma,B)}$ where
  $\eta_{(S,T,\sigma,B)}:\mathcal{T}_{(P,Q,\psi)}\longrightarrow B$ is
  the homomorphism given by Theorem \ref{theor:toeplitz}, and it
  follows that there is a unique ring
  homomorphism
  $\tau: B_{\coprod_{\lambda\in \Lambda}(S_\lambda,T_\lambda,\sigma_\lambda,B_\lambda)}\longrightarrow B$ such that
  $\tau\circ S_{\coprod_{\lambda\in \Lambda}(S_\lambda,T_\lambda,\sigma_\lambda,B_\lambda)}=S$,
  $\tau\circ T_{\coprod_{\lambda\in \Lambda}(S_\lambda,T_\lambda,\sigma_\lambda,B_\lambda)}=T$ and
  $\tau\circ \sigma_{\coprod_{\lambda\in
      \Lambda}(S_\lambda,T_\lambda,\sigma_\lambda,B_\lambda)}
  =\sigma$, and such that
  $\tau\circ\phi_\lambda=\psi_\lambda$ for each $\lambda\in \Lambda$.
  If there in addition is a ring homomorphism
  $\varphi:B\longrightarrow B_{\coprod_{\lambda\in
      \Lambda}(S_\lambda,T_\lambda,\sigma_\lambda,B_\lambda)}$
  such that
  $\varphi\circ S=S_{\coprod_{\lambda\in \Lambda}(S_\lambda,T_\lambda,\sigma_\lambda,B_\lambda)}$,
  $\varphi\circ T=T_{\coprod_{\lambda\in \Lambda}(S_\lambda,T_\lambda,\sigma_\lambda,B_\lambda)}$ and
  $\varphi\circ \sigma
  =\sigma_{\coprod_{\lambda\in \Lambda}(S_\lambda,T_\lambda,\sigma_\lambda,B_\lambda)}$,
  then $\tau$ is an inverse of $\varphi$, and it follows that
  $\varphi$ is an isomorphism.
\end{proof}

\begin{lem} \label{lemma:graded and injective}
  Let $R$ be a ring and let $(P,Q,\psi)$ be an $R$-system. If
  $(S_1,T_1,\sigma_1,B_1)$ and $(S_2,T_2,\sigma_2,B_2)$ are two
  covariant representations of $(P,Q,\psi)$ and $\phi:B_1\longrightarrow B_2$ is a
  ring homomorphism satisfying $\phi\circ T_1=T_2$,
  $\phi\circ S_1=S_2$ and $\phi\circ\sigma_1=\sigma_2$, then the
  following holds:
  \begin{enumerate}
  \item If $(S_2,T_2,\sigma_2,B_2)$ is injective, then so is
    $(S_1,T_1,\sigma_1,B_1)$. \label{item:13}
  \item If $\phi$ is surjective and $(S_2,T_2,\sigma_2,B_2)$ is
    surjective and graded, then so is
    $(S_1,T_1,\sigma_1,B_1)$. \label{item:1}
  \end{enumerate}
\end{lem}

\begin{proof}
  That \eqref{item:13} holds is obvious. If $\phi$ is surjective and
  $(S_2,T_2,\sigma_2,B_2)$ is surjective and graded, then it follows
  from Proposition \ref{prop:graded} that
  $\oplus_{n\in\Z}\eta_{(S_2,T_2,\sigma_2,B_2)}(\mathcal{T}_{(P,Q,\psi)}^{(n)})$
  is a grading of $B_2$. It follows that $\oplus_{n\in\Z}\eta_{(S_1,T_1,\sigma_1,B_1)}
  (\mathcal{T}_{(P,Q,\psi)}^{(n)})$ is a grading of $B_1$, and
  thus that $(S_1,T_1,\sigma_1,B_1)$ is graded.
\end{proof}

\begin{prop} \label{prop:super}
  Let $R$ be a ring, let $(P,Q,\psi)$ be an $R$-system satisfying
  condition \textbf{(FS)} and let
  $\Omega=(\omega_\lambda)_{\lambda\in\Lambda}=((I_\lambda,J_\lambda))_{\lambda\in\Lambda}$
  be a non-empty family of $T$-pairs of $(P,Q,\psi)$.
  For each $\lambda\in\Lambda$ denote by
  $\Gamma_\lambda$ the covariant representation $(\iota_P^{\omega_\lambda},
      \iota_Q^{\omega_\lambda},
      \iota_R^{\omega_\lambda},
      \mathcal{O}_{({}_{I_\lambda}P,Q_{I_\lambda},R_{I_\lambda})}((J_\lambda)_{I_\lambda}))$.
  Then we have:
  \begin{enumerate}
  \item \label{item:16}
    If we let $I_{\prod\Omega}=\cap_{\lambda\in\Lambda}I_\lambda$ and
    $J_{\prod\Omega}=\cap_{\lambda\in\Lambda}J_\lambda$, then the pair
    $\omega_{\prod\Omega}=(I_{\prod\Omega},J_{\prod\Omega})$ is a
    $T$-pair of $(P,Q,\psi)$, and the covariant representation
    $$\left(S_{\prod_{\lambda\in\Lambda}\Gamma_\lambda},
      T_{\prod_{\lambda\in\Lambda}\Gamma_\lambda},
      \sigma_{\prod_{\lambda\in\Lambda}\Gamma_\lambda},
      B_{\prod_{\lambda\in\Lambda}\Gamma_\lambda}\right)$$
    is surjective and graded, and it is isomorphic to
    $$\left(\iota_P^{\omega_{\prod\Omega}},
      \iota_Q^{\omega_{\prod\Omega}},
      \iota_R^{\omega_{\prod\Omega}},
      \mathcal{O}_{({}_{I_{\prod\Omega}}P,
        Q_{I_{\prod\Omega}},
        \psi_{I_{\prod\Omega}})}\bigl((J_{\prod\Omega})_{I_{\prod\Omega}}\bigr)\right).$$
  \item \label{item:17}
    If we let $I$ be the smallest two-sided ideal of $R$ containing
    $\cup_{\lambda\in\Lambda}I_\lambda$, $J_{\coprod\Omega}$ be the
    smallest two-sided ideal of $R$ containing
    $\cup_{\lambda\in\Lambda}J_\lambda$ and $I_{\coprod\Omega}=\{x\in
    J_{\coprod\Omega}\mid \forall
    m\in\N:\Delta_I^m(x)(Q_I^{\otimes m})\subseteq Q_I^{\otimes m}(J_{\coprod\Omega})_I
    \land \exists n\in\N:\Delta_I^n(x)=0\}$, then the pair
    $\omega_{\coprod\Omega}= (I_{\coprod\Omega},J_{\coprod\Omega})$ is a
    $T$-pair of $(P,Q,\psi)$, and the covariant representation
    $$\left(S_{\coprod_{\lambda\in\Lambda}\Gamma_\lambda},
      T_{\coprod_{\lambda\in\Lambda}\Gamma_\lambda},
      \sigma_{\coprod_{\lambda\in\Lambda}\Gamma_\lambda},
      B_{\coprod_{\lambda\in\Lambda}\Gamma_\lambda}\right)$$
    is surjective and graded, and it is isomorphic to
    $$\left(\iota_P^{\omega_{\coprod\Omega}},
      \iota_Q^{\omega_{\coprod\Omega}},
      \iota_R^{\omega_{\coprod\Omega}},
      \mathcal{O}_{({}_{I_{\coprod\Omega}}P,
        Q_{I_{\coprod\Omega}},
        \psi_{I_{\coprod\Omega}})}\bigl((J_{\coprod\Omega})_{I_{\coprod\Omega}}\bigr)\right).$$
  \end{enumerate}
\end{prop}

\begin{proof}
  \eqref{item:16}: It follows from Lemma \ref{lemma:graded and
    injective} that the surjective covariant representation
  $$\left(S_{\prod_{\lambda\in\Lambda}\Gamma_\lambda},
    T_{\prod_{\lambda\in\Lambda}\Gamma_\lambda},
    \sigma_{\prod_{\lambda\in\Lambda}\Gamma_\lambda},
    B_{\prod_{\lambda\in\Lambda}\Gamma_\lambda}\right)$$
  is graded. It therefore follows from Proposition
  \ref{prop_31} and Theorem \ref{theor_1} that
  $$\left(S_{\prod_{\lambda\in\Lambda}\Gamma_\lambda},
    T_{\prod_{\lambda\in\Lambda}\Gamma_\lambda},
    \sigma_{\prod_{\lambda\in\Lambda}\Gamma_\lambda},
    B_{\prod_{\lambda\in\Lambda}\Gamma_\lambda}\right)$$
  is isomorphic to
  $(\iota_P^\omega,\iota_Q^\omega,\iota_R^\omega,\mathcal{O}_{({}_IP,Q_I,\psi_I)}(J_I))$
  for some $T$-pair $\omega=(I,J)$ of $(P,Q,\psi)$. It follows from
  Lemma \ref{lemma:compact} and Proposition \ref{prop:diamond} that we
  have
  \begin{equation*}
    x\in I \iff \sigma_{\prod_{\lambda\in\Lambda}\Gamma_\lambda}(x)=0
    \iff \forall \lambda\in\Lambda: \iota_R^{\omega_\lambda}(x)=0 \iff
    x\in \cap_{\lambda\in\Lambda}I_\lambda=I_{\prod\Omega}
  \end{equation*}
  and
  \begin{align*}
    x\in J &\iff \sigma_{\prod_{\lambda\in\Lambda}\Gamma_\lambda}(x)
    =\pi_{T_{\prod_{\lambda\in\Lambda}\Gamma_\lambda},S_{\prod_{\lambda\in\Lambda}\Gamma_\lambda}}
    (\Delta(x)) \\&\iff \forall \lambda\in\Lambda:
    \iota_R^{\omega_\lambda}(x)
    =\pi_{\iota_Q^\omega,\iota_P^\omega}(\Delta(x)) \iff x\in
    \cap_{\lambda\in\Lambda}J_\lambda= J_{\prod\Omega}
  \end{align*}
  from which \eqref{item:16} follows.

  \eqref{item:17}:
   It follows from Proposition \ref{prop:coprod} that the representation
  $$\left(S_{\coprod_{\lambda\in\Lambda}\Gamma_\lambda},
    T_{\coprod_{\lambda\in\Lambda}\Gamma_\lambda},
    \sigma_{\coprod_{\lambda\in\Lambda}\Gamma_\lambda},
    B_{\coprod_{\lambda\in\Lambda}\Gamma_\lambda}\right)$$
  is surjective and graded.

  It is easy to check that $I\subseteq J_{\coprod\Omega}$,
  that $I$ is $\psi$-invariant and that
  $\wp_I(J_{\coprod\Omega})\subseteq \Delta^{-1}_I(\mathcal{F}_{{}_IP}(Q_I))$.
  It therefore follows from Proposition \ref{prop_31} and \ref{prop:newt} that
  $(I_{\coprod\Omega},J_{\coprod\Omega})$ is a $T$-pair of $(P,Q,\psi)$.

  We have for each $\lambda\in\Lambda$ that $I_\lambda\subseteq I_{\coprod\Omega}$ and
  $J_\lambda\subseteq J_{\coprod\Omega}$ so it follows from Proposition \ref{prop_32} and
  Theorem \ref{theor_1} \eqref{item:18} that there exists a ring homomorphism
  $$\psi_\lambda: \mathcal{O}_{({}_{I_\lambda}P,Q_{I_\lambda},R_{I_\lambda})}((J_\lambda)_{I_\lambda})
  \longrightarrow \mathcal{O}_{({}_{I_{\coprod\Omega}}P,Q_{I_{\coprod\Omega}},R_{I_{\coprod\Omega}})}((J_{\coprod\Omega})_{I_{\coprod\Omega}})$$
  such that $\psi_\lambda\circ\iota_R^{\omega_\lambda}=\iota_R^{\omega_{\coprod\Omega}}$,
  $\psi_\lambda\circ\iota_Q^{\omega_\lambda}=\iota_Q^{\omega_{\coprod\Omega}}$ and
  $\psi_\lambda\circ\iota_P^{\omega_\lambda}=\iota_P^{\omega_{\coprod\Omega}}$.

  We will show that there exists a ring homomorphism
  $$\phi:
  \mathcal{O}_{({}_{I_{\coprod\Omega}}P,Q_{I_{\coprod\Omega}},
    R_{I_{\coprod\Omega}})}((J_{\coprod\Omega})_{I_{\coprod\Omega}})
  \longrightarrow B_{\coprod_{\lambda\in\Lambda}\Gamma_\lambda}$$
  such that $\phi\circ \iota_R^{\omega_{\coprod\Omega}}
  =\sigma_{\coprod_{\lambda\in\Lambda}\Gamma_\lambda}$, $\phi\circ
  \iota_Q^{\omega_{\coprod\Omega}}
  =T_{\coprod_{\lambda\in\Lambda}\Gamma_\lambda}$ and
  $\phi\circ \iota_P^{\omega_{\coprod\Omega}}
  =S_{\coprod_{\lambda\in\Lambda}\Gamma_\lambda}$.
  It will then follow from Proposition \ref{prop:coprod} that the two representations
  $$\left(S_{\coprod_{\lambda\in\Lambda}\Gamma_\lambda},
      T_{\coprod_{\lambda\in\Lambda}\Gamma_\lambda},
      \sigma_{\coprod_{\lambda\in\Lambda}\Gamma_\lambda},
      B_{\coprod_{\lambda\in\Lambda}\Gamma_\lambda}\right)$$
    and
    $$\left(\iota_P^{\omega_{\coprod\Omega}},
      \iota_Q^{\omega_{\coprod\Omega}},
      \iota_R^{\omega_{\coprod\Omega}},
      \mathcal{O}_{({}_{I_{\coprod\Omega}}P,
        Q_{I_{\coprod\Omega}},
        \psi_{I_{\coprod\Omega}})}\bigl((J_{\coprod\Omega})_{I_{\coprod\Omega}}\bigr)\right)$$
    are isomorphic.

    We have for each $\lambda\in\Lambda$ that there is a ring homomorphism
    $\phi_\lambda:\mathcal{O}_{({}_{I_\lambda}P,Q_{I_\lambda},R_{I_\lambda})}((J_\lambda)_{I_\lambda}
    \longrightarrow B_{\coprod_{\lambda\in\Lambda}\Gamma_\lambda}$ such that $\phi_\lambda\circ
    \iota_R^{\omega_\lambda}= \sigma_{\coprod_{\lambda\in\Lambda}\Gamma_\lambda}$,
    $\phi_\lambda\circ
    \iota_Q^{\omega_\lambda}= T_{\coprod_{\lambda\in\Lambda}\Gamma_\lambda}$
    and $\phi_\lambda\circ
    \iota_P^{\omega_\lambda}= S_{\coprod_{\lambda\in\Lambda}\Gamma_\lambda}$. It follows from
    Theorem \ref{theor_1} that we have $$I_\lambda\subseteq
    I_{(S_{\coprod_{\lambda\in\Lambda}\Gamma_\lambda},
      T_{\coprod_{\lambda\in\Lambda}\Gamma_\lambda},
      \sigma_{\coprod_{\lambda\in\Lambda}\Gamma_\lambda},
      B_{\coprod_{\lambda\in\Lambda}\Gamma_\lambda})}$$ and
    $$J_\lambda\subseteq J_{(S_{\coprod_{\lambda\in\Lambda}\Gamma_\lambda},
      T_{\coprod_{\lambda\in\Lambda}\Gamma_\lambda},
      \sigma_{\coprod_{\lambda\in\Lambda}\Gamma_\lambda},
      B_{\coprod_{\lambda\in\Lambda}\Gamma_\lambda})}.$$
    We therefore have that
    \begin{equation*}
      I\subseteq
      I_{(S_{\coprod_{\lambda\in\Lambda}\Gamma_\lambda},
        T_{\coprod_{\lambda\in\Lambda}\Gamma_\lambda},
        \sigma_{\coprod_{\lambda\in\Lambda}\Gamma_\lambda},
        B_{\coprod_{\lambda\in\Lambda}\Gamma_\lambda})}
    \end{equation*}
    and
    \begin{equation} \label{eq:8}
      J_{\coprod\Omega}\subseteq J_{(S_{\coprod_{\lambda\in\Lambda}\Gamma_\lambda},
        T_{\coprod_{\lambda\in\Lambda}\Gamma_\lambda},
        \sigma_{\coprod_{\lambda\in\Lambda}\Gamma_\lambda},
        B_{\coprod_{\lambda\in\Lambda}\Gamma_\lambda})}.
    \end{equation}
    It then follows from Lemma \ref{lemma_33} that there exists a covariant representation
    $(S,T,\sigma,B_{\coprod_{\lambda\in\Lambda}\Gamma_\lambda})$ of $({}_IP,Q_I,\psi_I)$ such
    that $S\circ\wp_I=S_{\coprod_{\lambda\in\Lambda}\Gamma_\lambda}$,
    $T\circ\wp_I=T_{\coprod_{\lambda\in\Lambda}\Gamma_\lambda}$ and
    $\sigma\circ\wp_I=\sigma_{\coprod_{\lambda\in\Lambda}\Gamma_\lambda}$.
    It follows from Equation \eqref{eq:8} that this representation is
    Cuntz-Pimsner invariant relative to $(J_{\coprod\Omega})_I$, and it then follows
    from Theorem \ref{univ_cuntz} that there is a ring homomorphism
    $\eta:\mathcal{O}_{({}_IP,Q_I,\psi_I)}((J_{\coprod\Omega})_I)\longrightarrow
    B_{\coprod_{\lambda\in\Lambda}\Gamma_\lambda}$ such that $\eta\circ
    \iota^{(J_{\coprod\Omega})_I}_{R_I}=\sigma$, $\eta\circ\iota^{(J_{\coprod\Omega})_I}_{Q_I}=T$ and
    $\eta\circ \iota^{(J_{\coprod\Omega})_I}_{{}_IP}=S$.
    It follows from Proposition \ref{prop:newt} that the two representations
    $$\left(\iota_P^{\omega_{\coprod\Omega}},
      \iota_Q^{\omega_{\coprod\Omega}},
      \iota_R^{\omega_{\coprod\Omega}},
      \mathcal{O}_{({}_{I_{\coprod\Omega}}P,
        Q_{I_{\coprod\Omega}},
        \psi_{I_{\coprod\Omega}})}\bigl((J_{\coprod\Omega})_{I_{\coprod\Omega}}\bigr)\right)$$
    and
    $$(\iota^{(J_{\coprod\Omega})_I}_{{}_IP}\circ \wp_I,\iota^{(J_{\coprod\Omega})_I}_{Q_I}
    \circ \wp_I,\iota^{(J_{\coprod\Omega})_I}_{R_I}\circ
    \wp_I,\mathcal{O}_{({}_IP,Q_I,\psi_I)}((J_{\coprod\Omega})_I))$$
    are isomorphic, and it follows that there exists a ring homomorphism
    $$\phi:
    \mathcal{O}_{({}_{I_{\coprod\Omega}}P,Q_{I_{\coprod\Omega}},
      R_{I_{\coprod\Omega}})}((J_{\coprod\Omega})_{I_{\coprod\Omega}})
    \longrightarrow B_{\coprod_{\lambda\in\Lambda}\Gamma_\lambda}$$
    such that $\phi\circ \iota_R^{\omega_{\coprod\Omega}}
    =\sigma_{\coprod_{\lambda\in\Lambda}\Gamma_\lambda}$, $\phi\circ
    \iota_Q^{\omega_{\coprod\Omega}}
    =T_{\coprod_{\lambda\in\Lambda}\Gamma_\lambda}$ and
    $\phi\circ \iota_P^{\omega_{\coprod\Omega}}
    =S_{\coprod_{\lambda\in\Lambda}\Gamma_\lambda}$.
\end{proof}

\begin{rema}
  Let $R$ be a ring, let $(P,Q,\psi)$ be an $R$-system and let
  $((S_\lambda,T_\lambda,\sigma_\lambda,B_\lambda))_{\lambda\in \Lambda}$ be a family of injective
  and surjective covariant representations of $(P,Q,\psi)$. Then the product
  $(S_{\prod_{\lambda\in \Lambda}(S_\lambda,T_\lambda,\sigma_\lambda,B_\lambda)},
  T_{\prod_{\lambda\in
      \Lambda}(S_\lambda,T_\lambda,\sigma_\lambda,B_\lambda)},
  \sigma_{\prod_{\lambda\in
      \Lambda}(S_\lambda,T_\lambda,\sigma_\lambda,B_\lambda)},
  B_{\prod_{\lambda\in
      \Lambda}(S_\lambda,T_\lambda,\sigma_\lambda,B_\lambda)})$ is also injective and surjective, but the coproduct
  $$(S_{\coprod_{\lambda\in \Lambda}(S_\lambda,T_\lambda,\sigma_\lambda,B_\lambda)},
  T_{\coprod_{\lambda\in
      \Lambda}(S_\lambda,T_\lambda,\sigma_\lambda,B_\lambda)},
  \sigma_{\coprod_{\lambda\in
      \Lambda}(S_\lambda,T_\lambda,\sigma_\lambda,B_\lambda)},
  B_{\coprod_{\lambda\in
      \Lambda}(S_\lambda,T_\lambda,\sigma_\lambda,B_\lambda)})$$ is not necessarily injective. Example \ref{example_two_maximal} gives us an example of this phenomena.
\end{rema}

\subsection{Graded ideals of $\mathcal{O}_{(P,Q,\psi)}(J)$}
\label{sec:graded-ideals}

Let $R$ be a ring and $(P,Q,\psi)$ an $R$-system satisfying condition \textbf{(FS)}.
We will now show how the classification of surjective and graded
representations of $(P,Q,\psi)$ can be used to describe the graded
two-sided ideals of $\mathcal{O}_{(P,Q,\psi)}(J)$ for any faithful $\psi$-compatible two-sided ideal $J$ of $R$, and in particular of $\mathcal{T}_{(P,Q,\psi)}$ and $\mathcal{O}_{(P,Q,\psi)}$ (if it exists).

\begin{defi} \label{def:omegahk}
Let $R$ be a ring, let $(P,Q,\psi)$ be an $R$-system satisfying condition
\textbf{(FS)} and let $K$ be a
two-sided ideal of $R$ such that $K\subseteq
\Delta^{-1}(\mathcal{F}_P(Q))$ and $K\cap \ker\Delta=0$. For
a two-sided ideal $H$ of $\mathcal{O}_{(P,Q,\psi)}(K)$ we define two two-sided ideals
$I^K_H$ and $J^K_H$ of $R$ by
$$I^K_H:=\{x\in R\mid \iota^{K}_R(x)\in H\}\qquad \text{and}\qquad
J^K_H:=\{x\in R\mid \iota^{K}_R(x)\in H+\mathcal{F}_P(Q)\}\,.$$
We set $\omega^K_H=(I^K_H,J^K_H)$.
\end{defi}

\begin{prop}\label{prop_33}
Let $R$ be a ring, let $(P,Q,\psi)$ be an $R$-system satisfying
condition \textbf{(FS)} and let $K$ be a two-sided ideal of $R$
such that $K\subseteq \Delta^{-1}(\mathcal{F}_P(Q))$ and $K\cap
\ker\Delta=0$. For a two-sided ideal $H$ of
$\mathcal{O}_{(P,Q,\psi)}(K)$, denote by $\wp_H$ the projection
from $\mathcal{O}_{(P,Q,\psi)}(K)$ to
$\mathcal{O}_{(P,Q,\psi)}(K)/H$. If we consider the covariant
representation
$$(S_H,T_H,\sigma_H,\mathcal{O}_{(P,Q,\psi)}(K)/H):=(\wp_H\circ\iota^{K}_P,\wp_H\circ
\iota^{K}_Q,\wp_H\circ \iota^{K}_R,\mathcal{O}_{(P,Q,\psi)}(K)/H),$$
then we have that
$\omega^K_H=\omega_{(S_H,T_H,\sigma_H,\mathcal{O}_{(P,Q,\psi)}(K)/H)}$.
Hence $\omega^K_H$ is a $T$-pair satisfying $K \subseteq J^K_H$.

We furthermore have that the representation
$(S_H,T_H,\sigma_H,\mathcal{O}_{(P,Q,\psi)}(K)/H)$ is graded if and only if $H$ is graded.
\end{prop}

\begin{proof}
By using that $\wp_H\circ\iota_R^K=\sigma_H$ and
$\wp_H\circ\pi^K=\pi_{T_H,S_H}$, it is straightforward to check that
$I_H^K=I_{(S_H,T_H,\sigma_H,\mathcal{O}_{(P,Q,\psi)}(K)/H)}$ and
$J_H^K=J_{(S_H,T_H,\sigma_H,\mathcal{O}_{(P,Q,\psi)}(K)/H)}$,
and thus that $\omega^K_H=\omega_{(S_H,T_H,\sigma_H,\mathcal{O}_{(P,Q,\psi)}(K)/H)}$. It is also
easy to check that $K \subseteq J^K_H$. That
$\omega_{(S_H,T_H,\sigma_H \mathcal{O}_{(P,Q,\psi)}(K)/H)}$, and thus $\omega^K_H$, is a $T$-pair
follows from Proposition \ref{prop_31}.

Assume that $H$ is graded. If $x=\sum_{i=1}^mx^{n_i}\in H$ where each $x^{n_i}\in \wp_K(\mathcal{T}_{(P,Q,\psi)}^{(n_i)})$, then each $x^{n_i}\in H$. This shows that
$\oplus_{n\in\Z}\wp_H(\wp_K(\mathcal{T}_{(P,Q,\psi)}^{(n)}))$ is a grading of $\mathcal{O}_{(P,Q,\psi)}(K)/H$, and it follows that $(S_H,T_H,\sigma_H,\mathcal{O}_{(P,Q,\psi)}(K)/H)$ is graded.

If $(S_H,T_H,\sigma_H,\mathcal{O}_{(P,Q,\psi)}(K)/H)$ is graded and
$x=\sum_{i=1}^mx^{n_i}\in H$ where each $x^{n_i}\in \wp_K(\mathcal{T}_{(P,Q,\psi)}^{(n_i)})$, then each $\wp_H(x^{(n_i)})=0$ which shows that $H=\oplus_{n\in Z}(\wp_K(\mathcal{T}_{(P,Q,\psi)}^{(n)})\cap H)$, and thus that $H$ is graded.
\end{proof}

\begin{lem} \label{lemma:psi}
Let $R$ be a ring, let $(P,Q,\psi)$ be an $R$-system satisfying
condition \textbf{(FS)} and let $K$ be a faithful $\psi$-compatible 
two-sided ideal of $R$.
If $\omega=(I,J)$ is a $T$-pair such that $K
\subseteq J$, then there exists a unique surjective and graded ring homomorphisms
$\Psi^K_\omega:\mathcal{O}_{(P,Q,\psi)}(K) \longrightarrow
\mathcal{O}_{({}_IP,Q_I,\psi_I)}(J_I)$ such that $\Psi^K_\omega\circ\iota_R^K=\iota_R^\omega$,
$\Psi^K_\omega\circ\iota_Q^K=\iota_Q^\omega$ and $\Psi^K_\omega\circ\iota_P^K=\iota_P^\omega$.
\end{lem}

\begin{proof}
We have that
$(\iota_P^\omega,\iota_Q^\omega,\iota_R^\omega,\mathcal{O}_{({}_IP,Q_I,\psi_I)}(J_I))$
is a surjective and graded covariant representation of $(P,Q,\psi)$,
and since $K \subseteq J$, this representation is Cuntz-Pimsner invariant relative to $K$.
The uniqueness and existence of $\Psi^K_\omega$ then follows from Theorem \ref{univ_cuntz}. It is easy to check that $\Psi^K_\omega$ is graded.
\end{proof}

\begin{defi} \label{defi:H}
Let $R$ be a ring, let $(P,Q,\psi)$ be an $R$-system satisfying
condition \textbf{(FS)} and let $K$ be a faithful $\psi$-compatible  
two-sided ideal of $R$.
Given a $T$-pair $\omega=(I,J)$ such that $K
\subseteq J$. We define $H_\omega^K$ to be the two-sided ideal $\ker\Psi^K_\omega$ of $R$ where $\Psi^K_\omega$ is as in Lemma \ref{lemma:psi}.
\end{defi}

\begin{lem} \label{lemma:sur}
Let $R$ be a ring, let $(P,Q,\psi)$ be an $R$-system satisfying
condition \textbf{(FS)} and let $K$ be a faithful $\psi$-compatible  
two-sided ideal of $R$.
If $\omega=(I,J)$ is a $T$-pair such that $K
\subseteq J$, then $H_\omega^K$ is a graded two-sided ideal of
$\mathcal{O}_{(P,Q,\psi)}(K)$ and satisfies
$\omega_{H_\omega^K}=\omega$.
\end{lem}

\begin{proof}
Let $\Psi^K_\omega$ be the homomorphism from Lemma \ref{lemma:psi}.
That $H_\omega^K$ is a graded two-sided ideal follows from the fact that
$\Psi^K_\omega$ is graded.

To show $\omega_{H_\omega^K}=\omega$ we have to show that
$I=(\iota_R^K)\inv(\ker\Psi^K_\omega)$ and that
$J=(\iota_R^K)\inv(\ker\Psi^K_\omega+\pi^K(\mathcal{F}_P(Q)))$. If
$x\in I$, then
$\Psi^K_\omega(\iota_R^K(x))=\iota_{R_I}^{J_I}(\wp_I(x))=0$. Thus
$I\subseteq(\iota_R^K)\inv(\ker\Psi^K_\omega)$. If $x\in R$ and
$\Psi^K_\omega(\iota_R^K(x))=0$, then $\iota_{R_I}^{J_I}(\wp(x))=0$,
and since $\iota_{R_I}^{J_I}$ is injective, it follows that
$x\in\ker \wp_I=I$. Thus $I=(\iota_R^K)\inv(\ker\Psi^K_\omega)$.

Let $x\in J$. Then $\wp_I(x)\in J_I$, so we have
\begin{equation*}
\Psi^K_\omega(\iota_R^K(x))=\iota_{R_I}^{J_I}(\wp_I(x))
=\pi^{J_I}(\Delta_I(\wp_I(x))).
\end{equation*}
Thus there exist $q_1,q_2,\dots,q_n\in Q$ and
$p_1,p_2,\dots,p_n\in P$ such that
$$\Psi^K_\omega(\iota_R^K(x))=
\sum_{i=1}^n\iota_{Q_I}^{J_I}(\wp_I(q_i))\iota_{{}_IP}^{J_I}(\wp_I(p_i)).$$
We then have that
$\iota_R^K(x)-\sum_{i=1}^n\iota_Q^K(q_i)\iota_P^K(p_i)\in
\ker\Psi^K_\omega$, which shows that $J\subseteq
(\iota_R^K)\inv(\ker\Psi^K_\omega+\pi^K(\mathcal{F}_P(Q)))$.

Let $\sigma_\omega:=\iota^{J_I}_{R_I}\circ \wp_I$,
$T_\omega:=\iota^{J_I}_{Q_I}\circ \wp_I$ and
$S_\omega:=\iota^{J_I}_{{}_IP}\circ \wp_I$. It follows from
Proposition \ref{prop_32} that
$\sigma_\omega^{-1}(\pi_{T_\omega,S_\omega}(\mathcal{F}_P(Q)))=J$.
If $x\in R$, $y\in\ker\Psi^K_\omega$, $q_1,q_2,\dots,q_n\in Q$,
$p_1,p_2,\dots,p_n\in P$ and
$\iota_R^K(x)=y+\sum_{i=1}^n\iota_Q^K(q_i)\iota_P^K(p_i)$, then
\begin{align*}
\sigma_\omega(x) &
=\iota^{J_I}_{R_I}(\wp_I(x))=\Psi^K_\omega(\iota_R^K(x))=
\Psi^K_\omega(\sum_{i=1}^n\iota_Q^K(q_i)\iota_P^K(p_i))\\ & =
\sum_{i=1}^n\iota^{J_I}_{Q_I}(\wp_I(q_i))\iota^{J_I}_{{}_IP}(\wp_I(p_i)) = \pi_{T_\omega,S_\omega}(\sum_{i=1}^n\theta_{q_i,p_i})\in
\pi_{T_\omega,S_\omega}(\mathcal{F}_P(Q))\,,
\end{align*}
 so $x\in J$. Thus
$J=(\iota_R^K)\inv(\ker\Psi^K_\omega+\pi^K(\mathcal{F}_P(Q)))$.
\end{proof}

\begin{prop}\label{theor_2}
Let $R$ be a ring, let $(P,Q,\psi)$ be an $R$-system satisfying
condition \textbf{(FS)} and let $K$ be a faithful $\psi$-compatible  
two-sided ideal of $R$.
Let $H$ be a two-sided ideal of
$\mathcal{O}_{(P,Q,\psi)}(K)$ and let $\omega=(I,J)$ be a $T$-pair of $(P,Q,\psi)$.
Let $\wp_H$ denote the quotient map from $\mathcal{O}_{(P,Q,\psi)}(K)$ to
$\mathcal{O}_{(P,Q,\psi)}(K)/H$.
Then we have:
\begin{enumerate}
\item \label{item:22}
  If there exists a ring homomorphism $\Upsilon:\mathcal{O}_{({}_IP,Q_I,\psi_I)}(J_I)
  \longrightarrow \mathcal{O}_{(P,Q,\psi)}(K)/H$ such that
  $\Upsilon\circ\iota_R^\omega=\wp_H\circ\iota_R^K$,
  $\Upsilon\circ\iota_Q^\omega=\wp_H\circ\iota_Q^K$ and
  $\Upsilon\circ\iota_P^\omega=\wp_H\circ\iota_P^K$, then $I\subseteq I_H^K$ and
  $J\subseteq J_H^K$.
\item \label{item:23}
  If $I\subseteq I_H^K$ and $J\subseteq J_H^K$, then there exists a unique
  ring homomorphism $\Upsilon:\mathcal{O}_{({}_IP,Q_I,\psi_I)}(J_I)
  \longrightarrow \mathcal{O}_{(P,Q,\psi)}(K)/H$ such that
  $\Upsilon\circ\iota_R^\omega=\wp_H\circ\iota_R^K$,
  $\Upsilon\circ\iota_Q^\omega=\wp_H\circ\iota_Q^K$ and
  $\Upsilon\circ\iota_P^\omega=\wp_H\circ\iota_P^K$.
\item \label{item:24}
  If $I\subseteq I_H^K$ and $J\subseteq J_H^K$, then the
  ring homomorphism $\Upsilon$ is an isomorphism
  if and only if $H$ is graded and $\omega_H^K=\omega$.
\end{enumerate}
\end{prop}

\begin{proof}
  \eqref{item:22}: Assume that there exists a ring homomorphism
  $\Upsilon:\mathcal{O}_{({}_IP,Q_I,\psi_I)}(J_I)
  \longrightarrow \mathcal{O}_{(P,Q,\psi)}(K)/H$ such that
  $\Upsilon\circ\iota_R^\omega=\wp_H\circ\iota_R^K$,
  $\Upsilon\circ\iota_Q^\omega=\wp_H\circ\iota_Q^K$ and
  $\Upsilon\circ\iota_P^\omega=\wp_H\circ\iota_P^K$.
  If $x\in I$, then  it follows from Proposition \ref{prop_32}
  that$\wp_H(\iota_R^K(x))=\Upsilon(\iota_R^\omega(x))=0$, so $x\in I_H^K$.
  If $x\in J$, then  it follows from Lemma \ref{lemma:compact} and
  Proposition \ref{prop_32}
  that
  $$\wp_H(\iota_R^K(x))=\Upsilon(\iota_R^\omega(x))
  =\Upsilon(\pi_{\iota_Q^\omega, \iota_P^\omega}(\Delta(x)))
  =\wp_H(\pi^K(\Delta(x))),$$
  so $x\in J_H^K$.

  \eqref{item:23}: Assume that $I\subseteq I_H^K$ and $J\subseteq J_H^K$.
  Let $(S_H,T_H,\sigma_H,\mathcal{O}_{(P,Q,\psi)}(K)/H)$ be as in
  Proposition \ref{prop_33}. Then we have
  $(I,J)\subseteq \omega_H^K=\omega_{(S_H,T_H,\sigma_H)}$, so the existence and
  uniqueness of $\Upsilon$ follows from Theorem \ref{theor_1}.

  \eqref{item:24}: It also follows from Theorem \ref{theor_1} that $\Upsilon$ is
  an isomorphism if and only if the representation
  $(S_H,T_H,\sigma_H,\mathcal{O}_{(P,Q,\psi)}(K)/H)$
  is surjective and graded and $\omega=\omega_{(S_H,T_H,\sigma_H)}=\omega_H^K$.
  The representation $(S_H,T_H,\sigma_H,\mathcal{O}_{(P,Q,\psi)}(K)/H)$
  is always surjective, and it follows from Proposition \ref{prop_33} that
  it is graded if and only if $H$ is graded, and the desired result follows.
\end{proof}

\begin{theor}\label{ideals_cuntz}
Let $R$ be a ring and let $(P,Q,\psi)$ be an $R$-system satisfying
condition \textbf{(FS)}. Let $K$ be a faithful $\psi$-compatible  
two-sided ideal of $R$.
Then
$$H\longmapsto \omega^K_H,\quad\omega \longmapsto H^K_w$$
is a bijective correspondence between
the set of all the graded two-sided ideals $H$ of
$\mathcal{O}_{(P,Q,\psi)}(K)$ and the set of
all $T$-pairs $\omega=(I,J)$ of $(P,Q,\psi)$ satisfying $K\subseteq J$.
This bijection preserves inclusion, and if $(H_\lambda)_{\lambda\in\Lambda}$ is a non-empty
family of graded two-sided ideals of $\mathcal{O}_{(P,Q,\psi)}(K)$ and
$\Omega=(\omega^K_{H_\lambda})_{\lambda\in\Lambda}$, then
$H^K_{\omega_{\prod\Omega}}=\cap_{\lambda\in\Lambda}H_\lambda$ and $H^K_{\omega_{\coprod\Omega}}$ is the smallest two-sided ideal of $\mathcal{O}_{(P,Q,\psi)}(K)$
containing $\cup_{\lambda\in\Lambda}H_\lambda$.
\end{theor}

\begin{proof}
  If $\omega=(I,J)$ is a $T$-pair of $(P,Q,\psi)$ satisfying $K\subseteq J$, then it
  follows from Lemma \ref{lemma:sur} that $H^K_\omega$ is a graded two-sided
  ideal of $\mathcal{O}_{(P,Q,\psi)}(K)$, and that $\omega_{H^K_\omega}=\omega$.

  If $H$  is a graded two-sided ideal of $\mathcal{O}_{(P,Q,\psi)}(K)$, then it follows from
  Proposition \ref{prop_33} that $\omega^K_H$ is a $T$-pair of $(P,Q,\psi)$
  satisfying $K \subseteq J^K_H$. Let $\Psi^K_\omega$ be the unique ring homomorphism
  $\Psi^K_\omega:\mathcal{O}_{(P,Q,\psi)}(K) \longrightarrow
  \mathcal{O}_{({}_IP,Q_I,\psi_I)}(J_I)$ satisfying
  $\Psi^K_\omega\circ\iota_R^K=\iota_R^\omega$,
  $\Psi^K_\omega\circ\iota_Q^K=\iota_Q^\omega$ and
  $\Psi^K_\omega\circ\iota_P^K=\iota_P^\omega$.
  Let $(I,J)=\omega=\omega^K_H$. Then it follows from Proposition \ref{theor_2}
  that there is a ring isomorphism
  $\Upsilon:\mathcal{O}_{({}_IP,Q_I,\psi_I)}(J_I)
  \longrightarrow \mathcal{O}_{(P,Q,\psi)}(K)/H$ such that
  $\Upsilon\circ\iota_R^\omega=\wp_H\circ\iota_R^K$,
  $\Upsilon\circ\iota_Q^\omega=\wp_H\circ\iota_Q^K$ and
  $\Upsilon\circ\iota_P^\omega=\wp_H\circ\iota_P^K$. We then have that
  $\Upsilon\circ\Psi^K_\omega$ is the quotient map from $\mathcal{O}_{(P,Q,\psi)}(K)$ to
  $\mathcal{O}_{(P,Q,\psi)}(K)/H$, and it follows that $H^K_{\omega^K_H}=\ker\Psi^K_H=H$.

  Thus $H\longmapsto \omega^K_H$ and $\omega \longmapsto H^K_w$ is a
  bijective correspondence between the set of all the graded two-sided
  ideals of $\mathcal{O}_{(P,Q,\psi)}(K)$ and the set of all the $T$-pairs
  $\omega=(I,J)$ of $(P,Q,\psi)$ satisfying $K\subseteq J$.
  It is easy to check that the correspondence preserve inclusion.

  Let $(H_\lambda)_{\lambda\in\Lambda}$ be a non-empty
  family of graded two-sided ideals of $\mathcal{O}_{(P,Q,\psi)}(K)$ and let
  $\Omega=(\omega^K_{H_\lambda})_{\lambda\in\Lambda}$.
  For each $\lambda\in\Lambda$ let
  $(S_{H_\lambda},T_{H_\lambda},\sigma_{H_\lambda},\mathcal{O}_{(P,Q,\psi)}(K)/{H_\lambda})$
  be as in Proposition \ref{prop_33}.
  It follows from Proposition \ref{theor_2} that
  $(S_{H_\lambda},T_{H_\lambda},\sigma_{H_\lambda},\mathcal{O}_{(P,Q,\psi)}(K)/{H_\lambda})$
  is isomorphic to the covariant representation
  $$\left(\iota_P^{\omega^K_{H_\lambda}},
  \iota_Q^{\omega^K_{H_\lambda}},
  \iota_R^{\omega^K_{H_\lambda}},
  \mathcal{O}_{\bigl({}_{I^K_{H_\lambda}}P,
    Q_{I^K_{H_\lambda}},R_{I^K_{H_\lambda}}\bigr)}
  \bigl((J^K_{H_\lambda})_{I^K_{H_\lambda}}\bigr)\right).$$
  It therefore follows from Proposition \ref{prop:super} that
  there exists a ring isomorphism
  $$\phi:B_{\prod_{\lambda\in \Lambda}(S_\lambda,T_\lambda,\sigma_\lambda,B_\lambda)}
  \longrightarrow \mathcal{O}_{\bigl({}_{I_{\prod\Omega}}P,
    Q_{I_{\prod\Omega}},
    \psi_{I_{\prod\Omega}}\bigr)}\bigl((J_{\prod\Omega})_{I_{\prod\Omega}}\bigr)$$
  satisfying
  $\phi\circ \sigma_{\prod_{\lambda\in \Lambda}
    (S_\lambda,T_\lambda,\sigma_\lambda,B_\lambda)}
  =\iota_R^{\omega_{\prod\Omega}}$,
  $\phi\circ T_{\prod_{\lambda\in \Lambda}
    (S_\lambda,T_\lambda,\sigma_\lambda,B_\lambda)}
  =\iota_Q^{\omega_{\prod\Omega}}$
  and $\phi\circ S_{\prod_{\lambda\in \Lambda}
    (S_\lambda,T_\lambda,\sigma_\lambda,B_\lambda)}
  =\iota_P^{\omega_{\prod\Omega}}$.

  If $x\in K$, then we have for all $\lambda\in\Lambda$
  that $\sigma_{H_\lambda}(x)-\pi_{T_H,S_H}(\Delta(x))=0$, and it thus follows from
  Proposition \ref{prop:diamond} that
  $\sigma_{\prod_{\lambda\in \Lambda}(S_\lambda,T_\lambda,\sigma_\lambda,B_\lambda)}(x)=
  \pi_{T_{\prod_{\lambda\in \Lambda}(S_\lambda,T_\lambda,\sigma_\lambda,B_\lambda)},
    S_{\prod_{\lambda\in \Lambda}(S_\lambda,T_\lambda,\sigma_\lambda,B_\lambda)}}(\Delta(x))$. Thus
  the covariant representation
  $$\bigl(S_{\prod_{\lambda\in \Lambda}(S_\lambda,T_\lambda,\sigma_\lambda,B_\lambda)},
  T_{\prod_{\lambda\in \Lambda}(S_\lambda,T_\lambda,\sigma_\lambda,B_\lambda)},
  \sigma_{\prod_{\lambda\in \Lambda}(S_\lambda,T_\lambda,\sigma_\lambda,B_\lambda)},
  B_{\prod_{\lambda\in \Lambda}(S_\lambda,T_\lambda,\sigma_\lambda,B_\lambda)}\bigr)$$
  of $(P,Q,\psi)$ is Cuntz-Pimsner invariant relative to $K$. It therefore
  follows from Theorem \ref{univ_cuntz} that there exists a ring homomorphism
  $\eta:\mathcal{O}_{(P,Q,\psi)}(K)\longrightarrow
  B_{\prod_{\lambda\in \Lambda}(S_\lambda,T_\lambda,\sigma_\lambda,B_\lambda)}$
  such that
  $\eta\circ \iota^K_R
  =\sigma_{\prod_{\lambda\in\Lambda}(S_\lambda,T_\lambda,\sigma_\lambda,B_\lambda)}$,
  $\eta\circ \iota^K_Q
  =T_{\prod_{\lambda\in \Lambda}(S_\lambda,T_\lambda,\sigma_\lambda,B_\lambda)}$ and
  $\eta\circ \iota^K_P
  =S_{\prod_{\lambda\in \Lambda}(S_\lambda,T_\lambda,\sigma_\lambda,B_\lambda)}$.
  We then have that
  $$\phi\circ\eta:\mathcal{O}_{(P,Q,\psi)}(K)\longrightarrow
  \mathcal{O}_{\bigl({}_{I_{\prod\Omega}}P,
    Q_{I_{\prod\Omega}},
    \psi_{I_{\prod\Omega}}\bigr)}\bigl((J_{\prod\Omega})_{I_{\prod\Omega}}\bigr)$$
  is a ring homomorphism satisfying
  $\phi\circ\eta\circ \iota^K_R
  =\iota_R^{\omega_{\prod\Omega}}$,
  $\phi\circ \eta\circ\iota^K_Q
  =\iota_Q^{\omega_{\prod\Omega}}$
  and $\phi\circ \eta\circ\iota^K_P
  =\iota_P^{\omega_{\prod\Omega}}$. It therefore follows that $H^K_{\omega_{\prod\Omega}}=
  \ker(\phi\circ\eta)=\ker\eta$, and since it follows from
  Proposition \ref{prop:diamond} that
  $\ker\eta=\cap_{\lambda\in\Lambda} H_\lambda$, we can conclude that
  $H^K_{\omega_{\prod\Omega}}=\cap_{\lambda\in\Lambda}H_\lambda$.

  It follows from Proposition \ref{prop:coprod}, \ref{prop:super} and
  \ref{theor_2} that there for
  each $\lambda\in\Lambda$ exists a ring homomorphism
  $\psi_\lambda:\mathcal{O}_{(P,Q,\psi)}(K)/H_\lambda\longrightarrow
  \mathcal{O}_{(P,Q,\psi)}(K)/H^K_{\omega_{\coprod\Omega}}$
  such that $\psi_\lambda\circ\sigma_{H_\lambda}=\sigma_{H^K_{\omega_{\coprod\Omega}}}$,
  $\psi_\lambda\circ T_{H_\lambda}=T_{H^K_{\omega_{\coprod\Omega}}}$ and
  $\psi_\lambda\circ S_{H_\lambda}=S_{H^K_{\omega_{\coprod\Omega}}}$. It follows that
  $\psi_\lambda\circ\wp_{H_\lambda}=\wp_{H^K_{\omega_{\coprod\Omega}}}$,
  and thus that $H_\lambda\subseteq H^K_{\omega_{\coprod\Omega}}$.

  Let $H$ be a two-sided ideal of $\mathcal{O}_{(P,Q,\psi)}(K)$ containing
  $\cup_{\lambda\in\Lambda}H_\lambda$.
  Then we have for each $\lambda\in\Lambda$ that there exists a ring homomorphism
  $\psi_\lambda:\mathcal{O}_{(P,Q,\psi)}(K)/H_\lambda\longrightarrow \mathcal{O}_{(P,Q,\psi)}(K)/H$
  such that $\psi_\lambda\circ\sigma_{H_\lambda}=\sigma_H$,
  $\psi_\lambda\circ T_{H_\lambda}=T_H$ and
  $\psi_\lambda\circ S_{H_\lambda}=S_H$. It therefore follows from Proposition
  \ref{prop:coprod} and \ref{prop:super} that there exists a ring homomorphism
  $$\tau:
  \mathcal{O}_{\bigl({}_{I_{\coprod\Omega}}P,
    Q_{I_{\coprod\Omega}},
    \psi_{I_{\coprod\Omega}}\bigr)}\bigl((J_{\coprod\Omega})_{I_{\coprod\Omega}}\bigr)
  \longrightarrow \mathcal{O}_{(P,Q,\psi)}(K)/H$$
  satisfying
  $\tau\circ\iota^{\omega_{\coprod\Omega}}_R=\wp_H\circ\iota^K_R$,
  $\tau\circ\iota^{\omega_{\coprod\Omega}}_Q=\wp_H\circ\iota^K_Q$
  and $\tau\circ\iota^{\omega_{\coprod\Omega}}_P=\wp_H\circ\iota^K_P$.
  It then follows that $\tau\circ\Psi^K_{\omega_{\coprod\Omega}}=\wp_H$, and thus
  that $H^K_{\omega_{\coprod\Omega}}=\ker\Psi^K_{\omega_{\coprod\Omega}}\subseteq H$. Hence
  $H^K_{\omega_{\coprod\Omega}}$ is the smallest two-sided ideal of $\mathcal{O}_{(P,Q,\psi)}(K)$
  containing $\cup_{\lambda\in\Lambda}H_\lambda$.
\end{proof}

\begin{corol}\label{ideals_toeplitz}
Let $R$ be a ring, let $(P,Q,\psi)$ be an $R$-system satisfying
condition \textbf{(FS)}. Then
$$H\longmapsto \omega^{\{0\}}_H,\quad\omega \longmapsto H^{\{0\}}_w$$
is a bijective correspondence between
the set of all the graded two-sided ideals $H$ of
$\mathcal{T}_{(P,Q,\psi)}$ and the set of
all $T$-pairs $\omega=(I,J)$ of $(P,Q,\psi)$.
This bijection preserves inclusion, and if $(H_\lambda)_{\lambda\in\Lambda}$ is a non-empty
family of graded two-sided ideals of $\mathcal{T}_{(P,Q,\psi)}$ and
$\Omega=(\omega^{\{0\}}_{H_\lambda})_{\lambda\in\Lambda}$, then
$H^{\{0\}}_{\omega_{\prod\Omega}}=\cap_{\lambda\in\Lambda}H_\lambda$ and $H^{\{0\}}_{\omega_{\coprod\Omega}}$ is the smallest two-sided ideal of $\mathcal{T}_{(P,Q,\psi)}$
containing $\cup_{\lambda\in\Lambda}H_\lambda$.
\end{corol}

\begin{corol} \label{corol:ideal}
Let $R$ be a ring, let $(P,Q,\psi)$ be an $R$-system satisfying
condition \textbf{(FS)} and assume that there exists a uniquely maximal
faithful $\psi$-compatible two-sided ideal $K$ of $R$.
Then
$$H\longmapsto \omega^K_H,\quad\omega \longmapsto H^K_w$$
is a bijective correspondence between
the set of all the graded two-sided ideals $H$ of
$\mathcal{O}_{(P,Q,\psi)}$ and the set of
all $T$-pairs $\omega=(I,J)$ of $(P,Q,\psi)$ satisfying $K\subseteq J$.
This bijection preserves inclusion, and if $(H_\lambda)_{\lambda\in\Lambda}$ is a non-empty
family of graded two-sided ideals of $\mathcal{O}_{(P,Q,\psi)}$ and
$\Omega=(\omega^K_{H_\lambda})_{\lambda\in\Lambda}$, then
$H^K_{\omega_{\prod\Omega}}=\cap_{\lambda\in\Lambda}H_\lambda$ and $H^K_{\omega_{\coprod\Omega}}$ is the smallest two-sided ideal of $\mathcal{O}_{(P,Q,\psi)}$
containing $\cup_{\lambda\in\Lambda}H_\lambda$.
\end{corol}

\begin{exem} \label{examples: ideals}
  Let us once again return to Example
  \ref{examples:item:1}. We saw in Example
  \ref{examples_cuntz:cross} that if $R$ is a ring with
  local units, $\varphi\in\aut(R)$, $P=R_\varphi$, $Q=R_{\varphi\inv}$ and
\begin{eqnarray*}
  \psi:P\otimes_R Q & \longrightarrow R \\
  p\otimes q & \longmapsto p\varphi(q),
\end{eqnarray*}
then $(P,Q,\psi)$ is a $R$-system which satisfies condition
\textbf{(FS)}, $\ker\Delta=\{0\}$,
$\Delta^{-1}(\mathcal{F}_P(Q))=R$,
and $\mathcal{O}_{(P,Q,\psi)}=\mathcal{O}_{(P,Q,\psi)}(R)$ is the universal ring generated by
elements  $\{[r,k]: r\in R,\ k\in\Z\}$ satisfying
$[r_1,k]+[r_2,k]=[r_1+r_2,k]$ and
$[r_1,k_1][r_2,k_2]=[r_1\varphi^{k_1}(r_2),k_1+k_2]$.

It is easy to see that a two-sided ideal $I$ of $R$ is $\psi$-invariant if
and only if $\varphi(I)\subseteq I$. It is also easy to see that if
$I$ is a $\psi$-invariant ideal, then
$\ker\Delta_I=\varphi\inv(I)+I$. Thus $(I,R)$ is a $T$-pair if and
only if $I$ is a two-sided ideal of $R$ such that $\varphi(I)=I$. It
therefore follows from Corollary \ref{corol:ideal} that we have a
bijective correspondence between $\varphi$-invariant ideals of $R$
and graded two-sided ideals of $\mathcal{O}_{(P,Q,\psi)}$ which takes a
$\psi$-invariant ideal $I$ to the graded two-sided ideal
$\{[x,k]\in\mathcal{O}_{(P,Q,\psi)}: x\in I,\ k\in\Z\}$, which is
isomorphic to the crossed product $I\times_\varphi\Z$.

It is easy to see that if we by $\varphi_I$ denote the automorphism
of $R_I=R/I$ induced by $\varphi$, then ${}_IP=(R/I)_{\varphi_I}$
and $Q_I=(R/I)_{\varphi_I\inv}$. It follows from Proposition
\ref{theor_2} that the quotient of $\mathcal{O}_{(P,Q,\psi)}$ by the
ideal $\{[x,k]\in\mathcal{O}_{(P,Q,\psi)}: x\in I,\ k\in\Z\}$ is
isomorphic to
$\mathcal{O}_{({}_IP,Q_I,\psi_I)}(R_I)=\mathcal{O}_{({}_IP,Q_I,\psi_I)}$
and thus to the crossed product $(R/I)\times_{\varphi_I}\Z$.
\end{exem}

\begin{exem} \label{examples: ideals: graph ideals}
Let $E=(E^0,E^1)$ be an oriented graph and
  $F$ a commutative
  unital ring. Let $R$ be the ring and $(P,Q,\psi)$
  the $R$-system associated with $E$ in Example
  \ref{examples:graph-toeplitz} and Example
  \ref{examples_cuntz:graph_cuntz}. For an ideal $I$ of $R$, let
  $H=\{v\in E^0: \mathbf{1}_{v}\in I\}$. We then have that
  $I=\text{span}_F\{\mathbf{1}_v: v\in I\}$. We may identify $R_I$
  with $\text{span}_F\{\wp_I(\mathbf{1}_v): v\in E^0\setminus H\}$. It is
  easy to see that
  $I$ is $\psi$-invariant if and only if
  the set of vertices $H$ is
  \textit{hereditary}, i.e. whenever $e\in E^1$ with $s(e)\in
  H$ then $r(e)\in H$. In that case we have
  $$IP=\text{span}_F\{\textbf{1}_{\overline{e}}: e\in E^1,\
  r(e)\in H\}\qquad\text{and}\qquad
  QI=\text{span}_F\{\textbf{1}_{e}: e\in E^1,\ r(e)\in H\}\,,$$
  so we may, and will, identify ${}_IP$ with
  $\text{span}_F\{\wp_I(\textbf{1}_{\overline{e}}): e\in E^1,\
  r(e)\notin H\}$ and $Q_I$ with $\text{span}_F\{\wp_I(\textbf{1}_{e}):
  e\in E^1,\ r(e)\notin H\}$.
  We then have that that $$\ker
  \Delta_I=\text{span}\{\wp_I(\textbf{1}_{v}):v\in\partial H\text{ or }
  s\inv(v)=\emptyset  \}\subseteq \ker
  \Delta$$
  where $\partial
  H:=\{v\in E^0: 0<|s\inv(v)|<\infty\text{ and }r(s^{-1}(v))\subseteq
  H\}$. The set $H$ is called
  \textit{saturated} if $\partial H\subseteq H$. We define the
  set of \textit{breaking vertices} of $H$ to be $$B_H:=\{v\in
  E^0_{inf}\setminus H: 0<|s^{-1}(v)\cap r^{-1}(E^0\setminus
  H)|<\infty\}$$
  where $E^0_{inf}=\{v\in E^0: |s^{-1}(v)|=\infty\}$.
  We then have that
  $$\Delta^{-1}_I(\mathcal{F}_{{}_IP}(Q_I))=\text{span}\{\wp_I(\textbf{1}_{v}):v\in
  E^0_{reg}\setminus H\text{ or }v\in B_H\}$$
  where $E^0_{reg}:=\{v\in E^0: 0\leq|s^{-1}(v)|<\infty\}$.

  Let $J$ be an ideal of
  $R$. Then $I\cup \Delta^{-1}(\mathcal{F}_P(Q))\subseteq J$ if and
  only if we for all $v\in H$ and all $v\in E^0$ with
  $0<|s\inv(v)|<\infty$ have that $\textbf{1}_v\in J$, and we have
  that $\wp_I(J)\subseteq
  \Delta^{-1}_I(\mathcal{F}_{{}_IP}(Q_I)) \cap (\ker\Delta_I)^\perp$
  if and only if
  we for $v\in E^0\setminus H$ with $\textbf{1}_v\in J$ have that
  $v\in E^0_{reg}\cup B_H$, $v\notin \partial H$ and
  $s\inv(v)\ne\emptyset$. So if $H$ is not saturated, then there does
  not exist any ideal $J$ of $R$ such that $I\cup
  \Delta^{-1}(\mathcal{F}_P(Q))\subseteq J$ and $\wp_I(J)\subseteq
  \Delta^{-1}_I(\mathcal{F}_{{}_IP}(Q_I)) \cap (\ker\Delta_I)^\perp$;
  and if $H$ is saturated, then there is a bijective correspondence
  between ideals $J$ of $R$ such that $I\cup
  \Delta^{-1}(\mathcal{F}_P(Q))\subseteq J$ and $\wp_I(J)\subseteq
  \Delta^{-1}_I(\mathcal{F}_{{}_IP}(Q_I)) \cap (\ker\Delta_I)^\perp$,
  and subsets of $B_H$. This correspondence takes a subset $S$ of
  $B_H$ to the ideal $\text{span}_F\{\mathbf{1}_v:v\in H\cup S\text{
    or }0<|s\inv(v)|<\infty\}$.

  So it follows from Corollary \ref{corol:ideal} that there is a
  bijective correspondence between pairs $(H,S)$ where $H$ is a
  hereditary and saturated subset of $E^0$ and $S$ is a subset of
  $B_H$, and graded ideals of $\mathcal{O}_{(P,Q,\psi)}$. This
  correspondence takes a graded ideal $K$ to $(H,S)$ where
  $$H=\{v\in
  E^0: p_v\in K\}$$ and
  $$S=\{v\in B_H:
  p_v-\sum_{e\in s\inv(v)\cap r\inv(E^0\setminus
    H)}x_ey_e\in
  K\}\,.$$
  It takes a pair $(H,S)$
  to the graded ideal generated by $$\{p_v: v\in
  H\}\cup\{p_v-
  \sum_{e\in s\inv(v),\ r(e)\notin H}
  x_ey_e:
  v\in S\}\,.$$

Thus we recover the result of \cite[Theorem 5.7(1)]{TF}.
\end{exem}

\end{document}